\documentclass[11pt]{article}

\usepackage{amssymb,amsmath,epsfig,graphics,euscript}

\newcommand{\beq}{\begin{equation}}
\newcommand{\eeq}{\end{equation}}
\newcommand{\bea}{\begin{eqnarray}}
\newcommand{\eea}{\end{eqnarray}}
\newcommand{\beas}{\begin{eqnarray*}}
\newcommand{\eeas}{\end{eqnarray*}}

\newtheorem{theorem}{Theorem}[section]
\newtheorem{definition}[theorem]{Definition}
\newtheorem{proposition}[theorem]{Proposition}
\newtheorem{corollary}[theorem]{Corollary}
\newtheorem{lemma}[theorem]{Lemma}
\newtheorem{remark}[theorem]{Remark}
\newtheorem{example}[theorem]{Example}
\newtheorem{examples}[theorem]{Examples}
\newtheorem{foo}[theorem]{Remarks}

\newenvironment{proof}{\addvspace{\medskipamount}\par\noindent{\it
Proof}.}
{\unskip\nobreak\hfill$\Box$\par\addvspace{\medskipamount}}

\newcommand{\R}[1]{\mathbb{R}}     % absolute value
\newcommand{\bG}{\mathbb G}

\newcommand{\p}{\partial}
\newcommand{\ee}{\ell}
\newcommand{\bM}{\mathbb M}
\newcommand{\ri}{\text{Ric}}
\newcommand{\Rn}{\mathbb R^n}
\newcommand{\Om}{\Omega}
\newcommand{\F}{\mathcal F}
\newcommand{\Ho}{\mathcal H}

\newcommand{\di}{\mathfrak h}

\title{Generalized Bochner formulas and Ricci lower bounds for sub-Riemannian manifolds of rank two}

\author{Fabrice Baudoin, Nicola Garofalo \footnote{Second author supported in part by NSF Grant DMS-0701001}
\\
{\small Department of Mathematics} \\
{\small Purdue University} \\
}

\begin{document}

\maketitle

\begin{abstract}
We study a new class of rank two sub-Riemannian manifolds
encompassing Riemannian manifolds, CR manifolds with vanishing
Webster-Tanaka torsion, orthonormal bundles over Riemannian
manifolds, and graded nilpotent Lie groups of step two. These
manifolds admit a canonical horizontal connection and a canonical
sub-Laplacian. We construct on these manifolds an analogue of the
Riemannian  Ricci tensor and prove Bochner type formulas for the
sub-Laplacian. As a consequence, it is possible to formulate on
these spaces a sub-Riemannian analogue of the so-called curvature
dimension inequality. Sub-Riemannian manifolds for which this
inequality is satisfied are shown to share many properties in common
with Riemannian manifolds whose Ricci curvature is bounded from
below.
\end{abstract}

\tableofcontents

\section{Introduction}

A sub-Riemannian manifold is a smooth Riemannian manifold $\bM$
equipped with a fiber inner product $g_R(\cdot,\cdot)$ on the tangent
bundle $T\bM$ and a non-holonomic, or bracket generating, subbundle
$\mathcal H \subset T\bM$. This means that if we denote by
$L(\mathcal H)$ the Lie algebra of the vector fields generated by
the global $C^\infty$ sections of $\mathcal H$, then $\text{span}
\{X(x)\mid X\in L(\mathcal H)\} = T_x(\bM)$ for every $x\in \bM$. A
piecewise smooth curve $\gamma:[a,b]\to \bM$ is called admissible,
or horizontal, if it is tangent to $\mathcal H$, i.e. if $\gamma'(t)
\in \mathcal H_{\gamma(t)}$, whenever $\gamma'(t)$ is defined. The
horizontal length of $\gamma$ is defined as in Riemannian geometry
\[ \ee_\mathcal H(\gamma) = \int_a^b g_R(\gamma'(t),\gamma'(t)) dt. \]
Denoting by $\mathcal H(x,y)$ the collection of all horizontal
curves joining $x, y\in \bM$, one defines a distance $d(x,y)$
between $x$ and $y$ by minimizing on the length of all $\gamma\in
\mathcal H(x,y)$, i.e.
\[
d(x,y) = \underset{\gamma \in \mathcal H(x,y)}{\inf}\ \ee_\mathcal
H(\gamma). \] Such distance was introduced by Carath\'eodory in his
seminal paper \cite{Car} on formalization of the classical
thermodynamics. In such framework horizontal curves correspond
loosely speaking to adiabatic processes.

To be precise, in \cite{Car} the question of whether $d(x,y)$ be a
true distance was left open. This question was answered by the
fundamental connectivity theorem of Chow \cite{Ch} and Rashevsky
\cite{Ra} which states that if $\bM$ is connected and $\mathcal H$
is bracket generating, then $\mathcal H(x,y)\not= \varnothing$ for
every $x,y\in \bM$. As a consequence, $d(x,y)$ is finite and
therefore it is a true distance. Such metric is nowadays known as
the control, or Carnot-Carath\'eodory distance on $\bM$ (after
Gromov, Lafontaine and Pansu \cite{GLP}, see also \cite{Gromov},
\cite{Gromov2}). Besides the cited references the reader should
consult E. Cartan's pioneering address \cite {Ca} at the Bologna
International Congress of Mathematicians in 1928, as well as the
articles by Rayner \cite{Rayner} (where sub-Riemannian manifolds are
called parabolic spaces), by Mitchell \cite{mitchell} and Strichartz
\cite{Strichartz}, see also \cite{stricorr}. One should also consult
the monographs \cite{Be}, \cite{Montgomery}, \cite{A} and
\cite{baudoin}.

\

We note that when $\mathcal H = T\bM$, then the distance $d(x,y)$ is
simply the Riemannian distance associated with the inner product
$g_R (\cdot,\cdot)$, and thus sub-Riemannian manifolds encompass
Riemannian ones. However, some aspects of the geometry of
sub-Riemannian manifolds are considerably less regular than their
Riemannian ancestors. Some of the major differences between the two
geometries are the following:
\begin{enumerate}
\item The Hausdorff dimension of the metric space $(\bM,d)$ is usually greater than the manifold
dimension;
\item The exponential map defined by the geodesics of the metric space   $(\bM,d)$ is in general not a local diffeomorphism in a
neighborhood of the point at which it is based (see \cite{Rayner});
\item The space of horizontal paths joining
two fixed points may have singularities (the so-called abnormal geodesics, see \cite{Montgomery}).
\end{enumerate}

\

During the last two decades there have been several advances in the
study of sub-Riemannian spaces and the closely connected theory of
sub-elliptic pde's, see \cite{steinNice}, \cite{Fo}, \cite{RS},
\cite{FP1}, \cite{NSW}, \cite{SC}, \cite{JSC}, \cite{Jer},
\cite{FSC}, \cite{KS1}, \cite{KS2}, \cite{VSC}. However, these
developments are of a local nature. As a consequence, the theory
presently lacks a body of results which, similarly to the Riemannian
case, connect properties of solutions of the relevant pde's to the
geometry of the ambient manifold.

\

The purpose of this paper is to begin a program aimed at filling
this gap. Precisely, in a sub-Riemannian manifold of rank two we
introduce a new notion of Ricci curvature tensor, and with such
notion we obtain various results which parallel those of Riemannian
manifolds with Ricci tensor bounded from below.

\

To put our work in the proper perspective we recall that in
Riemannian geometry a first point of view on the Ricci tensor is to
understand it as a measure of volume distortion in geodesic normal
coordinates. This point of view has recently led several authors (
Lott-Villani \cite{villani-lott}, Sturm \cite{sturm1},
\cite{sturm2}, Ollivier \cite{ollivier}) to introduce on metric
spaces more general than the Riemannian ones a suitable notion of
Ricci curvature based on the theory of optimal transport (see
\cite{villani}). However, as pointed out in \cite{juillet}, this
notion of Ricci curvature bound may not be applied in a
sub-Riemannian framework. In sub-Riemannian geometry, the
quantification of volume distortion properties is particularly
difficult to handle because of the singular nature of the
exponential map. To the authors' best knowledge, the only
significant results in this direction are at the moment known only
in the special case of three-dimensional  contact manifolds (see
\cite{rumin}, \cite{juillet} and \cite{agrachev}).

\

There is a second point of view on the Riemannian Ricci tensor and
this is the one that we have adopted in the present paper. On a
Riemannian manifold, there is a canonical second order differential
operator: the Laplace-Beltrami operator $\Delta$. Properties of this
operator and of the associated heat flow are intimately related to
the geometry and the topology of the underlying manifold. One of the
cornerstones of the interplay between the analysis of the
Laplace-Beltrami operator and geometry is given by the celebrated
Bochner formula:
 \[
\Delta(|\nabla f|^2) = 2 ||\nabla^2 f||^2 + 2 <\nabla
f,\nabla(\Delta f)> + 2\ \text{Ric}(\nabla f,\nabla f), \] This is
where the Ricci curvature tensor appears in the study of the
Laplace-Beltrami operator, and it is then seen that a lower bound
assumption on the Ricci curvature is equivalent to a
\textit{coercivity} property of a canonical bilinear differential
form associated to $\Delta$.  More precisely, associated with
$\Delta$ are the two following differential bilinear forms on smooth
functions $f, g :\mathbb{M} \rightarrow \mathbb{R}$,
\begin{equation*}
\Gamma(f,g) =\frac{1}{2}(\Delta(fg)-f\Delta g-g\Delta f)=( \nabla f
, \nabla g ),
\end{equation*}
and
\begin{equation*}
\Gamma_{2}(f,g) = \frac{1}{2}\big[\Delta\Gamma(f,g) - \Gamma(f,
\Delta g)-\Gamma (g,\Delta f)\big].
\end{equation*}
As an application of the Bochner's formula, which we can re-write as
\[ \Delta \Gamma(f,f) = 2 ||\nabla^2 f||^2 + 2 \Gamma(f,\Delta f) +
2\ \text{Ric}(\nabla f,\nabla f), \] one obtains
\[
\Gamma_{2}(f,f)= \| \nabla^2 f \|_2^2 +\text{Ric}(\nabla f, \nabla
f).
\]
With the aid of Schwarz inequality, which gives $\| \nabla^2 f
\|_2^2\ge \frac{1}{d} (\Delta f)^2$, the assumption that the
Riemannian Ricci tensor on $\bM$ is bounded from below by $\rho \in
\mathbb{R}$,  translates then into the so-called curvature-dimension
inequality:
\begin{equation}\label{curvature_dimension}
\Gamma_{2}(f,f) \ge \frac{1}{d} (\Delta f)^2 + \rho
\Gamma(f,f).
\end{equation}
The inequality  \eqref{curvature_dimension} perfectly  captures the
notion of Ricci curvature lower bound, see example \ref{CDriemann}
below. In the hands of D. Bakry, M. Ledoux and their co-authors it
has proven to be a powerful tool in recovering most of the
well-known theorems which, in Riemannian geometry, are obtained
under the assumption that the Ricci curvature be bounded from below
(see for instance \cite{bakry-stflour}, \cite{ledoux-zurich},
\cite{li}).

\

The purpose of this work is to generalize this point of view in a
sub-Riemannian setting. More precisely, we assume that on $\bM$ we
are given smooth vector fields $X_1,...,X_d$ which generate the
horizontal subbundle $\mathcal H$. The commutators $[X_i,X_j]$ are
supposed to satisfy some structural assumptions, see \eqref{bra1}
and \eqref{bra2} below, with $Z_{mn}, m,n=1,...,\di$, being the
non-horizontal, or vertical directions. Setting \[ \mathcal H(x)  =
\text{span} \{X_1(x),...,X_d(x)\},\ \ \ \mathcal V(x) = \text{span}
\{Z_{mn}(x)\mid 1\le m<n\le \di\}, \] we assume that \[ T_x \bM =
\mathcal H(x) \oplus \mathcal V(x),\ \ \ x\in \bM. \]

We also assume that $\mathcal H$ is bracket-generating of rank two,
i.e. \[ T_x \bM = \text{span} \{X_i(x), [X_j,X_k](x)\},\ \ \ \ x\in
\bM. \]

The first step will be to equip $\bM$ with a canonical subelliptic
operator
\[
L=X_0+\sum_{i=1}^d X_i^2,
\]
which shall play the role of the Laplace-Beltrami operator. We then
equip $\bM$ with a degenerate metric tensor $g(\cdot,\cdot)$ for
which $\{X_1(x),...,X_d(x)\}$ is orthonormal at every $x\in \bM$,
the spaces $\mathcal H(x)$ and $\mathcal V(x)$ are orthogonal, and
such that $g|_\mathcal V = 0$. In such a manifold we introduce a
canonical connection $\nabla$ generalizing the Levi-Civita
connection. One of the fundamental assumptions in this work is that
the torsion be vertical. Besides the Riemannian case, basic examples
covered by our framework are CR manifolds with vanishing
Webster-Tanaka torsion (Sasakian manifolds), orthonormal bundles
over Riemannian manifolds, and graded nilpotent Lie groups of step
two.

\

An essential tool in our analysis are two sub-Riemannian Bochner
formulas for $L$, one in the horizontal direction and one in the
vertical one. Such formulas are established in section
\ref{S:bochner}. We mention here that their technical complexity is
the main \emph{raison d'\^{e}tre} of the above rank two assumption
(in this regard, one should see the closing comments at the end of
this introduction). By means of these formulas we identify a tensor
$\mathcal R$ which plays the role of the Riemannian Ricci tensor. We
show in Proposition \ref{P:miniblackhole}  that for every $f\in
C^\infty(\bM)$ one has
\begin{align*}
\mathcal{R} (f,f)= &\sum_{\ee,k=1}^d \text{Ric}(X_\ee,X_k) X_\ee f
X_k f -( (\nabla_{X_\ee} T) (X_\ee,X_k)f )(X_k f)  +\frac{1}{4}
\left( T(X_\ee,X_k)f\right)^2.
\end{align*}
where $T$ is the torsion of the canonical connection $\nabla$ and $
\text{Ric}$ its Ricci curvature. In fact, in the Riemannian case,
one has \[ \mathcal R(f,f) = \text{Ric}(\nabla f,\nabla f), \ \ \ \
\ f\in C^\infty(\bM). \] A lower bound assumption on the tensor
$\mathcal R$ will translate into a generalized curvature dimension
inequality for $L$ that writes in the form:
\begin{equation}\label{CDintro}
\Gamma_{2}(f,f)+\nu \Gamma^Z_{2}(f,f) \ge \frac{1}{d} (Lf)^2 + \left( \rho_1 -\frac{\kappa}{\nu}\right)  \Gamma (f,f) + \rho_2 \Gamma^Z (f,f), \quad \nu >0,
\end{equation}
where $\Gamma^Z$ and $\Gamma^Z_{2}$ are two bilinear differential
forms on the vertical subbundle. The parameters of this inequality
are $\rho_1,\rho_2,\kappa$, and a recurrent theme of our work is
that all estimates solely depend in an explicit quantitative way on
these parameters. As shown in the following summary of the main
results, the parameter $\rho_1$ is of particular importance:
\begin{enumerate}
\item \textbf{Dodziuk-Yau type theorem:} If the inequality (\ref{CDintro}) holds
for some constants $\rho_1 \in \mathbb{R},\rho_2 >0,\kappa>0$, then
the heat semigroup $P_t=e^{tL}$ is stochastically complete and
bounded solutions of the heat equation are characterized by their
initial condition;
\item \textbf{ Li-Yau type inequality:} If the inequality
(\ref{CDintro})  holds for some constants $\rho_1 \in
\mathbb{R},\rho_2 >0,\kappa>0$, then the heat semigroup $P_t$
satisfies a Li-Yau type gradient estimate. Exploiting the latter, we
prove that the heat kernel $p(x,y,t)$ of $P_t$ satisfies a uniform
Harnack inequality and a Gaussian upper bound estimate;
\item \textbf{Yau-Liouville type theorem:} If the inequality (\ref{CDintro}) holds
for some constants $\rho_1 \ge 0,\rho_2 >0,\kappa>0$, then there is
no non constant bounded harmonic function;
\item \textbf{Volume and isoperimetry:} If the inequality (\ref{CDintro})
holds for some constants $\rho_1 \ge 0,\rho_2 >0,\kappa>0$, then the
volume growth of the geodesic balls is at most polynomial. Moreover,
under the assumption that the volume of balls is at least
polynomial, we obtain a global isoperimetric inequality;
\item  \textbf{Myers type theorem:} If the inequality (\ref{CDintro})
holds for some constants $\rho_1 > 0,\rho_2 >0,\kappa>0$, then the
metric space $(\mathbb{M},d)$ is compact in the metric topology with
a Hausdorff dimension less than $d\left( 1+\frac{3 \kappa}{2\rho_2}
\right)$ and we have \[ \emph{diam}\ \bM \le 2\sqrt{3} \pi \sqrt{
\frac{\kappa+\rho_2}{\rho_1\rho_2} \left(
1+\frac{3\kappa}{2\rho_2}\right)d };
\]
\item \textbf{$L^1$ Poincar\'e inequality:} If the inequality (\ref{CDintro})
holds for some constants $\rho_1 > 0,\rho_2 >0,\kappa>0$, then the
following inequality is satisfied
\[
\inf_{c \in \mathbb{R}} \int_\bM | f-c | d \mu \le  6 d \left( 1+\frac{3\kappa}{2\rho_2} \right)
\sqrt{\frac{\kappa+\rho_2}{d\rho_1 \rho_2}} \int_\bM
\sqrt{\Gamma(f)} d\mu.
\]
\item  \textbf{Lichnerowicz type estimate:} If the inequality (\ref{CDintro}) holds
for some constants $\rho_1 > 0,\rho_2 >0,\kappa>0$, then the first
non zero eigenvalue $\lambda_1$ of $-L$ satisfies the estimate
\[
\lambda_1 \ge \frac{\rho_1 \rho_2}{\frac{d-1}{d} \rho_2 +\kappa}.
\]
\end{enumerate}

The basic source of the above listed results is the study of
monotone entropy type functionals of the heat semigroup. More precisely, generalizing \cite{bakry-baudoin},  \cite{Bakry-Ledoux} and \cite{baudoin-bonnefont},  one of the most important observations is that in our framework, the inequality (\ref{CDintro}) implies that for the functionals
\[
\Phi_1 (t)=P_t \left( (P_{T-t} f) \Gamma (\ln P_{T-t}f) \right),
\]
\[
\Phi_2 (t)=P_t \left( (P_{T-t} f) \Gamma^Z (\ln P_{T-t}f) \right).
\]
we have the following differential inequality
\[
\left(- \frac{b'}{2\rho_2} \Phi_1 +b \Phi_2 \right)' \ge  -\frac{2b' \gamma}{d\rho_2} LP_Tf + \frac{b'  \gamma^2}{d\rho_2}  P_T f,
\]
where $b$ is any smooth, positive and decreasing function on the time
interval $[0,T]$ and
\[
\gamma=\frac{d}{4} \left( \frac{b''}{b'} +\frac{\kappa}{\rho_2} \frac{b'}{b} +2\rho_1 \right).
\]
Depending on the value of $\rho_1$, a careful choice of the function $b$ leads then to a generalized Li-Yau type inequality from which it is possible to deduce the above results.

\

Some final comments are in order. We have mentioned above that the
choice of working with sub-Riemannian manifolds of rank two is
closely connected with the complexity of the Bochner type formulas
in section \ref{S:bochner}. On one hand, the rank two setting is
rich enough to encompass at one time the case of Riemannian and CR
(Sasakian) manifolds. On the other hand, following the program in
section \ref{S:bochner}, for manifolds of arbitrary rank $\mathfrak
r$ one would need to establish $\mathfrak r$ Bochner type formulas
(of decreasing complexity, the most difficult one being the
horizontal one). Whereas it would be desirable to treat
sub-Riemannian manifolds of arbitrary rank, we have felt that the
increased technical difficulties connected with this endeavor would
distract from the main ideas, and have consequently decided to defer
the treatment of manifolds of rank $\mathfrak \ge 3$ to a future
study. We should also mention that most of the constants appearing
in the main results in this paper are not optimal. This can be seen
by considering the special setting of graded nilpotent Lie groups
where one can use the underlying non-isotropic dilations to obtain
sharper constants. Finally, we mention that our methods strongly
rely on the assumption that the torsion of the canonical connection
is vertical.

\

In closing we mention that a pseudo-hermitian version of the
Bonnet-Myers theorem for contact manifolds of dimension three was
proved by Rumin, see Theorem 16 in \cite{rumin}. We thank S. Webster
for bringing this to our attention.

\section{The framework and its geometrical interpretation
}\label{S:framework}

\subsection{Prelimimaries}\label{SS:framework}

Henceforth in this paper, $\mathbb{M}$ will be a smooth connected
Riemannian manifold. We assume that $X_1,...,X_d$ are given smooth
vector fields on $\mathbb M$ satisfying the following commutation
relations:
\begin{equation}\label{bra1}
[X_i,X_j]=\sum_{\ee=1}^d \omega_{ij}^\ee X_\ee +\sum_{m,n=1}^{\di}
\gamma_{ij}^{mn} Z_{mn},
\end{equation}
\begin{equation}\label{bra2} [X_i,Z_{mn}]=\sum_{\ee=1}^d \delta_{imn}^\ee X_\ee,
\end{equation}
for some smooth vector fields $\{Z_{mn}\}_{1\le m,n \le \di}$  and
smooth functions $ \omega_{ij}^\ee $, $ \gamma_{ij}^{mn}$ and
$\delta_{imn}^\ee $. By convention $Z_{mn}=-Z_{nm}$,
$\omega_{ij}^\ee =-\omega_{ji}^\ee$ and
$\gamma_{ij}^{mn}=-\gamma_{ji}^{mn}$ .  We will assume that
\begin{equation}\label{deltas}
\delta_{imn}^\ee=-\delta_{\ee mn}^i,\ \ i, \ee = 1,...,d,\
\text{and}\ m, n=1,...,\di.
\end{equation}
Note that \eqref{deltas} implies $\delta^i_{imn} = 0$ for
$i=1,..,d$, and $m,n=1,...,\di$. The assumption \eqref{deltas} plays
a pervasive role in the results of this paper.

\

We shall moreover assume that the vector fields $X_i$'s satisfy
H\"ormander's finite rank condition of step two \cite{Ho}: for every
$x \in \mathbb{M}$,
\[
\text{span}\left\{ X_i (x),[X_j,X_k](x), 1 \le i \le d, 1 \le j <k
\le d \right\} = T_x\mathbb{M}.
\]
We denote by
\[
\mathcal{H}(x)=\text{span} \left\{X_1 (x),...,X_d (x) \right\},
\quad x\in \mathbb{M},
\]
the set of \textit{horizontal} directions at $x$, and by
\[
\mathcal{V}(x)=\text{span}\left\{  Z_{mn} (x), 1\le m<n \le \di
\right\},
\]
that of \emph{vertical} directions. We shall assume that
\[
\text{dim} \mathcal{H}(x) =d, \quad \text{dim}
\mathcal{V}(x)=\frac{\di(\di-1)}{2},
\]
at each point $x\in \bM$, and that
\[
\mathcal{H}(x) \oplus \mathcal{V}(x)= T_x \mathbb{M}.
\]
The dimension of $\bM$ is therefore $d+\frac{\di(\di-1)}{2}$, but
such number will never explicitly appear in the results in this
paper. We indicate with $\mathcal H = \bigcup_{x\in \bM} \mathcal
H(x)$, and $\mathcal V = \bigcup_{x\in \bM} \mathcal V(x)$,
respectively the horizontal and vertical subbundles of $T\bM$.

\

Our goal is to study the second order subelliptic  operator $L$ given by
\begin{equation}\label{L}
L= \sum_{i=1}^d X_i^2 + X_0,
\end{equation}
where
\begin{equation}\label{X0}
X_0=-\sum_{i,k=1}^d \omega_{ik}^k X_i.
\end{equation}
We will assume that with respect to the Riemannian measure $\mu$ of $\bM$,
\begin{equation}\label{Zstar}
Z_{mn}^*=-Z_{mn},\ \ \ \ \ \ \ \ L^* = L,
\end{equation}
where $Z_{mn}^*$ denotes the formal adjoint of $Z_{mn}$ and $L^*$ the formal adjoint of $L$. We note
explicitly that \eqref{Zstar} means that, if we set for every $\phi,
\psi\in C^\infty_0(\bM)$
\[
<\phi,\psi> = \int_{\bM} \phi \psi d\mu,
\]
then one has
\begin{equation}\label{symm}
<Z_{mn} \phi,\psi> = - <\phi,Z_{mn} \psi>\ ,\ \ \ \ \ \ <L\phi,\psi>
= <\phi,L\psi>.
\end{equation}

We recall that, thanks to H\"ormander's hypoellipticity theorem
\cite{Ho}, distributional solutions to $Lf =0$ in $\mathbb M$ are
$C^\infty$ functions. We stress that, thanks to the second identity
in the assumption \eqref{Zstar}, the operator $L$ can be realized as
\begin{equation}\label{subL}
L = - \sum_{i=1}^d X^*_i X_i.
\end{equation}

We also assume that $\bM$ is endowed with a Levi-Civita connection
with respect to which the Laplace-Beltrami operator is given by
\begin{equation}\label{LB}
\Delta = \sum_{i=1}^d X_i^2 + \sum_{1\le m<n\le \di} Z_{mn}^2 + X_0.
\end{equation}

Let us now give some examples that, at least locally, fit into the
previous framework and that constitute a basic motivation for our
study.

\begin{example}\label{riemannian}[Laplace-Beltrami operator on a Riemannian
manifold]\label{E:riemannian} Let $(\mathbb{M},g)$ be a
$d$-dimensional connected Riemannian manifold with Levi-Civita
connection $\nabla$. Let $X_1,...,X_d$ be a local orthonormal frame
around a point $x_0 \in \mathbb{M}$. In that case, we have
\[
[X_i,X_j]= \nabla_{X_i} X_j - \nabla_{X_j} X_i = \sum_{k=1}^d \left(
\Gamma_{ij}^k -\Gamma_{ji}^k \right) X_k
\]
where $\Gamma_{ij}^k$ are the Christoffel symbols of the Levi-Civita
connection. Thus, in this particular case we have $\gamma^{mn}_{ij}
= 0$ in \eqref{bra1}. The Laplace-Beltrami operator on $\mathbb{M}$
reads
\[
\Delta= \sum_{i=1}^d X_i^2 + X_0,
\]
where
\[
X_0=-\sum_{i,k=1}^d \left(\Gamma_{ik}^k -\Gamma_{ki}^k \right) X_i.
\]
Finally, we can observe that $L$ is symmetric with respect to the Riemannian measure on $\mathbb{M}$.
\end{example}
\begin{example}\label{nilpotent} [Graded nilpotent Lie groups of step two]
Let $\mathbb{G}$ be a connected and simply connected nilpotent Lie
group of step two. This means that its Lie algebra can be written as
$\mathfrak g = V_1 \oplus V_2$, where $[V_1,V_1] = V_2$, and
$[V_1,V_2]=\{0\}$. Let $L_x(y) = x y$ be the operator of
left-translation on $\bG$, and indicate with $dL_x$ its
differential. If $e_1,...,e_d$ is an orthonormal basis of $V_1$, we
indicate with $X_1,...,X_d$, where $X_j(x) = dL_x(e_j)$, the
corresponding system of left-invariant vector fields on $\bG$. We
assume that $\bG$ is endowed with a left-invariant Riemannian inner
product with respect to which $\{X_1,...,X_d\}$ constitutes a global
orthonormal frame. In this framework, we see that \eqref{bra1} holds
with
\[
\omega_{ij}^\ee = 0,
\]
\[
Z_{mn}=[X_m,X_n],
\]
\[
\begin{cases}
\gamma_{mn}^{mn}=-\gamma_{nm}^{mn}=\frac{1}{2}, \quad m \neq n,
\\
\gamma_{ij}^{mn}=0, \quad \text{otherwise},
\end{cases}
\]
\[
\delta^\ee_{imn}=0.
\]
In view of \eqref{subL} we see that $L = \sum_{i=1}^d X_i^2$ is the
sub-Laplacian associated with $X_1,...,X_d$, see \cite{Fo},
\cite{Strichartz}. In this case $L$ is symmetric with respect to the
bi-invariant Haar measure on $\bM$.

\end{example}

\begin{example}\label{horizontalbochner}[Horizontal Bochner Laplace operator] Let $(\mathbb{M},g)$ be a $d$-dimensional
connected  smooth Riemannian manifold endowed with the Levi-Civita
connection. Let us consider the orthonormal frame bundle $\
\mathcal{O} \left( \mathbb{M}\right) $ over $\mathbb{M}$. For each
$x \in \mathbb{R}^d$ we can define  a horizontal vector field $H_x$
on $\mathcal{O}\left( \mathbb{M}\right)$ by the property that at
each point $u \in \mathcal{O} ( \mathbb{M} )$, $H_{x} (u)$ is the
horizontal lift of $u (x)$ from $u$. If  $(e_1,...,e_d)$ is the
canonical basis of $\mathbb{R}^d$,  the fundamental horizontal
vector fields are then defined by
\begin{equation*}
H_i= H_{e_i}.
\end{equation*}
 Now, for every $M \in \mathfrak{o}_d
(\mathbb{R})$ (space of $d
\times d$ skew symmetric matrices), we can define a vertical vector field $V_M$ on
$\mathcal{O} \left( \mathbb{M}\right)$ by
\[
(V_M F)(u) = \lim_{t \rightarrow 0} \frac{F \left( u e^{tM}
\right) -F(u)}{t},
\]
where $u \in \mathcal{O} \left( \mathbb{M}\right)$ and $F:
\mathcal{O} \left( \mathbb{M}\right) \rightarrow \mathbb{R}$. If
$E_{ij}$, $1 \le i <j \le d$ denote the canonical basis of
$\mathfrak{o}_d (\mathbb{R})$ ($E_{ij}$ is the matrix whose
$(i,j)$-th entry is $1/2$, $(j,i)$-th entry is $-1/2$ and all other
entries are zero), then the fundamental vertical vector fields are
given by
\[
V_{ij}= V_{E_{ij}}.
\]
It can be shown that we have the following Lie bracket relations:
\[
[H_i,H_j]=-2\sum_{k<l} \Omega_{ij}^{kl} V_{kl},
\]
\[
[H_i,V_{jk}]=-\delta_{ij} \frac{1}{2} H_k +\delta_{ik} \frac{1}{2} H_j,
\]
where $\delta_{ij}=1$ if $i=j$ and $0$ otherwise,  and where $\Omega$ is the Riemannian curvature form:
\[
\Omega (X,Y)(u)=u^{-1} R(\pi_* X , \pi_* Y)u,~~X,Y \in
\mathrm{T}_u \mathcal{O} \left( \mathbb{M}\right),
\]
$R$ denoting the Riemannian curvature tensor on $\mathbb{M}$ and $\pi$ the canonical projection $\mathcal{O}\left( \mathbb{M}\right) \rightarrow \mathbb{M}$.
In this setting, the Bochner's horizontal Laplace operator  is
by definition the  operator on $\mathcal{O}\left( \mathbb{M}\right)$ given by
\begin{equation*}
\Delta _{\mathcal{O}\left( \mathbb{M}\right)
}=\sum_{i=1}^{d}H_{i}^{2}.
\end{equation*}
Its fundamental property is
that it is the lift of the Laplace-Beltrami operator $\Delta_{
\mathbb{M}}$ of $\mathbb{M}$. That is, for every smooth $f:
\mathbb{M} \rightarrow \mathbb{R}$,
\[
\Delta _{\mathcal{O}\left( \mathbb{M}\right)} (f \circ
\pi)=(\Delta _{ \mathbb{M}} f)\circ \pi.
\]

For the reader unfamiliar with the above construction, we refer for
instance to Chapter 3 in \cite{baudoin} for further details. It is
also proved in the last reference that if the curvature form
$\Omega$ is everywhere non degenerate, then $\Delta
_{\mathcal{O}\left( \mathbb{M}\right) }$ is subelliptic. Under this
last assumption, it is then readily checked that the study of
$\Delta _{\mathcal{O}\left( \mathbb{M}\right)}$ falls into our
framework.
\end{example}

\begin{example}\label{CR}[The subelliptic Laplace operator on CR manifolds]
Let $\mathbb{M}$ be a non degenerate CR manifold of real
hypersurface type and dimension $d +1$, where $d= 2n$. Let $\theta$
be a contact form on $\mathbb{M}$ with respect to which the Levi
form $L_\theta$ is positive definite. Let us assume that the
pseudo-Hermitian torsion of the Tanaka-Webster connection of
$(\mathbb{M},\theta)$ is zero. We denote by $T$ the characteristic
direction of $\theta$ and consider a local orthonormal frame
$T_1,...,T_n$, that is $ L_\theta ( T_i ,\bar{T_j} )=\epsilon_{ij}$
where $\epsilon_{ij}$ is the Kronecker symbol. The following
commutations properties hold:
\[
[T_i, T_j]=\sum_{k=1}^n \left( \Gamma_{ij}^k -\Gamma_{ji}^k \right)
T_k,
\]
\[
[T_i, \bar{T_j}]=-2 \sqrt{-1} \epsilon_{ij} T+ \sum_{k=1}^n \Gamma_{i\bar{j}}^{\bar{k}} T_{\bar{k}} - \Gamma_{\bar{j} i}^k T_k,
\]
\[
[T,T_i]=\sum_{k=1}^n \ \Gamma_{0i}^k  T_k,
\]
where the  $\Gamma_{ij}^k$ are the Christoffel symbols of the Tanaka-Webster connection (see \cite{CR} pp. 32). The sub-Laplacian (locally) reads:
\[
\Delta=-\sum_{k=1}^n T_i^* T_i +\bar{T_i}^* \bar{T_i},
\]
where $T_i^*$ is the adjoint of $T_i$ with respect to the volume form $\theta \wedge (d \theta)^n$. If we denote
\[
X_i=\frac{1}{\sqrt{2}} (T_i +\bar{T_i}), X_{i+n}= \frac{\sqrt{-1}}{\sqrt{2}} (T_i -\bar{T_i}), \quad 1 \le i \le n,
\]
then we fall into the previous framework.
\end{example}

\subsection{Canonical connection}\label{SS:cc}

On $\mathbb{M}$ there is a canonical connection associated with the
geometry of the differential system generated by the vector fields
$X_1,...,X_d$. We consider the degenerate metric tensor $g$ on
$\mathbb{M}$, such that $\{X_1 (x),...,X_d (x)\}$ is orthonormal at
each point $x\in \bM$, and for which the spaces $\mathcal{H}(x)$ and
$\mathcal{V}(x)$ are orthogonal, and $g_{/ \mathcal{V}(x)}=0$.
%For the sake of
%simplicity, in what follows we will denote by $<\cdot,\cdot>$,
%instead of $g(\cdot,\cdot)$, the metric tensor.

\begin{proposition}\label{P:horconn}
On $\mathbb{M}$, there is a unique affine connection $\nabla$ that satisfies the following properties:
\begin{itemize}
\item $\nabla g =0$;
\item $\nabla_{X_i} X_j$ is horizontal for $1 \le i,j \le d$;
\item $\nabla Z_{mn}=0$, $1\le m,n \le d$;
\item If $X,Y$ are horizontal vector fields,
the torsion field $T(X,Y)=\nabla_{X} Y -\nabla_{Y} X-[X,Y]$ is
vertical and $T(X_i,Z_{mn})=0$, $1\le i \le d, 1\le m,n \le \di$.
\end{itemize}
This connection is characterized by the formulas:
\begin{equation}\label{christoffel1}
\nabla_{X_i} X_j= \sum_{k=1}^d \Gamma^k_{ij} X_k = \sum_{k=1}^d
\frac{1}{2} \left(\omega_{ij}^k +\omega_{ki}^j - \omega_{jk}^i
\right) X_k,
\end{equation}
\begin{equation}\label{christoffel2}
\nabla_{Z_{mn}} X_i=-\sum_{\ee=1}^d \delta_{imn}^\ee X_\ee,
\end{equation}
\begin{equation}\label{christoffel3}
\nabla Z_{mn}=0,
\end{equation}
where we have denoted by $\Gamma^k_{ij}$ the Christoffel symbols of
the connection.
\end{proposition}

\begin{proof}
First, it is easy to check that the connection given by the formulas
\eqref{christoffel1}, \eqref{christoffel2}, \eqref{christoffel3}
satisfies the above properties.

 Let now $\nabla$ be an affine
connection that satisfies these properties. One first has
\begin{align*}
\nabla_{X_i} X_j  & = \sum_{k=1}^d g(\nabla_{X_i} X_j,X_k) X_k
\end{align*}
Next, by using the fact that $\nabla g=0$ and that the torsion tensor has to be vertical, we easily obtain that the Koszul identity holds for $\nabla$, that is
\begin{align}\label{ko}
g(\nabla_X Y,Z) & = \frac{1}{2}\big\{Xg(Y,Z) + Yg(Z,X) - Zg(X,Y)
\\
& + g([X,Y],Z) - g([Y,Z],X) + g([Z,X],Y)\big\}. \notag
\end{align}
If we use the orthonormality assumptions on the $X_i's$ and the $Z_{mn}$'s,
and \eqref{bra1}, we obtain
\begin{align*}
\Gamma^k_{ij} & = g(\nabla_{X_i} X_j,X_k)  = \frac{1}{2}
\left\{g([X_i,X_j],X_k) -
g([X_j,X_k],X_i) + g([X_k,X_i],X_j)\right\} \\
& = \frac{1}{2} \left\{\omega_{ij}^k  + \omega_{ki}^j -
\omega_{jk}^i\right\}.
\end{align*}
Similarly, since $T(X,Z_{mn})=0$ and $\nabla_{X_i} Z_{mn} =0 $ one has
\begin{align*}
\nabla_{Z_{mn}} X_i & = \sum_{\ee=1}^d g(\nabla_{Z_{mn}} X_i,X_\ee)
X_\ee .
\end{align*}
Using \eqref{ko} again, along with \eqref{bra2} and \eqref{deltas},
we find
\begin{align*}
g(\nabla_{Z_{mn}} X_i,X_\ee) & = \frac{1}{2}
\left\{g([Z_{mn},X_i],X_\ee) - g([X_i,X_\ee],Z_{mn}) +
g([X_\ee,Z_{mn}],X_i) \right\}
\\
& = \frac{1}{2} \left\{- \delta_{imn}^\ee  + \delta_{\ee mn}^i
\right\} = - \delta_{imn}^\ee.
\end{align*}
\end{proof}

In the sequel it will be useful to have the expression of the
torsion tensor on the vector fields $X_1,...,X_d$.

\begin{proposition}\label{P:torsion}
For every $k,\ee =1,...,d$ one has
\[
T(X_\ee,X_k) = - \sum_{m,n=1}^{\di} \gamma^{mn}_{\ee k} Z_{mn}.
\]
\end{proposition}

\begin{proof}
One has
\begin{align*}
T(X_\ee,X_k) & = \nabla_{X_\ee} X_k - \nabla_{X_k} X_\ee -
[X_\ee,X_k]
\\
& = \sum_{s=1}^d \left(\Gamma^s_{\ee k} - \Gamma^s_{k\ee} -
\omega^s_{\ee k}\right) X_s - \sum_{m,n=1}^{\di} \gamma^{mn}_{\ee k}
Z_{mn}.
\end{align*}
Using \eqref{christoffel1} and the skew-symmetry of the matrix
$[\omega^s_{ij}]_{i,j=1,...,d}$ we now easily conclude that
\[
\Gamma^s_{\ee k} - \Gamma^s_{k\ee} - \omega^s_{\ee k} = 0.
\]
\end{proof}

We also record the following consequence of \eqref{christoffel1}
\begin{equation}\label{christoffel4}
\nabla_{X_i} X_i = - \sum_{j=1}^d \omega^i_{ij} X_j.
\end{equation}

\begin{remark}
We can observe that
\[
L=\sum_{i=1}^d X^2_i-\nabla_{X_i} X_i
\]
\end{remark}

\begin{example}
For the Example \ref{riemannian}, it is readily checked that $\nabla$ is the Levi-Civita connection.
\end{example}

\begin{example}
Let us now identify $\nabla$ for the Example \ref{horizontalbochner}, whose notations are in force in what follows. Let us  recall (see for instance Chapter 3 in \cite{baudoin}) that the Ehresmann connection form $\alpha$  on $\mathcal{O} \left(
\mathbb{M}\right)$   is the unique
skew-symmetric matrix $\alpha$ of one forms on $\mathcal{O} \left(
\mathbb{M}\right)$ such that:
\begin{enumerate}
\item $\alpha (X) =0$ if and only if $X \in \mathcal{H} \mathcal{O} ( \mathbb{M}
)$;
\item $V_{\alpha (X) }=X$ if and only if $X \in \mathcal{V} \mathcal{O} ( \mathbb{M}
)$,
\end{enumerate}
where $ \mathcal{H} \mathcal{O} ( \mathbb{M})$ denotes the horizontal bundle and $\mathcal{V} \mathcal{O} ( \mathbb{M}
)$ the vertical bundle. It is then easily checked that for a vector field $Y$ on $ \mathcal{O} ( \mathbb{M})$,
\[
\nabla_Y H_i=\sum_{k=1}^d \alpha_j^k (Y) H_k.
\]
Let us observe for later use that if $X,Y$ are smooth horizontal vector fields then we have for the torsion:
\[
T(X,Y)=-V_{\Omega (X,Y)}.
\]
\end{example}

\begin{example}
 For the Example \ref{CR}, it is an immediate consequence of Theorem 1.3 in \cite{CR} that $\nabla$ is the Webster-Tanaka
connection.

\end{example}

\subsection{Sub-Riemannian distance and local volume growth}\label{SS:srd}

In sub-Riemannian geometry the Riemannian distance $d_R$ of $\mathbb
M$ is most of the times confined to the background (see in this
regard the discussion in section 0.1 of Gromov's
\emph{Carnot-Carath\'eodory spaces seen from within} in \cite{Be}).
There is another distance on $\mathbb M$, that was introduced by
Carath\'eodory in his seminal paper \cite{Car}, which plays a
central role. A piecewise $C^1$ curve $\gamma:[0,T]\to \mathbb M$ is
called subunitary at $x$ if for every $\xi\in T_x\mathbb M$ one has
\[
g_R(\gamma'(t),\xi)^2 \le
 \sum_{i=1}^d g_R(X_i(\gamma(t)),\xi)^2.
\]
We define the subunit length of $\gamma$ as $\ell_s(\gamma) = T$. If
we indicate with $S(x,y)$ the family of subunit curves such that
$\gamma(0) = x$ and $\gamma(T) = y$, then thanks to the fundamental
accessibility theorem of Chow-Rashevsky the connectedness of
$\mathbb M$ implies that $S(x,y) \not= \varnothing$ for every
$x,y\in \mathbb M$, see \cite{Ch}, \cite{Ra}. This allows to define
the sub-Riemannian distance on $\mathbb M$ as follows
\[
d(x,y) = \inf\{\ell_s(\gamma)\mid \gamma\in S(x,y)\}.
\]
We refer the reader to the cited contribution of Gromov to
\cite{Be}, and to the opening article by Bella\"iche in the same
volume.

We next recall that a metric space $(S,\rho)$ is called a
length-space, or intrinsic, if for any $x,y\in S$
\[
\rho(x,y)=\inf\ell(\gamma_{xy}),
\]
where the infimum is taken over all continuous, rectifiable curves
$\gamma_{xy}$, joining $x$ to $y$. For a continuous curve
$\gamma:[a,b]\to S$, one defines
$\ell(\gamma)=\sup\sum^p_{i=1}\rho(\gamma(t_i),\gamma(t_{i+1}))$,
the supremum being taken on all finite partitions
$a=t_1<t_2<\dots<t_p<t_{p+1}=b$ of the interval $[a,b]$. Since by
the triangle inequality we trivially have
$\rho(x,y)\leq\ell(\gamma_{xy})$ for any continuous curve joining
$x$ to $y$, it follows that in any metric space $(S,\rho)$,
\[
\rho(x,y)\leq\inf\ell(\gamma_{xy}),\  \ \ \ x,y\in S.
\]
In particular, such inequality is therefore valid in the space
$(\mathbb M,d)$. Now it is proved in Proposition 2.2 in \cite{pisa}
that in any sub-Riemannian space the opposite inequality is also
valid. Therefore, every sub-Riemannian space $(\mathbb M,d)$ is a
length-space.

\

Another elementary consequence of the Chow-Rashevsky theorem is that
$i:(\mathbb M,d) \hookrightarrow (\mathbb M,d_R)$ is continuous. On
the other hand, it was proved in \cite{NSW} that for any connected
set $\Omega \subset \mathbb M$ which is bounded in the distance
$d_R$ there exist $C = C(\Omega)>0$, and $\epsilon =
\epsilon(\Omega)>0$, such that
\[
d(x,y) \leq  C d_R(x,y)^\epsilon,\ \ \ x, y \in \Omega.
\]
This implies that also the inclusion $i: (\mathbb M,
d_R)\hookrightarrow (\mathbb M, d)$ is continuous, and thus, the
topologies of $d_R$ and $d$ coincide. In particular, compact sets
are the same in either topology. The metric space $(M,d)$ is locally
compact and, furthermore, for every compact set $K\subset \mathbb M$
there exists $r_0(K)>0$ such that for $x\in K$ and $0<r<r_0$, the
closed balls $\overline B(x,r)$ in the metric $d$ are compact, see
Proposition 1.1 in \cite{GN}. If $\bM$ is unbounded and the vector
fields grow too fast at infinity, then balls of large radii may not
be bounded in $d_R$ in general, and the space $(M,d)$ may fail to be
complete. The lack of completeness and the fact that not all metric
balls are bounded in the metric  $d_R$ are two equivalent properties
in view of the following powerful generalization of the classical
Theorem of Hopf-Rinow due to Cohn-Vossen, see \cite{Bus}.

\begin{theorem}[of Hopf-Rinow type]\label{T:HR} In any locally compact
length-space $(S,\rho)$, the completeness of the metric space is
equivalent to the compactness of the closed balls.
\end{theorem}

Since as we have mentioned $(M,d)$ is a locally compact
length-space, if we want to guarantee the compactness of closed
balls (in the $d$ metric) of arbitrary radii, we need to assume that
$(M,d)$ is a complete metric space. By what has been said, this is
equivalent to requiring that the space $(\bM,d_R)$ is complete.

\

In section \ref{S:myer} it will be expedient to work with yet
another distance on $\mathbb M$. For every $x, y \in \mathbb{M}$, we
set
\begin{equation}\label{d}
\rho(x,y)=\sup_{f\in C^\infty(\mathbb M), \Gamma_\infty(f) \le 1}
\left| f(x) -f(y) \right|,
\end{equation}
where $\Gamma_\infty(f) = \underset{\mathbb M}{\sup}\ \Gamma(f,f)$,
and the quantity $\Gamma(f,f)$ is defined in \eqref{gamma} below.
Thanks to Lemma 5.43 in \cite{CKS} we know that
\[
d(x,y) = \rho(x,y),\ \ \ x, y\in \mathbb M,
\]
hence we can work indifferently with either one of the distances $d$
or $\rho$.

The following fundamental result proved by Nagel, Stein and Wainger
in \cite{NSW} provides a uniform local control of the growth of the
metric balls in $(\bM,d)$.

\begin{theorem}\label{T:doubling}
For any $x\in \bM$ there exist constants $C(x), R(x)>0$ such that
with $Q(x) = \log_2 C(x)$ one has
\[
\mu(B(x,tr)) \ge C(x)^{-1} t^{Q(x)} \mu(B(x,r)),\ \ \ 0\le t\le 1,\
0<r\le R(x).
\]
Given any compact set $K\subset \bM$ one has
\[
\underset{x\in K}{\inf}\ C(x) >0,\ \ \ \underset{x\in K}{\inf}\ R(x)
>0.
\]
\end{theorem}

\subsection{Geodesics}\label{S:geodesics}

In this section we prove a first variation formula for the geodesics
of the metric space $(\bM,d)$. More precisely, we will describe with
our connection $\nabla$ the local minima of the energy functional.
To this end, let us first introduce a  family of natural
skew-symmetric linear operators of the horizontal bundle. For $V \in
\mathcal{H}$, we define:
\[
J_{mn} (V)=\sum_{i,j=1}^d  g( X_j,V) \gamma_{ij}^{mn} X_i,\ \ \ \ 1
\le m,n \le \di.
\]
With this in hands, we can now extend the method of Rumin
\cite{rumin} to provide the equation of the geodesics. We first have
the following result.

\begin{lemma}
For  $1 \le m,n \le \di$, let us denote by $\theta_{mn}$ the
one-form on $\bM$ such that $\theta_{mn} (Z_{mn})=1$ , $\theta_{mn}
(Z_{pq})=0$ if $\{m,n\} \neq \{p,q \}$ and $\theta (X_i)=0$. If
$V_1$ is a smooth vector field on $\bM$, and $V_2$ is a smooth
horizontal vector field, we have
\[
d\theta_{mn} (V_1,V_2)=-g(V_1,J_{mn}V_2).
\]
\end{lemma}

\begin{proof}
Let us first assume that $V_1$ and $V_2$ are both horizontal. From Cartan's formula
\[
d\theta_{mn} (V_1,V_2)= V_1 \theta_{mn} (V_2)-V_2 \theta_{mn} (V_1)-\theta_{mn} ([V_1,V_2]).
\]
Therefore
\[
d\theta_{mn} (V_1,V_2)= -\theta_{mn} ([V_1,V_2]).
\]
>From the definition of $\theta_{mn}$, we easily check that for $1\le
i,j\le d$,
\[
\theta_{mn} ([X_i,X_j])=\gamma_{ij}^{mn},
\]
so that
\[
d\theta_{mn} (V_1,V_2)=-g(V_1,J_{mn}V_2).
\]
Now, if $V_1=Z_{pq}$, since from our assumption \eqref{bra2} the
bracket $[Z_{pq},X_i]$ is always a horizontal vector field, again
from Cartan's formula we obtain
\[
d\theta_{mn} (V_1,V_2)=0.
\]
\end{proof}

\begin{proposition}
Let $x,y \in \bM$ and denote by $\mathcal H(x,y)$ the set of
horizontal curves $\gamma(t)$, $0\le t \le1$, going from $x$ to $y$.
Then $\gamma \in \mathcal H(x,y)$ is a critical point of the energy
functional
\[
E(\gamma)=\int_0^1 g(\gamma'(t),\gamma'(t)) dt
\]
if and only if there exist constants $a_{mn}$, $1\le m,n \le \di$
such that
\[
\nabla_{\gamma'(t)} \gamma'(t)=\sum_{m,n=1}^{\di} a_{mn} J_{mn}
\gamma'(t).
\]
\end{proposition}

\begin{proof}
Let us consider a family of curve $\gamma_u(t)$, $-\epsilon \le u
\le \epsilon$, $0 \le t \le 1$, such that $\gamma_0(t) \in
\mathcal{H}(x,y)$ and $V=\frac{\partial \gamma_u}{\partial u}_{|
u=0}$ is a Legendre vector field of the constrained problem
corresponding to the variational problem we are looking at. We
denote $X=\gamma'(t)$. We have
\[
V(x)=V(y)=0,
\]
and
\[
[V,X] \in \mathcal{H}.
\]
Since the torsion of $\nabla$ is vertical, we now compute
\begin{align*}
\frac{\partial E(\gamma_u)}{\partial u}_{| u=0}& =\int_0^1 V g(X,X) dt \\
 &=2\int_0^1  g(\nabla_V X,X) dt \\
 &=2\int_0^1  g(\nabla_X V,X) dt \\
 &=2 \int_0^1 Xg(V,X)- g( V,\nabla_X X) dt \\
 &=-2 \int_0^1  g( V,\nabla_X X) dt
\end{align*}
Therefore, if $\gamma$ is a critical point, we must have $g(X,\nabla_X X)=0$. We now write
\[
\nabla_X X = \sum_{m,n=1}^{\di} a_{mn}(t) J_{mn} X +Y,
\]
where $Y$ is orthogonal to $X$ and to $\sum_{m,n=1}^{\di} a_{mn}(t)
J_{mn} X$. We have then
\[
g(V,\sum_{m,n=1}^{\di} a_{mn}(t) J_{mn}X)=\sum_{m,n=1}^{\di}
a_{mn}(t) g(V,J_{mn}X).
\]
But, from the previous lemma and Cartan's formula
\[
 g(V,J_{mn}X)=-d\theta_{mn} (V,X)=X \theta(V).
 \]
 Therefore,
\begin{align*}
\frac{\partial E(\gamma_u)}{\partial u}_{| u=0}&=-2 \int_0^1
\sum_{m,n=1}^{\di} a_{mn}(t) X \theta(V) +g(V,Y)  dt
\end{align*}
Integrating by parts, we obtain
\begin{align*}
\frac{\partial E(\gamma_u)}{\partial u}_{| u=0}&=-2 \int_0^1
\sum_{m,n=1}^{\di} -a'_{mn}(t)  \theta(V) +g(V,Y)  dt,
\end{align*}
so that $\gamma_0$ is a critical point if and only if $a_{mn}$ is constant and $Y=0$.
\end{proof}

\section{The curvature tensor}\label{S:ct}

In this section we introduce a first-order differential quadratic
form which plays a pervasive role in the results of this paper. If
$f\in C^\infty(\bM)$ we define
\begin{align}\label{R}
\mathcal R(f,f) & = \sum_{k,\ee=1}^d \bigg\{\bigg(\sum_{j=1}^d
\sum_{m,n=1}^\di \gamma_{kj}^{mn} \delta_{jmn}^\ee\bigg) +
\sum_{j=1}^d
(X_\ee\omega^j_{kj} - X_j\omega^k_{\ee j}) \\
& + \sum_{i,j=1}^d \omega_{ji}^i \omega^\ee_{k j} - \sum_{i=1}^d
\omega_{k i}^i \omega_{\ee i}^i + \frac{1}{2} \sum_{1\le i<j\le d}
\bigg(\omega^\ee_{ij} \omega^k_{ij} - (\omega_{\ee j}^i +\omega_{\ee
i}^j)(\omega^i_{kj} + \omega^j_{ki})\bigg)\bigg\}X_k f X_\ee f
\notag\\
&   +  \sum_{k=1}^d \sum_{m,n=1}^\di \bigg(\sum_{\ell,j=1}^d
\omega_{j\ee}^\ee \gamma_{kj}^{mn} + \sum_{1\le \ee<j\le d}
\omega^k_{\ee j} \gamma^{mn}_{\ee j}  - \sum_{j=1}^d X_j
\gamma^{mn}_{kj}\bigg) Z_{mn} f X_k f \notag\\
& + \frac{1}{2} \sum_{1\le \ee<j\le d}\bigg(\sum_{m,n=1}^\di
\gamma^{mn}_{\ee j} Z_{mn} f\bigg)^2. \notag
\end{align}

The geometric meaning of the differential quadratic form $\mathcal
R(f,f)$ will become fully clear in section \ref{S:bochner}. The main
purpose of the present section is to prove its tensorial nature. In
fact, in Proposition \ref{P:miniblackhole} below we show that
$\mathcal R(f,f)$ can be expressed solely in terms of the curvature
and torsion tensors with respect to the canonical connection
$\nabla$ introduced in section \ref{SS:cc}. At the end of this
section we work out the formulas for $\mathcal R(f,f)$ in some
special examples. In example \ref{E:Rriemann} we show that in the
Riemannian case this form coincides with the Riemannian Ricci
tensor.

In what follows $\nabla$ indicates the canonical connection on $\bM$
introduced in section \ref{SS:cc}. Let us recall that the torsion
tensor is given by
\[
T(X,Y)=\nabla_{X} Y -\nabla_{Y}X -[X,Y],
\]
the curvature tensor is given by
\[
R(X,Y)Z= \nabla_{X}\nabla_{Y}Z
-\nabla_{Y}\nabla_{X}Z-\nabla_{[X,Y]}Z,
\]
and the Ricci tensor by
\[
\ri(X,Y)=\sum_{i=1}^d g(R(X_i,X)Y,X_i),
\]
where $X,Y,Z$ are smooth vector fields.

\begin{proposition}\label{P:miniblackhole}
If $f : \mathbb{M} \rightarrow \mathbb{R}$ is a smooth function,
then
\begin{align*}
\mathcal{R} (f,f)=  &\sum_{\ee,k=1}^d \emph{Ric}(X_\ee,X_k) X_\ee f
X_k f -( (\nabla_{X_\ee} T) (X_\ee,X_k)f )(X_k f)  +\frac{1}{4}
\left( T(X_\ee,X_k)f\right)^2.
\end{align*}
As a consequence, $\mathcal R$ is a $(0,2)$ tensor on $\bM$ and
therefore it is coordinate free.
\end{proposition}

\begin{proof}
We begin by writing the quantity $\mathcal R(f,f)$ in \eqref{R} as
follows
\[
\mathcal R(f,f) = \mathcal R_I(f,f) + \mathcal R_{II}(f,f) +
\mathcal R_{III}(f,f),
\]
where
\begin{align*}
\mathcal R_I(f,f) & = \sum_{k,\ee=1}^d \bigg\{\bigg(\sum_{j=1}^d
\sum_{m,n=1}^\di \gamma_{kj}^{mn} \delta_{jmn}^\ee\bigg) +
\sum_{j=1}^d
(X_\ee\omega^j_{kj} - X_j\omega^k_{\ee j}) \\
& + \sum_{i,j=1}^d \omega_{ji}^i \omega^\ee_{k j} - \sum_{i=1}^d
\omega_{k i}^i \omega_{\ee i}^i + \frac{1}{2} \sum_{1\le i<j\le d}
\bigg(\omega^\ee_{ij} \omega^k_{ij} - (\omega_{\ee j}^i +\omega_{\ee
i}^j)(\omega^i_{kj} + \omega^j_{ki})\bigg)\bigg\}X_k f X_\ee f,
\notag
\end{align*}
\begin{equation*}
\mathcal R_{II}(f,f)  =  \sum_{k=1}^d \sum_{m,n=1}^\di
\bigg(\sum_{\ell,j=1}^d \omega_{j\ee}^\ee \gamma_{kj}^{mn}  +
\sum_{1\le \ee<j\le d} \omega^k_{\ee j} \gamma^{mn}_{\ee j}  -
\sum_{j=1}^d X_j \gamma^{mn}_{kj}\bigg) Z_{mn} f X_k f,
\end{equation*}
\begin{equation*}
\mathcal R_{III}(f,f) =  \frac{1}{2} \sum_{1\le \ee<j\le
d}\bigg(\sum_{m,n=1}^\di \gamma^{mn}_{\ee j} Z_{mn} f\bigg)^2.
\end{equation*}
The proof will be completed if we show that
\begin{equation}\label{RI}
\mathcal R_I(f,f) = \sum_{k,\ee=1}^d \ri(X_k,X_\ee) X_kf X_\ee f,
\end{equation}
\begin{equation}\label{RII}
\mathcal R_{II}(f,f) = - \sum_{\ee,k=1}^d ((\nabla_{X_\ee}
T)(X_\ee,X_k)f)(X_k f),
\end{equation}
\begin{equation}\label{RIII}
\mathcal R_{III}(f,f) = \frac{1}{4} \sum_{\ee,k=1}^d\left(
T(X_\ee,X_k)f\right)^2.
\end{equation}
To establish \eqref{RI} we observe that
\begin{align}\label{miniblackhole}
\ri(X_k,X_\ee) & = \sum_{i=1}^d g(R(X_i,X_k)X_\ee,X_i) =
\sum_{j=1}^d \sum_{m,n=1}^\di \gamma_{kj}^{mn} \delta^\ee_{jmn} +
\sum_{j=1}^d \left(X_j \Gamma^j_{k\ell} - X_k
\Gamma^j_{j\ell}\right)
\\
& + \sum_{i,j=1}^d \left(\Gamma^j_{k\ee} \Gamma^i_{ij} -
\Gamma^j_{i\ee}\Gamma^i_{kj} - \omega^j_{ik} \Gamma^i_{j\ee}\right).
\notag\end{align} The identity \eqref{miniblackhole} can be verified
in a standard fashion. Our goal is thus proving that
\begin{align*}
 &  \sum_{k,\ee=1}^d \left(  \sum_{j=1}^d
\left(X_j \Gamma^j_{k\ell} - X_k \Gamma^j_{j\ell}\right) +
\sum_{i,j=1}^d \left(\Gamma^j_{k\ee} \Gamma^i_{ij} -
\Gamma^j_{i\ee}\Gamma^i_{kj} - \omega^j_{ik} \Gamma^i_{j\ee}\right)\right) (X_kf) (X_\ee f) \\
= & \sum_{k,\ee=1}^d \bigg(\sum_{j=1}^d (X_\ee\omega^j_{kj} -
X_j\omega^k_{\ee j}) + \sum_{i,j=1}^d \omega_{ji}^i \omega^\ee_{k j}
 - \sum_{i=1}^d \omega_{k i}^i \omega_{\ee i}^i
 \\
 & + \frac{1}{2}
\sum_{1\le i<j\le d} \bigg(\omega^\ee_{ij} \omega^k_{ij} -
(\omega_{\ee j}^i +\omega_{\ee i}^j)(\omega^i_{kj} +
\omega^j_{ki})\bigg) \bigg) (X_kf) (X_\ee f).
\end{align*}
To show this we use the formula \eqref{christoffel1}
\[
 \Gamma_{ij}^k= \frac{1}{2} \left(\omega_{ij}^k +\omega_{ki}^j - \omega_{jk}^i\right),
 \]
which allows to express the Christoffel symbols $\Gamma^k_{ij}$ in
terms of the structural constants $\omega^k_{ij}$. Keeping in mind
that $\omega^\ee_{ij} = - \omega^\ee_{ji}$, one easily recognizes
that
\[
\sum_{k,l=1}^d \left(\sum_{j=1}^d \left(X_j \Gamma^j_{k\ell} - X_k
\Gamma^j_{j\ell}\right)\right) X_kf X_\ee f = \sum_{k,l=1}^d \left(
\sum_{j=1}^d \left(X_\ee \omega^j_{kj} - X_j \omega^k_{\ell
j}\right)\right) X_kf X_\ee f . \] To complete the proof of
\eqref{RI} we are thus left with proving that
\begin{align}\label{ft}
& \sum_{k,\ee=1}^d  \sum_{i,j=1}^d \left(\Gamma^j_{k\ee}
\Gamma^i_{ij} - \Gamma^j_{i\ee}\Gamma^i_{kj} - \omega^j_{ik}
\Gamma^i_{j\ee}\right) X_kf X_\ee f
\\
& = \sum_{k,\ee=1}^d \bigg\{\sum_{i,j=1}^d \omega_{ji}^i
\omega^\ee_{k j} - \sum_{i=1}^d \omega_{k i}^i \omega_{\ee i}^i +
\frac{1}{2} \sum_{1\le i<j\le d} \bigg(\omega^\ee_{ij} \omega^k_{ij}
- (\omega_{\ee j}^i +\omega_{\ee i}^j)(\omega^i_{kj} +
\omega^j_{ki})\bigg)\bigg\}X_k f X_\ee f. \notag\end{align} We now
notice that again from \eqref{christoffel1} and the skew-symmetry of
$\omega^\ee_{jk}$ in the lower indices we obtain
\begin{align*}
\sum_{k,\ee=1}^d  \sum_{i,j=1}^d \Gamma^j_{k\ee} \Gamma^i_{ij} X_kf
X_\ee f
= & \sum_{k,\ee=1}^d  \sum_{i,j=1}^d \frac{\Gamma^j_{k\ee} + \Gamma^j_{\ee k}}{2} \Gamma^i_{ij} X_kf X_\ee f \\
=& \sum_{k,\ee=1}^d  \sum_{i,j=1}^d \frac{\omega^k_{j\ee} +
\omega^\ee_{j k}}{2} \omega^i_{ij} X_kf X_\ee f = \sum_{k,\ee=1}^d
\sum_{i,j=1}^d \omega^\ee_{j k} \omega^i_{ij} X_kf X_\ee f \\
& = \sum_{k,\ee=1}^d \sum_{i,j=1}^d \omega^i_{ji}\omega^\ee_{kj}
X_kf X_\ee f.
\end{align*}
Next, we have
\begin{align*}
& - \sum_{k,\ee=1}^d \sum_{i,j=1}^d \left(\Gamma^j_{i\ell}
\Gamma^i_{kj} + \omega^j_{ik} \Gamma^i_{j\ell}\right) X_kf X_\ee f =
- \sum_{k,\ee=1}^d \sum_{i,j=1}^d \Gamma^j_{i\ell}
\left(\Gamma^i_{kj} + \omega^i_{jk}\right) X_kf X_\ee f.
\end{align*}
By \eqref{christoffel1} one has
\[
\Gamma^i_{kj} + \omega^i_{jk} = \frac{1}{2} \left(\omega^j_{ik} +
\omega^k_{ij} + \omega^i_{jk}\right),
\]
and thus
\begin{align*}
& - \sum_{k,\ee=1}^d \sum_{i,j=1}^d \left(\Gamma^j_{i\ell}
\Gamma^i_{kj} + \omega^j_{ik} \Gamma^i_{j\ell}\right) X_kf X_\ee f
\\
& = \frac{1}{4} \sum_{k,\ee=1}^d \sum_{i,j=1}^d
\left(\omega^\ee_{ij} - \omega^j_{\ee i} - \omega^i_{\ee
j}\right)\left(\omega^k_{ij} + \omega^j_{ik} +
\omega^i_{jk}\right)X_kf X_\ee f
\\
& = \sum_{k,\ee=1}^d \sum_{i=1}^d \omega^i_{ki} \omega^i_{\ee i}
X_kf X_\ee f + \frac{1}{2} \sum_{k,\ee=1}^d \sum_{1\le i<j\le d}
\omega^\ee_{ij} \omega^k_{ij} X_kf X_\ee f
\\
& - \frac{1}{2} \sum_{k,\ee=1}^d \sum_{1\le i<j\le d} (\omega^i_{\ee
j} + \omega^j_{\ee i})(\omega^i_{lj} + \omega^j_{ki}) X_kf X_\ee f,
\end{align*}
where the last equality is obtained by expanding the product and
canceling equal terms. This proves \eqref{ft}, thus completing the
proof of \eqref{RI}.

Next, we turn to the proof of \eqref{RII}. One has
\begin{align}\label{dertensor}
& \sum_{\ee,k=1}^d ((\nabla_{X_\ee} T)(X_\ee,X_k)f) X_k f =
\sum_{\ee,k=1}^d \nabla_{X_\ee}\left(T(X_\ee,X_k)\right)f X_kf
\\
& - \sum_{\ee,k=1}^d T\left(\nabla_{X_\ee} X_\ee,X_k\right)f X_k f -
\sum_{\ee,k=1}^d  T\left(X_\ee,\nabla_{X_\ee} X_k\right)f X_k f.
\notag\end{align} Proposition \ref{P:torsion} and the fact that
$\nabla Z = 0$ now give
\begin{align*}
& \sum_{\ee,k=1}^d \nabla_{X_\ee}\left(T(X_\ee,X_k)\right)f X_kf = -
\sum_{\ee,k=1}^d \sum_{m,n=1}^\di (X_\ee \gamma^{mn}_{\ee k}) Z_{mn}
X_k f =  \sum_{k=1}^d \sum_{m,n=1}^d \sum_{j=1}^d (X_j
\gamma^{mn}_{kj}) Z_{mn} X_k f .
\end{align*}
Next, using \eqref{christoffel4} and Proposition \ref{P:torsion}
again one obtains
\begin{align*}
& - \sum_{\ee,k=1}^d T\left(\nabla_{X_\ee} X_\ee,X_k\right)f X_k f =
\sum_{\ee,k=1}^d \sum_{j=1}^d \omega^\ee_{\ee j} T(X_j,X_k) f X_k f
\\
& = - \sum_{k=1}^d \sum_{m,n=1}^\di  \sum_{j,\ee=1}^d
\omega^\ee_{j\ee} \gamma^{mn}_{kj} Z_{mn} f X_k f.
\end{align*}
Finally,
\begin{align*}
&- \sum_{\ee,k=1}^d  T\left(X_\ee,\nabla_{X_\ee} X_k\right)f X_k f =
- \sum_{j,\ee,k=1}^d \Gamma_{\ee k}^j T\left(X_\ee,X_j\right)f X_k f
\\
& =  \sum_{k=1}^d \sum_{m,n=1}^\di \sum_{j,\ee=1}^d \gamma_{\ee
j}^{mn}\Gamma_{\ee k}^j Z_{mn} f X_k f = \frac{1}{2} \sum_{k=1}^d
\sum_{m,n=1}^\di  \sum_{j,\ee=1}^d \gamma_{\ee j}^{mn}(\omega_{\ee
k}^j + \omega^k_{j\ee} + \omega^\ee_{jk})Z_{mn} f X_k f
\\ & = \frac{1}{2}
\sum_{k=1}^d \sum_{m,n=1}^\di \sum_{1\le \ee<j\le d} \gamma_{\ee
j}^{mn}(\omega_{\ee k}^j + \omega^k_{j\ee} + \omega^\ee_{jk}-
\omega^\ee_{jk} - \omega^k_{\ee j} - \omega^j_{\ee k})Z_{mn} f X_k f
\\
& = - \sum_{k=1}^d \sum_{m,n=1}^\di \sum_{1\le \ee<j\le d}
\gamma_{\ee j}^{mn} \omega^k_{\ee j}Z_{mn} f X_k f.
\end{align*}
Substitution in \eqref{dertensor} gives
\begin{align*}
\sum_{\ee,k=1}^d ((\nabla_{X_\ee} T)(X_\ee,X_k)f)(X_k f) = -
\mathcal R_{II}(f,f),
\end{align*}
which proves \eqref{RII}. In order to complete the proof we are left
with establishing \eqref{RIII}. From Proposition \ref{P:torsion} and
the skew-symmetry of $\gamma^{mn}_{\ee k}$ in the lower indices, we
easily find
\begin{align*}
\frac{1}{4} \sum_{\ee,k=1}^d \left(T(X_\ee,X_k)f\right)^2 &
=\frac{1}{2} \sum_{\ee,k=1}^d \left(\sum_{m,n=1}^\di \gamma_{\ee
k}^{mn} Z_{mn}f\right)^2 = \mathcal R_{III}(f,f).
\end{align*}
This finishes the proof.

\end{proof}

We can now compute $\mathcal{R}$ in the geometric examples
introduced in section \ref{SS:framework}.

\begin{example}\label{E:Rriemann}
For a Riemannian manifold $\mathbb M$, see Example \ref{riemannian},
we have
\[
\mathcal{R} (f,f)=\emph{Ric}(\nabla f, \nabla f).
\]
This follows immediately from Proposition \ref{P:miniblackhole}.
\end{example}

\begin{example}\label{E:Rnilpotent}
When $\mathbb M$ is a graded nilpotent Lie group of step two, see
Example \ref{nilpotent}, we immediately obtain from \eqref{R}
\[
\mathcal R(f,f) = \frac{1}{4} \sum_{i,j=1}^\di (Z_{ij}f)^2 =
 \frac{1}{4} \sum_{i,j=1}^\di
([X_i,X_j]f)^2.
\]
We note here that, in terms of the quadratic form introduced in
\eqref{gamma2Z} below, one has
\begin{equation}\label{Rcg}
\mathcal{R} (f,f)= \frac{1}{4} \Gamma^Z(f,f). \end{equation}
\end{example}

\begin{example}\label{E:Rhorbochner}
In the Example \ref{horizontalbochner} we have
\begin{align*}
\mathcal{R} (f,f)= &\sum_{j,k=1}^d \emph{Ric}^* (H_j,H_j) (H_j
f)(H_k f) +V_{\nabla_{H_j} \Omega (H_j,H_k)} f H_kf   +\frac{1}{4}
\left( V_{\Omega(H_j,H_k)} f\right)^2,
\end{align*}
where for horizontal vector fields $X$ and $Y$,
\[
\emph{Ric}^* (X,Y)= \emph{Ric} (\pi_* X,\pi_*Y),
\]
with \emph{Ric} denoting the Ricci tensor of $\mathbb{M}$.
\end{example}

\section{Bochner type formulas}\label{S:bochner}

Our  goal in this section is to prove sub-Riemannian Bochner type
formulas which play a central role in this paper. Since the operator
$L$ is two-step generating, we prove two types of such formulas, one
involving the commutation between the horizontal gradient and $L$,
the other one involving the commutation between a vertical gradient
and $L$. As it is to be surmised, the horizontal one will prove
quite involved technically.

\

In the sequel, we will use the following differential bilinear
forms:
\begin{equation}\label{gamma}
\Gamma(f,g) =\frac{1}{2}(L(fg)-fLg-gLf)=\sum_{i=1}^d X_i f X_i g,
\end{equation}
\begin{equation}\label{gamma2}
\Gamma_{2}(f,g) = \frac{1}{2}\big[L\Gamma(f,g) - \Gamma(f,
Lg)-\Gamma (g,Lf)\big],
\end{equation}
\begin{equation}\label{gammaZ}
\Gamma ^Z (f,g)=\sum_{i,j=1}^\di  \left(Z_{ij} f \right)  \left(
Z_{ij}g \right),
\end{equation}
\begin{equation}\label{gamma2Z}
\Gamma^Z_{2}(f,g) = \frac{1}{2}\big[L\Gamma^Z (f,g) - \Gamma^Z(f,
Lg)-\Gamma^Z (g,Lf)\big].
\end{equation}

\subsection{The first Bochner  formula}\label{SS:firstbochner}

Henceforth, we adopt the notation
\[
f_{,ij} = \frac{X_i X_j f + X_j X_i f}{2}
\]
for the entries of the symmetrized Hessian of $f$ with respect to
the vector fields $X_1,...,X_d$. Noting that
\[
X_iX_j f\ =\ f_{,ij}\ +\ \frac{1}{2}\ [X_i,X_j] f,
\]
using \eqref{bra1} we obtain the useful formula
\begin{equation}\label{nonsym}
X_iX_jf = f_{,ij} + \frac{1}{2} \sum_{\ee=1}^d \omega^\ee_{ij} X_\ee
f + \frac{1}{2} \sum_{m,n=1}^\di \gamma^{mn}_{ij} Z_{mn} f\ .
\end{equation}

Our principal result of this section is the following:

\begin{theorem}[Horizontal Bochner formula]\label{T:bochner}
For every smooth function $f :\mathbb{M} \rightarrow \mathbb{R}$,
\begin{align}\label{bochner}
\Gamma_{2}(f,f)= & \sum_{\ee=1}^d \left(f_{,\ee\ee} -\sum_{i=1}^d
\omega_{i\ee}^\ee X_i f \right)^2 +2 \sum_{1 \le \ee<j \le d} \left(
f_{,\ee j}
 -\sum_{i=1}^d \frac{\omega_{i\ee}^j +\omega_{ij}^\ee}{2} X_i f \right)^2  \\
 & -2 \sum_{i,j=1}^d \sum_{m,n=1}^\di\gamma_{ij}^{mn} (X_j Z_{mn} f) (X_i f) +\mathcal{R}
 (f,f),
\notag
\end{align}
where $\mathcal R(f,f)$ is the quadratic form defined in \eqref{R}.
\end{theorem}

\begin{proof}
We begin by observing that for any smooth function $F$ on $\mathbb
M$
\[
L(F^2) = 2 F LF + 2 \Gamma(F,F),
\]
This and \eqref{gamma} give
\begin{align*}
L \Gamma(f,f) & = \sum_{i=1}^d L((X_i f)^2) = 2 \sum_{i=1}^d X_i f
L(X_i f) + 2 \sum_{i=1}^d \Gamma(X_i f,X_if) .
\end{align*}
We now have
\begin{align*}
L(X_i f) & = X_0 X_i f + \sum_{j=1}^d X_j^2 X_i f = X_i X_0 f +
[X_0,X_i]f + \sum_{j=1}^d X_j (X_iX_j f) + X_j[X_j,X_i]f
\\
& = X_i X_0 f + [X_0,X_i]f + \sum_{j=1}^d \big\{X_i (X_jX_j f) +
[X_j,X_i]X_j f +  X_j[X_j,X_i]f\big\} \notag\\
& = X_i(Lf) + [X_0,X_i]f + \sum_{j=1}^d \big\{[X_j,X_i]X_j f +
X_j[X_j,X_i]f\big\} \notag\\
& = X_i(Lf) + [X_0,X_i]f + 2 \sum_{j=1}^d [X_j,X_i]X_j f +
\sum_{j=1}^d [X_j,[X_j,X_i]]f. \notag
\end{align*}
Using this identity we find
\begin{align*}
L \Gamma(f,f) & = 2 \sum_{i=1}^d X_i f \left\{X_i(Lf) + [X_0,X_i]f +
2 \sum_{j=1}^d [X_j,X_i]X_j f + \sum_{j=1}^d
[X_j,[X_j,X_i]]f\right\} \\
& + 2 \sum_{i,j=1}^d (X_jX_i f)^2 \\
& = 2 \Gamma(f,Lf) + 2 \sum_{i=1}^d X_i f [X_0,X_i]f + 4
\sum_{i,j=1}^d X_if [X_j,X_i]X_j f + 2 \sum_{i,j=1}^d X_i f
[X_j,[X_j,X_i]]f \\
& + 2 \sum_{i,j=1}^d (X_jX_i f)^2.
\end{align*}

Since, thanks to the skew-symmetry of the matrix
 $\{[X_i,X_j]f\}_{i,j=1,...,d}$, we have
 \[
 \sum_{i,j=1}^d
 f_{,ij} [X_i,X_j] f = 0,
 \]
we find
\begin{align*}
\sum_{i,j=1}^d (X_jX_i f)^2  & =  \sum_{i,j=1}^d f_{,ij}^2 +
 \frac{1}{4}  \sum_{i,j=1}^d ([X_i,X_j] f)^2 +  \sum_{i,j=1}^d
 f_{,ij} [X_i,X_j] f
 \\
 & =  \sum_{i,j=1}^d f_{,ij}^2 +
 \frac{1}{4}  \sum_{i,j=1}^d ([X_i,X_j] f)^2.
 \end{align*}
We thus obtain \begin{align*} \frac{1}{2} \big[L \Gamma(f,f) - 2
\Gamma(f,Lf)\big] & =  \sum_{i=1}^d X_i f [X_0,X_i]f \\
& + 2 \sum_{i,j=1}^d X_i f [X_j,X_i]X_j f +  \sum_{i,j=1}^d X_i f [X_j,[X_j,X_i]]f \\
& +  \sum_{i,j=1}^d f_{,ij}^2 +
 \frac{1}{4}  \sum_{i,j=1}^d ([X_i,X_j] f)^2.
\end{align*}
Since \eqref{gamma2} gives $\Gamma_{2}(f,f) =
\frac{1}{2}\big[L\Gamma(f,f) - 2 \Gamma(f, Lf)\big]$, we conclude
\begin{align}\label{bochner1}
\Gamma_{2}(f,f) & = \sum_{i,j=1}^d f_{,ij}^2 - 2 \sum_{i,j=1}^d X_i
f [X_i,X_j]X_j f
 \\
& +
 \frac{1}{4}  \sum_{i,j=1}^d ([X_i,X_j] f)^2 + \sum_{i=1}^d X_i f [X_0,X_i]f +  \sum_{i,j=1}^d X_i f
[[X_i,X_j],X_j]f. \notag\end{align} To complete the proof we need to
recognize that the right-hand side in \eqref{bochner1} coincides
with that in \eqref{bochner}. With this objective in mind, using
\eqref{bra1} and \eqref{nonsym} we obtain
\begin{align*}
& \sum_{i,j=1}^d f_{,ij}^2  - 2 \sum_{i,j=1}^d X_i f [X_i,X_j]X_j f
\\
& = \sum_{\ell=1}^d f_{,\ell\ell}^2 + 2  \sum_{1 \le \ell<j \le d}
f_{,j\ell}^2 - 2 \sum_{i,j=1}^d X_i f \left(\sum_{\ee=1}^d
\omega_{ij}^\ee X_\ee +\sum_{m,n=1}^\di \gamma_{ij}^{mn}
Z_{mn}\right)X_j f
\\
& = \sum_{\ell=1}^d f_{,\ell\ell}^2 + 2  \sum_{1 \le \ee<j \le d}
f_{,j\ell}^2
\\
& - 2 \sum_{i,j=1}^d \sum_{\ee=1}^d \omega_{ij}^\ee X_\ee X_j f\ X_i
f - 2 \sum_{i,j=1}^d  \sum_{m,n=1}^\di \gamma_{ij}^{mn} Z_{mn}X_j f\
X_i f
\\
& = \sum_{\ell=1}^d f_{,\ell\ell}^2 + 2  \sum_{1 \le \ee<j \le d}
f_{,j\ell}^2  - 2 \sum_{i,j=1}^d \sum_{\ee=1}^d \omega_{ij}^\ee
f_{,\ee j}X_i f \\
&  -  \sum_{i,j=1}^d \sum_{\ee, k=1}^d \omega_{ij}^\ee \omega^k_{\ee
j} X_k f X_i f -  \sum_{i,j=1}^d \sum_{\ee=1}^d \sum_{m,n=1}^\di
\omega_{ij}^\ee \gamma^{mn}_{\ee j} Z_{mn} f\ X_i f
\\
& - 2 \sum_{i,j=1}^d  \sum_{m,n=1}^\di \gamma_{ij}^{mn} Z_{mn}X_j f\
X_i f
\\
& = \sum_{\ell=1}^d f_{,\ell\ell}^2 + 2  \sum_{1 \le \ee<j \le d}
f_{,j\ell}^2  - 2 \sum_{\ee,j=1}^d \left(\sum_{i=1}^d
\omega_{ij}^\ee
X_i f\right) f_{,\ee j} \\
&  -  \sum_{i,j=1}^d \sum_{\ee, k=1}^d \omega_{ij}^\ee \omega^k_{\ee
j} X_k f X_i f -  \sum_{i,j=1}^d \sum_{\ee=1}^d \sum_{m,n=1}^\di
\omega_{ij}^\ee \gamma^{mn}_{\ee j} Z_{mn} f\ X_i f
\\
& - 2 \sum_{i,j=1}^d  \sum_{m,n=1}^\di \gamma_{ij}^{mn} Z_{mn}X_j f\
X_i f
\\
& = \sum_{\ell=1}^d \left(f_{,\ell\ell}^2 - 2 \left(\sum_{i=1}^d
\omega^\ee_{i\ee} X_i f\right) f_{,\ee \ee}\right)
\\
& + 2  \sum_{1 \le \ee<j \le d}\left( f_{,j\ell}^2  - 2 \sum_{1\le
\ee<j\le d} \left(\sum_{i=1}^d \frac{\omega_{ij}^\ee +
\omega_{i\ee}^j}{2}
X_i f\right) f_{,\ee j} \right)
\\
&  -  \sum_{i,j=1}^d \sum_{\ee, k=1}^d \omega_{ij}^\ee \omega^k_{\ee
j} X_k f X_i f -  \sum_{i,j=1}^d \sum_{\ee=1}^d \sum_{m,n=1}^\di
\omega_{ij}^\ee \gamma^{mn}_{\ee j} Z_{mn} f\ X_i f
\\
& - 2 \sum_{i,j=1}^d  \sum_{m,n=1}^\di \gamma_{ij}^{mn} Z_{mn}X_j f\
X_i f.
\end{align*}
If we now complete the squares we obtain
\begin{align}\label{monster1}
& \sum_{i,j=1}^d f_{,ij}^2  - 2 \sum_{i,j=1}^d X_i f [X_i,X_j]X_j f
\\
& = \sum_{\ee=1}^d \left( f_{,\ell\ell} -\sum_{i=1}^d
\omega_{i\ee}^\ee X_i f \right)^2 +  2 \sum_{1 \le \ee<j \le d}
\left( f_{,j\ell} -\sum_{i=1}^d \frac{\omega_{ij}^\ee
+\omega_{i\ee}^j}{2} X_i f \right)^2
\notag\\
& - \sum_{\ee=1}^d \left(\sum_{i=1}^d \omega_{i\ee}^\ee X_i f
\right)^2 -  2 \sum_{1 \le \ee<j \le d} \left(\sum_{i=1}^d
\frac{\omega_{ij}^\ee +\omega_{i\ee}^j}{2} X_i f \right)^2
\notag\\
&  -  \sum_{i,j,k,\ee=1}^d  \omega_{ij}^\ee \omega^k_{\ee j} X_k f
X_i f -  \sum_{i,j=1}^d \sum_{\ee=1}^d \sum_{m,n=1}^\di
\omega_{ij}^\ee \gamma^{mn}_{\ee j} Z_{mn} f\ X_i f
\notag\\
& - 2 \sum_{i,j=1}^d \sum_{m,n=1}^\di \gamma_{ij}^{mn} X_j Z_{mn}f\
X_i f - 2 \sum_{i,j=1}^d  \sum_{m,n=1}^\di \gamma_{ij}^{mn}
[Z_{mn},X_j] f\ X_i f. \notag
\end{align}
Next, we have from \eqref{X0}
\begin{align}\label{monster2}
\sum_{i=1}^d X_i f [X_0,X_i]f  & =  \sum_{i,j,k,\ee=1}^d
\omega^k_{jk} \omega^\ee_{ij} X_\ee f X_i f
\\
& + \sum_{i=1}^d \sum_{j,k=1}^d \sum_{m,n=1}^\di \omega_{jk}^k
\gamma_{ij}^{mn} Z_{mn} f  X_i f + \sum_{i=1}^d \sum_{j,k=1}^d (X_i
\omega^k_{jk}) X_if X_j f, \notag\end{align} and also
\begin{align*}
\sum_{i,j=1}^d X_i f [[X_i,X_j],X_j]f  & = \sum_{i,j=1}^d \sum_{\ee
= 1}^d [\omega^\ee_{ij} X_\ee,X_j]f X_if + \sum_{i,j=1}^d \sum_{m,n
= 1}^\di [\gamma^{mn}_{ij} Z_{mn},X_j]f X_if
\\
& =  \sum_{i,j=1}^d \sum_{\ee=1}^d \omega^\ee_{ij} X_i f [X_\ee,X_j]
f - \sum_{i,j=1}^d \sum_{\ee=1}^d (X_j\omega^\ee_{ij}) X_if X_\ee f
\\
& + \sum_{i,j=1}^d \sum_{m,n=1}^\di \gamma^{mn}_{ij} [Z_{mn},X_j]f
X_i f - \sum_{i,j=1}^d \sum_{m,n=1}^\di (X_j \gamma^{mn}_{ij})
Z_{mn} f X_i f.
\end{align*}
Using \eqref{bra1} we find
\begin{align}\label{monster3}
\sum_{i,j=1}^d X_i f [[X_i,X_j],X_j]f  & = \sum_{i,j=1}^d
\sum_{\ee,k=1}^d \omega^\ee_{ij} \omega^k_{\ee j} X_i f X_kf +
\sum_{i,j,\ee=1}^d \sum_{m,n=1}^\di \omega^\ee_{ij} \gamma^{mn}_{\ee
j} Z_{mn}f X_i f
\\
& + \sum_{i,j=1}^d \sum_{m,n=1}^\di \gamma^{mn}_{ij} [Z_{mn},X_j]f
X_i f - \sum_{i,j=1}^d \sum_{m,n=1}^\di (X_j \gamma^{mn}_{ij})
Z_{mn} f
X_i f \notag\\
& - \sum_{i,j=1}^d \sum_{\ee=1}^d (X_j\omega^\ee_{ij}) X_if X_\ee f.
\notag
\end{align}
Again by \eqref{bra1} we have
\begin{align}\label{monster4}
\frac{1}{4} \sum_{i,j=1}^d ([X_i,X_j]f)^2 & = \frac{1}{2} \sum_{1\le
i<j\le d}\left(\sum_{\ee=1}^d \omega^\ee_{ij} X_\ee f\right)^2 +
\frac{1}{2} \sum_{1\le i<j\le d}\left(\sum_{m,n=1}^\di
\gamma^{mn}_{ij} Z_{mn} f\right)^2 \\
& + \sum_{1\le i<j\le d}\sum_{\ee=1}^d \sum_{m,n=1}^\di
\omega^\ee_{ij} \gamma^{mn}_{ij} Z_{mn}f X_\ee f. \notag\end{align}
Substituting \eqref{monster1}-\eqref{monster4} in \eqref{bochner1}
we obtain
\begin{align*}
\Gamma_2(f,f) & = \sum_{\ee=1}^d \left( f_{,\ell\ell} -\sum_{i=1}^d
\omega_{i\ee}^\ee X_i f \right)^2 +  2 \sum_{1 \le \ee<j \le d}
\left( f_{,j\ell} -\sum_{i=1}^d \frac{\omega_{ij}^\ee
+\omega_{i\ee}^j}{2} X_i f \right)^2
\\
& - 2 \sum_{i,j=1}^d  \sum_{m,n=1}^\di \gamma_{ij}^{mn} X_j Z_{mn}f\
X_i f + \mathcal{M}onster
\end{align*}
where we have let
\begin{align}\label{Monster}
\mathcal{M}onster & = - \sum_{\ee=1}^d \left(\sum_{i=1}^d
\omega_{i\ee}^\ee X_i f \right)^2 -  2 \sum_{1 \le \ee<j \le d}
\left(\sum_{i=1}^d \frac{\omega_{ij}^\ee +\omega_{i\ee}^j}{2} X_i f
\right)^2
\\
&  + \sum_{i,j,k,\ee=1}^d \omega^k_{jk} \omega^\ee_{ij} X_\ee f X_i
f -  \sum_{i,j,k,\ee=1}^d  \omega_{ij}^k \omega^\ee_{k j} X_\ee f
X_i f -  \sum_{i,j=1}^d \sum_{\ee=1}^d \sum_{m,n=1}^\di
\omega_{ij}^\ee \gamma^{mn}_{\ee j} Z_{mn} f\ X_i f
\notag\\
&  - \sum_{i,j=1}^d  \sum_{m,n=1}^\di \gamma_{ij}^{mn} [Z_{mn},X_j]
f\ X_i f  + \sum_{i=1}^d \sum_{j,k=1}^d \sum_{m,n=1}^\di
\omega_{jk}^k
\gamma_{ij}^{mn} Z_{mn} f  X_i f \notag\\
& + \sum_{i=1}^d \sum_{j,k=1}^d (X_i \omega^k_{jk}) X_if X_j f +
\sum_{i,j=1}^d \sum_{\ee,k=1}^d \omega^\ee_{ij} \omega^k_{\ee j} X_i
f X_kf + \sum_{i,j,\ee=1}^d \sum_{m,n=1}^\di \omega^\ee_{ij}
\gamma^{mn}_{\ee j} Z_{mn}f X_i f
\notag\\
&  - \sum_{i,j=1}^d \sum_{m,n=1}^\di (X_j \gamma^{mn}_{ij}) Z_{mn} f
X_i f  - \sum_{i,j=1}^d \sum_{\ee=1}^d (X_j\omega^\ee_{ij}) X_if
X_\ee f
\notag\\
& + \frac{1}{2} \sum_{1\le i<j\le d}\left(\sum_{\ee=1}^d
\omega^\ee_{ij} X_\ee f\right)^2 + \frac{1}{2} \sum_{1\le i<j\le
d}\left(\sum_{m,n=1}^\di
\gamma^{mn}_{ij} Z_{mn} f\right)^2 \notag\\
& + \sum_{1\le i<j\le d}\sum_{\ee=1}^d \sum_{m,n=1}^\di
\omega^\ee_{ij} \gamma^{mn}_{ij} Z_{mn}f X_\ee f.\notag
\end{align}
Simplifying the expression we obtain
\begin{align}\label{Monster2}
\mathcal{M}onster & = - \sum_{k,\ee=1}^d \sum_{i=1}^d \omega_{k i}^i
\omega_{\ee i}^i X_k f X_\ee f \notag\\
& -  \frac{1}{2} \sum_{k,l=1}^d \sum_{1 \le i<j \le d} (\omega_{\ee
j}^i +\omega_{\ee i}^j)(\omega^i_{kj} + \omega^j_{ki}) X_k f X_\ee f
\\
& + \sum_{k,\ee=1}^d \sum_{j=1}^d (X_\ee\omega^j_{kj} -
X_j\omega^k_{\ee j}) X_kf X_\ee f    +  \sum_{i,j,k,\ee=1}^d
\omega_{ji}^i \omega^\ee_{k j}  X_k f X_\ee f \notag\\
& + \frac{1}{2} \sum_{k,\ee=1}^d \sum_{1\le i<j\le d}
\omega^\ee_{ij} \omega^k_{ij} X_k f X_\ee f + \sum_{k,j=1}^d
\sum_{m,n=1}^\di \gamma_{kj}^{mn} [X_j,Z_{mn}] f\ X_k f
\notag\\
&   + \sum_{i=1}^d \sum_{j,k=1}^d \sum_{m,n=1}^\di \omega_{jk}^k
\gamma_{ij}^{mn} Z_{mn} f  X_i f + \sum_{1\le i<j\le
d}\sum_{\ee=1}^d \sum_{m,n=1}^\di \omega^\ee_{ij} \gamma^{mn}_{ij}
Z_{mn}f X_\ee f
\notag\\
&  - \sum_{i,j=1}^d \sum_{m,n=1}^\di (X_j \gamma^{mn}_{ij}) Z_{mn} f
X_i f  + \frac{1}{2} \sum_{1\le i<j\le d}\left(\sum_{m,n=1}^\di
\gamma^{mn}_{ij} Z_{mn} f\right)^2. \notag
\end{align}

To complete the proof we need to recognize that the right-hand side
of \eqref{Monster2} coincides with $\mathcal R(f,f)$ defined by
\eqref{R}. If we now use \eqref{bra2} we obtain
\begin{align*}
\sum_{k,j=1}^d \sum_{m,n=1}^\di \gamma_{kj}^{mn} [X_j,Z_{mn}] f\ X_k
f & = \sum_{k,j=1}^d \sum_{m,n=1}^\di \sum_{\ee=1}^d
\gamma_{kj}^{mn}
\delta_{jmn}^\ee X_k f X_\ee f \\
& = \sum_{k,\ee=1}^d \bigg(\sum_{j=1}^d \sum_{m,n=1}^\di
\gamma_{kj}^{mn} \delta_{jmn}^\ee\bigg) X_k f X_\ee f,
\end{align*}

Substituting in \eqref{Monster2} we find
\begin{align}\label{Monster3}
\mathcal{M}onster & = \sum_{k,\ee=1}^d \bigg\{\bigg(\sum_{j=1}^d
\sum_{m,n=1}^\di \gamma_{kj}^{mn} \delta_{jmn}^\ee\bigg) +
\sum_{j=1}^d
(X_\ee\omega^j_{kj} - X_j\omega^k_{\ee j}) \\
& + \sum_{i,j=1}^d \omega_{ji}^i \omega^\ee_{k j} - \sum_{i=1}^d
\omega_{k i}^i \omega_{\ee i}^i + \frac{1}{2} \sum_{1\le i<j\le d}
\bigg(\omega^\ee_{ij} \omega^k_{ij} - (\omega_{\ee j}^i +\omega_{\ee
i}^j)(\omega^i_{kj} + \omega^j_{ki})\bigg)\bigg\}X_k f X_\ee f
\notag\\
&   + \sum_{i=1}^d \sum_{j,k=1}^d \sum_{m,n=1}^\di \omega_{jk}^k
\gamma_{ij}^{mn} Z_{mn} f  X_i f + \sum_{1\le i<j\le
d}\sum_{\ee=1}^d \sum_{m,n=1}^\di\omega^\ee_{ij} \gamma^{mn}_{ij}
Z_{mn}f X_\ee f
\notag\\
&  - \sum_{i,j=1}^d \sum_{m,n=1}^\di (X_j \gamma^{mn}_{ij}) Z_{mn} f
X_i f  + \frac{1}{2} \sum_{1\le i<j\le d}\left(\sum_{m,n=1}^\di
\gamma^{mn}_{ij} Z_{mn} f\right)^2. \notag
\end{align}
Rearranging the indices we can rewrite \eqref{Monster3} as follows
\begin{align*}\label{Monster4}
\mathcal{M}onster & = \sum_{k,\ee=1}^d \bigg\{\bigg(\sum_{j=1}^d
\sum_{m,n=1}^\di  \gamma_{kj}^{mn} \delta_{jmn}^\ee\bigg) +
\sum_{j=1}^d
(X_\ee\omega^j_{kj} - X_j\omega^k_{\ee j}) \\
& + \sum_{i,j=1}^d \omega_{ji}^i \omega^\ee_{k j} - \sum_{i=1}^d
\omega_{k i}^i \omega_{\ee i}^i + \frac{1}{2} \sum_{1\le i<j\le d}
\bigg(\omega^\ee_{ij} \omega^k_{ij} - (\omega_{\ee j}^i +\omega_{\ee
i}^j)(\omega^i_{kj} + \omega^j_{ki})\bigg)\bigg\}X_k f X_\ee f
\notag\\
&   +  \sum_{k=1}^d \sum_{m,n=1}^\di \bigg(\sum_{\ell,j=1}^d
\omega_{j\ee}^\ee \gamma_{kj}^{mn} + \sum_{1\le \ee<j\le d}
\omega^k_{\ee j} \gamma^{mn}_{\ee j}  - \sum_{j=1}^d X_j
\gamma^{mn}_{kj}\bigg) Z_{mn} f X_k f \notag\\
& + \frac{1}{2} \sum_{1\le \ee<j\le d}\bigg(\sum_{m,n=1}^\di
\gamma^{mn}_{\ee j} Z_{mn} f\bigg)^2 = \mathcal R(f,f). \notag
\end{align*}
This completes the proof of Theorem \ref{T:bochner}.
\end{proof}

\begin{remark}
When $\mathbb M$ is a graded nilpotent Lie group of step two, then
using the structural constants in Example \ref{nilpotent} we obtain
from formula \eqref{bochner}
\[
\Gamma_{2}(f,f)=  \sum_{\ee=1}^d f_{,\ee\ee}^2 + 2 \sum_{1 \le \ee<j
\le d}  f_{,\ee j}^2 - 2 \sum_{i,j=1}^d \sum_{m,n=1}^\di
\gamma_{ij}^{mn} (X_j Z_{mn} f) (X_i f) +\mathcal{R} (f,f).
\]
Keeping \eqref{Rcg} in mind, we thus we obtain from the latter
equation
\[ \Gamma_{2}(f,f)= ||\nabla^2_H f||^2 + \frac{1}{4}
\sum_{i,j=1}^d ([X_i,X_j]f)^2 + 2 \sum_{i,j=1}^d [X_i,X_j] X_i f X_j
f, \] which is Proposition 3.3 in \cite{Gman} (the latter, however,
is relative to groups of arbitrary step). Here, $\nabla_H^2 f =
[f_{,ij}]$ represents the symmetrized horizontal Hessian of $f$.
\end{remark}

\subsection{The second Bochner formula}\label{SS:secondbochner}

To prove our next result we will need the following lemma.

\begin{lemma}\label{L:comm}
For every $m,n=1,...,\di$ one has
\[
[L,Z_{mn}] = 0.
\]
\end{lemma}

\begin{proof}
We begin by observing that, as consequence of the assumption
\eqref{deltas}, for every $m,n=1,...,\di$, the
commutator$[L,Z_{mn}]$ is a vector field. Indeed,  given  a smooth
function $f$, using the formulas \eqref{bra2}, \eqref{X0} and
\eqref{nonsym}, one obtains by elementary computations
\begin{align*}
[L,Z_{mn}]f & = [X_0,Z_{mn}]f + \sum_{i=1}^d X_i [X_i,Z_{mn}]f + [X_i,Z_{mn}] X_i f\\
& = - \sum_{i,k,\ee=1}^d \omega_{ik}^k \delta^{\ee}_{imn} X_\ee f+
\sum_{i,k=1}^d Z_{mn}(\omega^k_{ik}) X_i f+ \sum_{i,\ee=1}^d (X_i \delta^\ee_{imn}) X_\ee f +
\delta^\ee_{imn} (X_\ee X_i f+X_iX_\ee f)\\
& = - \sum_{i,k,\ee=1}^d \omega_{ik}^k \delta^{\ee}_{imn} X_\ee f+
\sum_{i,k=1}^d Z_{mn}(\omega^k_{ik}) X_i f+ \sum_{i,\ee=1}^d (X_i \delta^\ee_{imn}) X_\ee f
\end{align*}
where we used the crucial fact that
\[
\delta^\ee_{imn}  = -\delta^i_{\ee mn}.
\]
To complete the proof let $\phi,
\psi\in C^\infty_0(\bM)$, then \eqref{symm} gives
\begin{align*}
<[L,Z_{mn}]^*\phi,\psi> & = <\phi,[L,Z_{mn}]\psi> =
<\phi,L(Z_{mn}\psi)> - <\phi,Z_{mn}(L\psi)>
\\
& = <L\phi,Z_{mn}\psi> + <Z_{mn}\phi,L\psi> = <L(Z_{mn}\phi,\psi> -
<Z_{mn}(L\phi),\psi>
\\
& = <[L,Z_{mn}]\phi,\psi>,
\end{align*}
and thus $[L,Z_{mn}]^* = [L,Z_{mn}]$, i.e. $[L,Z_{mn}]$ is a
symmetric vector field. Such vector fields must vanish. Indeed, if
$V$ is symmetric vector field, one has for any $\phi, \psi\in
C^\infty(\bM)$
\begin{align*}
0 = & <V1,\phi \psi> = <V^* 1,\phi \psi>  = <1,V(\phi \psi)> =
<1,\phi V\psi> + <1,\psi V\phi>
\\
& = <\phi, V\psi> + <\psi, V\phi> = 2 <V\phi,\psi>.
\end{align*}
Taking $\psi = V\phi$ we conclude that it must be $V\phi = 0$. By
the arbitrariness of $\phi\in C^\infty(\bM)$ we conclude $V=0$.

\end{proof}

Our second Bochner type formula is expressed by the following
proposition.

\begin{proposition}[Vertical Bochner formula]\label{P:bochner2}
For every smooth function $f :\mathbb{M} \rightarrow \mathbb{R}$,
\[
\Gamma^Z_{2}(f,f) = \sum_{m,n=1}^\di \Gamma( Z_{mn} f,Z_{mn}f )
\]
\end{proposition}
\begin{proof}
>From \eqref{gamma2Z} we obtain
\[
\Gamma_2^Z(f,f) = \frac{1}{2} \bigg[L\Gamma^Z(f,f) - 2
\Gamma^Z(f,Lf)\bigg].
\]
We now have from \eqref{gammaZ} and from Lemma \ref{L:comm}
\begin{align*}
L\Gamma^Z(f,f) & = \sum_{m,n=1}^\di L((Z_{mn}f)^2) = 2
\sum_{m,n=1}^\di Z_{mn}f\ L(Z_{mn}f) + 2 \sum_{m,n=1}^\di
\sum_{k=1}^d \big(X_k(Z_{mn}f)\big)^2
\\
& = 2 \sum_{m,n=1}^\di Z_{mn}f\ Z_{mn}(Lf) + 2 \sum_{m,n=1}^\di
\sum_{k=1}^d \big(X_k(Z_{mn}f)\big)^2
\\
& = 2 \Gamma^Z(f,Lf) + 2 \sum_{m,n=1}^\di \Gamma(Z_{mn}f,Z_{mn}f),
\end{align*}
where in the last term we have used \eqref{gamma}.
\end{proof}

\section{A sub-Riemannian curvature dimension
inequality}\label{S:cd}

As a consequence of our two Bochner's formulas we obtain a
generalization of the curvature dimension inequality of Bakry. A new
feature of such inequality is the presence in the left hand-side of
the vertical quadratic form $\Gamma^Z_{2}(f,f)$. For a smooth
function $f$, we define the quantity
\begin{equation}\label{T}
\mathcal{T} (f,f)= \sum_{j=1}^d \sum_{m,n=1}^\di \left(\sum_{i=1}^d
\gamma_{ij}^{mn} X_i f \right)^2.
\end{equation}
We observe that with the notations of Section \ref{S:geodesics},
\[
\mathcal{T} (f,f)=\sum_{m,n=1}^\di g(J_{mn}(\nabla f), J_{mn}(\nabla f)).
\]
We notice explicitly that, since in the Riemannian case
$\gamma^{mn}_{ij} = 0$, see Example \ref{riemannian}, in that case
we have
\begin{equation}\label{Triemann}
\mathcal T(f,f) = 0. \end{equation} When instead $\mathbb M$ is a
graded nilpotent Lie group of step two, then from the expressions of
the $\gamma^{mn}_{ij} = 0$ in example \ref{nilpotent} we find
\begin{equation}\label{Tnilpotent}
\mathcal T(f,f) = \frac{d-1}{2} \Gamma(f,f).
\end{equation}

\begin{proposition}[Sub-Riemannian curvature-dimension inequality]\label{P:Inequality1}
For  every smooth function $f :\mathbb{M} \rightarrow \mathbb{R}$ and every $\nu >0$,
\[
\Gamma_{2}(f,f)+\nu \Gamma^Z_{2}(f,f) \ge \frac{1}{d} (Lf)^2 + \mathcal{R} (f,f) -\frac{1}{\nu} \mathcal{T} (f,f).
\]
\end{proposition}

\begin{proof}
>From \eqref{L} and \eqref{X0} and Schwarz inequality we find
\[
Lf = \sum_{\ee=1}^d \left(f_{,\ee \ee} - \sum_{i=1}^d
\omega^\ee_{i\ee} X_i f\right) \leq \sqrt d \left(\sum_{\ee=1}^d
\left(f_{,\ee \ee} - \sum_{i=1}^d \omega^\ee_{i\ee} X_i
f\right)^2\right)^{1/2}
\]
>From this inequality and from Theorem \ref{T:bochner}, using Schwarz
inequality again we obtain for every $\nu>0$
\begin{align*}
\frac{1}{d} (Lf)^2 & \leq \Gamma_2(f,f) + 2 \sum_{j=1}^d
\sum_{m,n=1}^\di \left(\sum_{i=1}^d \gamma^{mn}_{ij}X_if\right)
X_j(Z_{mn}f)  - \mathcal R(f,f)
\\
&  - 2 \sum_{1 \le \ee<j \le d} \left( f_{,\ee j}
 -\sum_{i=1}^d \frac{\omega_{i\ee}^j +\omega_{ij}^\ee}{2} X_i f \right)^2
\\
& \leq \Gamma_2(f,f) + \nu \sum_{j=1}^d \sum_{m,n=1}^\di
(X_j(Z_{mn}f))^2 + \frac{1}{\nu}  \sum_{j=1}^d \sum_{m,n=1}^\di
\left(\sum_{i=1}^d \gamma^{mn}_{ij}X_if\right)^2 - \mathcal R(f,f).
\end{align*}
Using Proposition \ref{P:bochner2} and the definition \eqref{T} we
find
\[
\frac{1}{d} (Lf)^2 \leq \Gamma_2(f,f) + \nu \Gamma^Z_{2}(f,f) +
\frac{1}{\nu} \mathcal T(f,f) - \mathcal R(f,f),
\]
which gives the desired conclusion.
\end{proof}

\begin{example}\label{CDriemann}
On a Riemannian manifold $\mathbb M$ with Laplace-Beltrami operator
$L$ we have $\Gamma_2^Z(f,f,) = \mathcal T(f,f) = 0$ (see
\eqref{Triemann}), and $\mathcal R(f,f) = \emph{Ric}\ (\nabla
f,\nabla f)$ (see Example \ref{E:Rriemann}). Proposition
\ref{P:Inequality1} thus gives
\[
\Gamma_{2}(f,f) \ge \frac{1}{d} (Lf)^2 + \emph{Ric}\ (\nabla
f,\nabla f).
\]
When \begin{equation}\label{riccirho} \emph{Ric}(V,V) \ge \rho_1
|V|^2, \ \ \ \rho_1\in \mathbb R, \end{equation}
we thus recover the
Riemannian curvature-dimension inequality
\begin{equation}\label{Rcdrho}
\Gamma_2(f,f) \ge \frac{1}{d} (L f)^2 + \rho_1 \Gamma(f,f).
\end{equation}
\end{example}

\begin{remark}\label{R:ricci}
It is important to keep in mind that, in view of Theorem 1.3 in
\cite{SVR} and Proposition 3.3 in \cite{bakry-tata}, the Riemannian
curvature dimension inequality \eqref{Rcdrho} is, in fact,
equivalent to \eqref{riccirho}.
\end{remark}

\begin{example}\label{CDnilpotent}
Let us assume that $L$ is the sub-Laplacian on a two-step graded
nilpotent Lie group, then for every smooth function $f$ and every
$\nu >0$,
 \[
 \Gamma_{2}(f,f)+\nu \Gamma^Z_{2}(f,f) \ge \frac{1}{d} (Lf)^2  -\frac{d-1}{2\nu}  \Gamma (f,f)
 +\frac{1}{4} \Gamma^Z (f,f).
 \]

\end{example}

\section{Gradient estimates for the heat semigroup}

In this whole section we assume that the metric space $(\bM,d)$ is
complete, where $d$ denotes the Carath\'eodory distance associated
to $L$ (see section \ref{SS:srd}). Recall that this is equivalent to
assuming that $\bM$ is complete with respect to its Riemannian
distance $d_R$. We also suppose that there exist constants $\rho_1
\in \mathbb{R}$, $\rho_2
>0$ and $\kappa >0$ such that for every smooth function $f
:\mathbb{M} \rightarrow \mathbb{R}$:
\begin{equation}\label{geombounds}
\begin{cases}
\mathcal{R}(f,f) \ge \rho_1 \Gamma (f,f) +\rho_2 \Gamma^Z (f,f),
\\
\mathcal{T}(f,f) \le \kappa \Gamma (f,f).
\end{cases}\end{equation}

We emphasize that, as we have seen in example \ref{CDriemann} and
remark \ref{R:ricci}, the assumptions \eqref{geombounds} should be
considered as the sub-Riemannian analogue of \eqref{riccirho}. In
particular, the requirement $\rho_1\ge 0$ in \eqref{geombounds}
corresponds to the Riemannian Ric$\ge 0$.

According to Proposition \ref{P:Inequality1} the assumptions
\eqref{geombounds} imply, for every smooth function $f :\mathbb{M}
\rightarrow \mathbb{R}$ and every $\nu >0$,
\begin{equation}\label{CD}
\Gamma_{2}(f,f)+\nu \Gamma^Z_{2}(f,f) \ge \frac{1}{d} (Lf)^2 + \left( \rho_1 -\frac{\kappa}{\nu}\right)  \Gamma (f,f) + \rho_2 \Gamma^Z (f,f).
\end{equation}
This is a sub-Riemannian version of the curvature-dimension
inequality \eqref{curvature_dimension}.

\subsection{The heat semigroup}

As an operator defined on $C^\infty_0(\bM)$ the operator $L$ is
symmetric with respect to the measure $\mu$ and non positive: For $f
\in C^\infty_0(\bM)$, $<Lf,f> \le 0$. Therefore, it admits a
self-adjoint extension to $L^2 (\bM,\mu)$, the Friedrichs extension.
Following an argument of Strichartz \cite{Strichartz}, Theorem 7.3
p. 246 and p. 261, we now prove that $L$ is essentially self-adjoint
on $C^\infty_0(\bM)$ (the completeness of  the metric space
$(\mathbb{M},d)$ is crucial here).

In what follows, to distinguish it from the canonical connection
$\nabla$ introduced in section \ref{SS:cc}, we use the notation
$\nabla_R$ for the Riemannian connection on $\bM$. We recall that,
thanks to the assumption \eqref{LB},
\[
\|\nabla_R f \|^2=\Gamma(f)+\frac{1}{2}\Gamma^Z(f),\ \ \ f\in
C^\infty(\bM).
\]

\begin{lemma}\label{L:exhaustion}
There exists an increasing
sequence $h_n\in C^\infty_0(\bM)$  such that $h_n\nearrow 1$ on
$\bM$, and $||\nabla_R h_n||_{\infty} \to 0$, as $n\to \infty$.
\end{lemma}

\begin{proof}
As in \cite{GW}, if we fix a base point $x_0\in \bM$, we can find an
exhaustion function $\rho\in C^\infty(\bM)$ such that
\[ |\rho - d_R(x_0,\cdot)| \le L,\ \ \ \ \ \   |\nabla_R \rho|\le L
\ \ \text{on}\ \bM. \] By the completeness of $(M,d_R)$ and the
Hopf-Rinow theorem, the level sets $\Omega_s = \{x\in \bM\mid
\rho(x)<s\}$ are relatively compact and, furthermore, $\Omega_s
\nearrow \bM$ as $s\to \infty$. We now pick an increasing sequence
of functions $\phi_n\in C^\infty([0,\infty))$ such that
$\phi_n\equiv 1$ on $[0,n]$, $\phi_n \equiv 0$ outside $[0,2n]$, and
$|\phi_n'|\le \frac{2}{n}$. If we set $h_n(x) = \phi_n(\rho(x))$,
then we have $h_n\in C^\infty_0(\bM)$, $h_n\nearrow1$ on $\bM$ as
$n\to \infty$, and \[ ||\nabla_R h_n||_{\infty} \le \frac{2L}{n}.
\]
\end{proof}

In what follows we define the action of the operator $-L$ on
$C^\infty_0(\bM)$ by the equation
\[
<-L \phi,\psi> = \int_{\bM} \Gamma(\phi,\psi) d\mu,\ \ \ \psi\in
C^\infty_0(\bM).
\]
We recall our assumption \eqref{symm}. As a corollary of the
previous lemma, we obtain:

\begin{proposition}
The operator $L$ is essentially self-adjoint on $C^\infty_0(\bM)$.
\end{proposition}

\begin{proof}
Denote by $\overline L$ the Friedrichs extension of $L$ initially
defined on $C^\infty_0(\bM)$. If $L^*$ is the adjoint of $\overline
L$, then according to Reed-Simon \cite{reed}, p. 137, it is enough
to prove that the eigenvalues of the adjoint $L^*$ are negative. We
thus have to show that if $L^* f=\lambda f$ with $\lambda >0$, then
$f=0$. From the hypoellipticity of $L$, we first deduce that $f$ has
to be a smooth function. Now, for $h \in C^\infty_0(\bM)$,
\begin{align*}
< \Gamma( f, h^2f) >=-<f, L(h^2f)>=-<L^*f ,h^2f>=-\lambda <f,h^2f>=-\lambda <f^2,h^2> \le 0.
\end{align*}
Since
\[
 \Gamma( f, h^2f)=h^2 \Gamma (f,f)+2 fh \Gamma(f,h),
 \]
 we deduce that
 \[
 <h^2, \Gamma (f,f)>+2 <fh, \Gamma(f,h)> \le 0.
 \]
 Therefore, by Schwarz inequality
 \[
 <h^2, \Gamma (f,f)> \le 4 \| f |_2^2 \| \Gamma (h,h) \|_\infty.
 \]
 If we now use the sequence $h_n$ constructed in Lemma \ref{L:exhaustion} and let $n \to \infty$, we obtain $\Gamma(f,f)=0$ and therefore $f=0$, as desired.
\end{proof}

If $L=-\int_0^{+\infty} \lambda dE_\lambda$ denotes the spectral
decomposition of $L$ in $L^2 (\bM,\mu)$, then by definition, the
heat semigroup $(P_t)_{t \ge 0}$ is given by $P_t= \int_0^{+\infty}
e^{-\lambda t} dE_\lambda$. It is a family of bounded operators on
$L^2 (\bM,\mu)$. Since the quadratic form $-<f,Lf>$ is a Dirichlet
form in the sense of Fukushima \cite{Fu}, we deduce from the above
result that $C^\infty_0(\bM)$ is dense in the domain of $P_t$ and
that $(P_t)_{t \ge 0}$ is a sub-Markov semigroup: it transforms
positive functions into positive functions and satisfies
\begin{equation}\label{submarkov}
P_t 1 \le 1.
\end{equation}
This property implies in particular \begin{equation}\label{sminfty}
||P_tf||_{L^1(\bM)} \le ||f||_{L^1(\bM)},\ \ \
||P_tf||_{L^\infty(\bM)} \le ||f||_{L^\infty(\bM)},
\end{equation}
and therefore by the Theorem of Riesz-Thorin
\begin{equation}\label{smp}
||P_tf||_{L^p(\bM)} \le ||f||_{L^p(\bM)},\ \  1\le p\le \infty.
\end{equation}
Moreover, it can be shown as in \cite{strichartz1}, Theorem 3.9:

\begin{proposition}\label{uniquenessLp}
The unique solution of the Cauchy problem
\[
\begin{cases}
\frac{\p u}{\p t} - Lu = 0,
\\
u(x,0) = f(x),\ \ \ \   f\in L^p(\bM),1<p<+\infty,
\end{cases}
\]
that satisfies $\| u(\cdot,t) \|_p \le C e^{Mt}$, for some constants $C$ and $M$, is given by $u(x,t)=P_t f(x)$.
\end{proposition}
Due to the hypoellipticity of $L$, $(t,x) \rightarrow P_t f(x)$ is
smooth on $\mathbb{M}\times (0,\infty) $ and
\[ P_t f(x)  = \int_{\mathbb M} p(x,y,t) f(y) d\mu(y),\ \ \ f\in
C^\infty_0(\mathbb M),\] where $p(x,y,t) > 0$ is the so-called heat
kernel associated to $P_t$. Such function is smooth outside the
diagonal of $\bM\times \bM$, and it is symmetric, i.e., \[ p(x,y,t)
= p(y,x,t). \]
 By the
semi-group property for every $x,y\in \bM$ and $0<s,t$ we have
\begin{equation}\label{sgp}
p(x,y,t+s) = \int_\bM p(x,z,t) p(z,y,s) d\mu(z) = \int_\bM p(x,z,t)
p(y,z,s) d\mu(z) = P_s(p(x,\cdot,t))(y).
\end{equation}

\subsection{A variational inequality}\label{S:vi}

The goal of this section is to establish a variational inequality
which is the cornerstone of all the gradient estimates we shall
obtain in the sequel. Henceforth, to simplify the notation we let \[
\Gamma(f) = \Gamma(f,f), \ \ \ \ f\in C^1(\mathbb M). \] For the
sequel it will be convenient to observe that if $\phi\in C^2(\mathbb
R)$, $f\in C^2(\mathbb M)$, then
\begin{equation}\label{comp}
L(\phi\circ f) = \phi''(f) \Gamma(f) +  \phi'(f) Lf.
\end{equation}
In particular, if $f\in C^\infty_0(\bM)$ and $f\geq 0$, then for any
fixed $T>0$ and $t<T$, we obtain from \eqref{comp}
\begin{equation}\label{complog}
\frac{L P_{T-t} f}{P_{T-t} f} = L(\log P_{T-t} f) + \Gamma(\log
P_{T-t} f).
\end{equation}

For such $f$'s we now introduce the two functionals
\[
\Phi_1 (t)=P_t \left( (P_{T-t} f) \Gamma (\ln P_{T-t}f) \right),
\]
\[
\Phi_2 (t)=P_t \left( (P_{T-t} f) \Gamma^Z (\ln P_{T-t}f) \right).
\]
Notice that for every $T>0$ the function $u(x,t) = P_{T-t} f(x)$
satisfies the backward Cauchy problem
\[
\begin{cases}  Lu + \frac{\p
u}{\p t} = 0 \ \ \ u\ge 0\ \text{in}\ M\times (-\infty,T), \\
u(x,T) = f(x). \end{cases} \] It follows that
\begin{equation}\label{phis}
\begin{cases}
\Phi_1(0) = P_T(f)\Gamma(\ln P_T f),\ \ \ \ \Phi_2(0) =
P_T(f)\Gamma^Z(\ln P_T f),
\\
\Phi_1(T) = P_T(f \Gamma(\ln f)), \ \ \ \ \Phi_2(T) = P_T(f
\Gamma^Z(\ln f)).
\end{cases}
\end{equation}

\begin{lemma}\label{L:derivatives}
We have
\[
\Phi'_1 (t)=2P_t \left( (P_{T-t} f) \Gamma_2 (\ln P_{T-t}f) \right),
\]
and
\[
\Phi'_2 (t)=2P_t \left( (P_{T-t} f) \Gamma_2^Z (\ln P_{T-t}f) \right).
\]
\end{lemma}

\begin{proof}
Let $V$ be a smooth vector field on $\mathbb{M}$ and let $\Phi_V$ be the functional
\[
\Phi_V(t)=P_t \left( (P_{T-t} f) (V \ln P_{T-t} f)^2 \right).
\]
We have
\begin{align*}
\Phi'_V(t) & = P_t \left( L ( (P_{T-t} f) (V \ln P_{T-t} f)^2)
\right)-P_t \left( (LP_{T-t} f) (V \ln P_{T-t} f)^2 \right) \\
& -2P_t \left( (VP_{T-t} f) V\left( \frac{LP_{T-t} f}{P_{T-t} f}
\right)\right).
\end{align*}
We now compute
\begin{align*}
& P_t \left( L ( (P_{T-t} f) (V \ln P_{T-t} f)^2) \right)  = P_t
\left( L (P_{T-t} f )(V\ln P_{T-t}f )^2\right) + P_t \left( (P_{T-t}f) L (V\ln P_{T-t}f )^2 \right)
\\
& +2 P_t \left( \Gamma( P_{T-t}f, (V\ln P_{T-t} f)^2 )\right)
\\
& = P_t \left( L (P_{T-t} f )(V\ln P_{T-t}f )^2\right)+2 P_t \left( (P_{T-t}f) (V \ln P_{T-t} f ) (LV \ln P_{T-t} f) \right) \\
& + 2 P_t \left( (P_{T-t}f) \Gamma(V\ln P_{T-t}f , V\ln P_{T-t} f )
\right)+4 P_t \left( V\ln P_{T-t} f\Gamma( P_{T-t}f, V\ln P_{T-t} f)
)\right).
\end{align*}
By taking into account \eqref{complog} we obtain
\begin{align*}
\Phi'_V(t)= & 2 P_t \left( (VP_{T-t} f) [L,V] \ln P_{T-t} f \right) +2 P_t \left( (P_{T-t}f)\Gamma(V\ln P_{T-t}f , V\ln P_{T-t} f ) \right) \\
 &+4P_t \left( V\ln P_{T-t} f\Gamma( P_{T-t}f, V\ln P_{T-t} f) )\right)-2P_t \left((VP_{T-t} f) V\Gamma( \ln P_{T-t} f, \ln P_{T-t}f) \right).
\end{align*}
We now observe that
\[
 V\Gamma( \ln P_{T-t} f, \ln P_{T-t}f) =2\Gamma ( \ln P_{T-t}f, Z\ln P_{T-t}f)+2\sum_{i=1}^d (X_i \ln P_{T-t} f) ( [Z,X_i]\ln P_{T-t} f).
 \]
 Thus,
\begin{align}\label{phiV}
\Phi'_V(t)& = 2P_t \left( (P_{T-t} f) \Gamma_2^V (\ln P_{T-t}f,\ln
P_{T-t}f ) \right) \\
& - 4 P_t \left((VP_{T-t}f) \sum_{i=1}^d (X_i \ln P_{T-t} f) (
[V,X_i]\ln P_{T-t} f) \right), \notag \end{align} where we have
defined
 \[
\Gamma^V_{2}(f,g) = \frac{1}{2}(L((Vf)(Vg)) - VfVLg-VgVLf.
\]
We first apply \eqref{phiV} with $V=X_j$, and sum in $j=1,...,d$ to
obtain
\[
\Phi'_1 (t)=2P_t \left( (P_{T-t} f) \Gamma_2 (\ln P_{T-t}f) \right).
\]
Here, we have used to skew-symmetry of $[X_i,X_j]$, which gives
\[
\sum_{j=1}^d (X_jP_{T-t}f) \sum_{i=1}^d (X_i \ln P_{T-t} f) ( [X_j,X_i]\ln P_{T-t} f)=0.
\]
We next apply \eqref{phiV} with $V=Z_{mn}$ and sum in
$m,n=1,...,\di$, obtaining
\begin{align*}
\Phi'_2 (t) & = 2 P_t \left( (P_{T-t} f) \Gamma_2^Z (\ln P_{T-t}f)
\right) \\
& - 4\sum_{m,n=1}^\di P_t \left( (Z_{mn}P_{T-t}f) \sum_{i=1}^d (X_i
\ln P_{T-t} f) ( [Z_{mn},X_i]\ln P_{T-t} f) \right).
\end{align*}
We now observe that, thanks to the crucial assumption
\eqref{deltas}, one has
\begin{align*}
  & \sum_{m,n=1}^\di  (Z_{mn}P_{T-t}f) \sum_{i=1}^d (X_i \ln P_{T-t} f) ( [Z_{mn},X_i]\ln P_{T-t} f)\\
 =&- \sum_{m,n=1}^\di  (Z_{mn}P_{T-t}f) \sum_{i,\ee=1}^d \delta_{imn}^\ee (X_i \ln P_{T-t} f) ( X_\ee \ln P_{T-t} f)=0.
\end{align*}
Therefore,
\[
\Phi'_2 (t)=2P_t \left( (P_{T-t} f) \Gamma_2^Z (\ln P_{T-t}f) \right).
\]
\end{proof}

\begin{proposition}\label{variational}
Let $b$ be a smooth, positive and decreasing function on the time
interval $[0,T]$. On $\bM\times [0,T]$, we have
\[
\left(- \frac{b'}{2\rho_2} \Phi_1 +b \Phi_2 \right)' \ge  -\frac{2b' \gamma}{d\rho_2} LP_Tf + \frac{b'  \gamma^2}{d\rho_2}  P_T f,
\]
where
\[
\gamma=\frac{d}{4} \left( \frac{b''}{b'} +\frac{\kappa}{\rho_2} \frac{b'}{b} +2\rho_1 \right).
\]
\end{proposition}

\begin{proof}
To prove this result we apply the sub-Riemannian curvature-dimension
inequality (\ref{CD}), in combination with Lemma
\ref{L:derivatives}. If $a$ and $b$ are positive functions we thus
obtain
\begin{align*}
(a \Phi_1 +b \Phi_2)'  & \ge \left(a'+2\rho_1 a -2\kappa \frac{a^2}{b} \right)\Phi_1 +(b'+2\rho_2 a) \Phi_2 +\frac{2a}{d} \left( P_t ( (P_{T-t} f) (L \ln P_{T-t} f )^2 )\right)
\end{align*}
But, for every $\gamma \in \mathbb{R}$,
\[
(L \ln P_{T-t} f )^2 \ge 2\gamma L \ln P_{T-t}f -\gamma^2,
\]
and
\[
 L \ln P_{T-t}f=\frac{LP_{T-t}f}{P_{T-t}f} -\Gamma(\ln P_{T-t} f ).
 \]
Therefore,
 \begin{align*}
(a \Phi_1 +b \Phi_2)'  & \ge \left(a'+2\rho_1 a -2\kappa \frac{a^2}{b}-4\frac{a\gamma}{d} \right)\Phi_1 +(b'+2\rho_2 a) \Phi_2 +\frac{4a\gamma}{d} LP_{T} f -2 \frac{a\gamma^2}{d} P_T f.
\end{align*}
By taking $a,b,\gamma$ such that
\[
a'+2\rho_1 a -2\kappa \frac{a^2}{b} -\frac{4}{d} \gamma a =0,
\]
\[
b'+2\rho_2 a=0,
\]
we obtain the desired conclusion.
\end{proof}

\subsection{Li-Yau type estimates}

In this section, we extend the celebrated Li-Yau inequality in
\cite{LY} to the heat semigroup associated with the subelliptic
operator $L$. Let us mention that, in this setting, related
inequalities were obtained by Cao-Yau \cite{Cao-Yau}. However, these
authors work only in the case of a compact manifold and do not base
their study on the analysis of a tensor like $\mathcal{R}$. As a
consequence they do not obtain a control of the constants in terms
of the geometry of the manifold.

\begin{proposition}[Gradient estimate]\label{P:ge}
Assume that \eqref{geombounds} hold. Let $f\in C^\infty_0(\mathbb
M)$, with $f\geq 0$, then the following inequality holds for $t>0$:
\[
\Gamma (\ln P_t f) +\frac{2 \rho_2}{3}  t \Gamma^Z (\ln P_t f) \le
\left(1+\frac{3\kappa}{2\rho_2}-\frac{2\rho_1}{3} t\right)
\frac{LP_t f}{P_t f} +\frac{d\rho_1^2}{6} t-\frac{\rho_1 d}{2}\left(
1+\frac{3\kappa}{2\rho_2}\right) +\frac{d\left(
1+\frac{3\kappa}{2\rho_2}\right)^2}{2t}.
\]
\end{proposition}

\begin{proof}
If we apply Proposition \ref{variational} with $b(t)=(T-t)^3$, we
obtain:
\[
\gamma (t)=\frac{d}{4} \left( \frac{ 2\rho_1 (T-t)
-\frac{3\kappa}{\rho_2} -2 }{T-t } \right) = \frac{d}{2}\left(\rho_1
 - \frac{1}{T-t}\left(1 + \frac{3\kappa}{2\rho_2}\right)\right),
\]
\[
\int_0^T b'(t) \gamma(t) dt =-\frac{\rho_1 d}{2} T^3 +\frac{3d}{4} \left( 1+ \frac{3\kappa}{2\rho_2}\right) T^2,
\]
and
\[
\int_0^T b'(t) \gamma(t)^2 dt =-\frac{3 d^2}{16} \left(
\frac{4\rho^2_1 }{3} T^3+ 4 \left( 1+ \frac{3\kappa}{2\rho_2}\right)
^2 T -4\rho_1\left( 1+ \frac{3\kappa}{2\rho_2}\right) T^2 \right).
\]
The result easily follows.
\end{proof}

\begin{remark}\label{R:rho1zero}
When $\rho_1 >0$, then \eqref{geombounds} and \eqref{CD} also hold
in particular with $\rho_1 = 0$. As a consequence of Proposition \ref{P:ge}
\begin{align}\label{liyaupositifzero}
\Gamma (\ln P_t f) +\frac{2 \rho_2}{3}  t \Gamma^Z (\ln P_t f) \le
(1+\frac{3\kappa}{2\rho_2}) \frac{LP_t f}{P_t f} +\frac{d\left(
1+\frac{3\kappa}{2\rho_2}\right)^2}{2t}.
\end{align}
This inequality leads to a optimal Harnack inequality only when
$\rho_1 = 0$ (which corresponds to \emph{Ric}$\ge 0$). Sharper
bounds in the strictly positive curvature case will be obtained in
\eqref{gamma_bound} by a different choice of the function $b(t)$.
\end{remark}

\begin{remark}\label{R:ultrac}
If we set \begin{equation}\label{D} D =  d  \left( 1+
\frac{3\kappa}{2\rho_2}\right), \end{equation} then, as a
consequence of \eqref{liyaupositifzero}, we obtain that in the case
$\rho_1 \ge 0$,
\begin{align}\label{ultracontractivity}
\frac{LP_t f}{P_t f} \ge -\frac{D}{2t}.
\end{align}

The constant $- \frac{D}{2}$ in \eqref{ultracontractivity} is, in
general, not sharp, as the example of the heat semigroup on a graded
nilpotent Lie group shows. In such case, in fact, one can argue as
in \cite{Fo} to show that the heat kernel $p(x,y,t)$ is homogeneous
of degree $-\frac{Q}{2}$ with respect to the parabolic dilations
$(x,t)\to (\delta(\lambda)(x),\lambda^2 t)$, where $\delta(\lambda)$
represent the non-isotropic dilations associated with the grading of
the Lie algebra of $\bM$, and $Q=d+2\di$ indicates the corresponding
homogeneous dimension of $\bM$. From such homogeneity of $p(x,y,t)$,
a scaling argument produces the estimate
\[
\frac{LP_t f}{P_t f} \ge -\frac{Q}{2t},
\]
which, unlike \eqref{ultracontractivity}, is best possible. However,
the estimate  \eqref{ultracontractivity} is sharp in the case of the
Laplace operator in $\mathbb{R}^d$ for which we have $\kappa = 0$.
It seems difficult to obtain sharp geometric constants by using only
the inequality \eqref{CD}. In part this is due to the fact that when
we apply Cauchy-Schwarz inequality in the Bochner formulas, we loose
a piece of information. Another difference is in the presence of the
commutator terms in \eqref{CD}. These aspects are quite different
from the Riemannian case, for which the $CD (d,R)$ inequality
\eqref{curvature_dimension} provides sharp geometric constants (see
\cite{bakry-tata}, \cite{ledoux-zurich}).
\end{remark}

\section{Stochastic completeness of the heat semigroup}

In \cite{Yau2} Yau proved that if $\bM$ is a complete Riemannian
manifold with a Ric $\ge \rho$, with $\rho \in \mathbb{R}$,  then
one has the stochastic completeness of the heat semigroup, i.e. $P_t
1=1$. Under the same assumptions, Dodziuk \cite{dodziuk} proved that
bounded solutions of the heat equation are characterized by their
initial condition.

\

In this section, we extend these theorems of Yau and Dodziuk to our
setting. Throughout this section, we assume that $(\bM,d)$ is
complete and that for every smooth function $f :\mathbb{M}
\rightarrow \mathbb{R}$ and every $\nu >0$,
\begin{equation}\label{CD2}
\Gamma_{2}(f,f)+\nu \Gamma^Z_{2}(f,f) \ge \frac{1}{d} (Lf)^2 + \left( \rho_1 -\frac{\kappa}{\nu}\right)  \Gamma (f,f) + \rho_2 \Gamma^Z (f,f),
\end{equation}
where $\rho_1 \in \mathbb{R}$, whereas $\rho_2,\kappa>0$. As in the
previous section, the idea is to study monotone increasing
functionals of the heat semigroup.

 \begin{proposition}\label{P:monofun}
 Let $f \in C^\infty_0(\mathbb M)$ and $T>0$. If $a,b:[0,T] \rightarrow \mathbb{R}$ are two positive functions such that
 \[
  a'(t)+2\rho_1 a(t)-\frac{2\kappa a(t)^2}{b(t)} \ge 0
  \]
  \[
  b'(t)+2\rho_2 a(t)\ge 0,
\]
then  the functional
 \[
 E(t)=a(t) P_t ( \Gamma (P_{T-t} f)) +b(t) P_t ( \Gamma^Z (P_{T-t} f))
 \]
 is monotone increasing.
 \end{proposition}

\begin{proof}
We begin by noting that, with $\Phi_1(t), \Phi_2(t)$ as in section
\ref{S:vi}, we have
\[
E(t)=a(t) \Phi_1(t) +b(t) \Phi_2(t).
 \]
Lemma \ref{L:derivatives} now gives
\[
E'(t)=a'(t) P_t ( \Gamma (P_{T-t} f)) +b'(t) P_t ( \Gamma^Z (P_{T-t}
f))+2a(t)P_t ( \Gamma_2 (P_{T-t} f)) +2b(t)P_t ( \Gamma_2^Z (P_{T-t}
f)).
\]
Using the curvature-dimension inequality (\ref{CD}) we now find
\begin{align*}
& a(t)P_t ( \Gamma_2 (P_{T-t} f)) +b(t)P_t ( \Gamma_2^Z (P_{T-t} f))
\\
& \ge a(t) \left( \left( \rho_1 -\frac{\kappa a(t) }{b(t)}\right)
P_t( \Gamma (P_{T-t} f) )+ \rho_2 P_t(\Gamma^Z (P_{T-t} f) )\right).
\end{align*}
Therefore,
\[
E'(t) \ge \left( a'(t)+2\rho_1 a(t)-\frac{2\kappa a(t)^2}{b(t)}
\right)P_t ( \Gamma (P_{T-t} f)) +(b'(t)+2\rho_2 a(t))P_t ( \Gamma^Z
(P_{T-t} f))  \ge 0.
\]
The desired conclusion immediately follows from this inequality.

 \end{proof}

The following energy estimate represents a basic consequence of
Proposition \ref{P:monofun}.

\begin{corollary}\label{C:expdecay}
There exists $\nu \in \mathbb{R}$ ($ \nu \le 2\min\{\rho_2, \rho_1
-\kappa\}$ will do), such that for every $f \in C^\infty_0(\mathbb
M)$, one has
\begin{equation}\label{pointwiseCaccioppoli}
\Gamma (P_t f)+\Gamma^Z (P_t f) \le  e^{-\nu t} \left( P_t
\Gamma(f)+ P_t \Gamma^Z(f)\right).
\end{equation}
As a consequence, for every $1\le p\le \infty$ one obtains
\begin{equation}\label{LpCaccioppoli}
\| \Gamma (P_t f) \|_{L^p(\bM)} \le e^{-\nu t} \left( \| \Gamma (f)
\|_{L^p(\bM)} +  \| \Gamma^Z (f) \|_{L^p(\bM)} \right), \quad t \ge
0.
\end{equation}
\end{corollary}

\begin{proof}
Let $t >0$ and consider the interval $0\le s\le t$. With the choice
\[ a(s)=b(s)=e^{-\nu s}, \quad 0\le s \le t,
\]
where $\nu \in \mathbb{R}$ satisfies $ \nu \le 2\min\{\rho_2, \rho_1
-\kappa\}$, Proposition \ref{P:monofun} gives $E(0) \leq E(s)$. This
inequality reads
\[
\Gamma (P_t f)+\Gamma^Z (P_t f) \le  e^{-\nu s} \left( P_s(
\Gamma(P_{t-s}f))+ P_s(\Gamma^Z(P_{t-s}f))\right).
\]
Letting $s\to t^-$ we obtain \eqref{pointwiseCaccioppoli}. Combining
\eqref{pointwiseCaccioppoli} with \eqref{smp} we deduce
\eqref{LpCaccioppoli}.

\end{proof}

%\begin{remark}\label{R:ricci}
%We note explicitly that in the Riemannian case of example
%\ref{E:riemannian} we have $2\min\{\rho_2, \rho_1 -\kappa\} = 0$,
%and thus Corollary \ref{C:expdecay} implies in particular for any
%$f\in C^\infty_0(\bM)$
%\[
%\| \Gamma (P_t f) \|_\infty \le  \| \Gamma (f) \|_\infty, \quad t
%\ge 0.
%\]
%In view of (v) in Theorem 1.3 in \cite{SVR}, we see that this latter
%inequality is equivalent to requiring \emph{Ric}\ $\ge 0$.
%\end{remark}

\begin{theorem}\label{T:sc}
For $t \ge 0$, one has $ P_t 1 =1$.
\end{theorem}
\begin{proof}
Let $f,g \in  C^\infty_0(\mathbb M)$, we have
\begin{align*}
\int_{\bM} (P_t f -f) g d\mu = \int_0^t \int_{\bM}\left(
\frac{\partial}{\partial s} P_s f \right) g d\mu ds= \int_0^t
\int_{\bM}\left(L P_s f \right) g d\mu ds=- \int_0^t \int_{\bM}
\Gamma ( P_s f , g) d\mu ds.
\end{align*}
By means of \eqref{LpCaccioppoli}, and Cauchy-Schwarz inequality, we
find
\begin{equation}\label{P1}
\left| \int_{\bM} (P_t f -f) g d\mu \right| \le \left(\int_0^t
e^{-\frac{\nu s}{2}} ds\right) \sqrt{ \| \Gamma (f) \|_\infty + \|
\Gamma^Z (f) \|_\infty } \int_{\bM}\Gamma (g)^{\frac{1}{2}}d\mu.
\end{equation}

We now apply \eqref{P1} with $f = h_n$, where $h_n$ are the
functions in Lemma \ref{L:exhaustion}, and then let $n\to \infty$.
Since by Beppo Levi's monotone convergence theorem we have $P_t
h_n(x)\nearrow P_t 1(x)$ for every $x\in \bM$, we see that the
left-hand side converges to $\int_{\bM} (P_t 1 -1) g d\mu$. We claim
that the right-hand side converges to zero. To see this observe
that, thanks to the assumption \eqref{LB}, we have for any $\phi \in
C^\infty(\bM)$
\[
|\nabla_R \phi|^2 = \Gamma(\phi) + \frac{1}{2} \Gamma^Z(\phi).
\]
Therefore,
\[
\sqrt{ \| \Gamma (h_n) \|_\infty + \| \Gamma^Z (h_n) \|_\infty }
\leq C ||\nabla_R h_n||_{\infty} \to 0,\ \ \text{as}\ n\to \infty.
\] We thus reach the conclusion
\[
\int_{\bM} (P_t 1 -1) g d\mu=0,\ \ \ g\in C^\infty_0(\bM).
\]
It follows that $P_t 1 =1$.

\end{proof}

Theorem \ref{T:sc} is equivalent to the uniqueness in the Cauchy
problem for initial data in $L^\infty(\bM)$.

\begin{proposition}\label{P:ucp}
The unique bounded solution of the Cauchy problem
\[
\begin{cases}
\frac{\p u}{\p t} - Lu = 0,
\\
u(x,0) = f(x),\ \ \ \  f\in L^\infty(\bM),
\end{cases}
\]
is given by $u(x,t)=P_t f(x)$.
\end{proposition}

\begin{proof}
Since  the vector fields $X_i$'s are locally Lipschitz (because smooth), for every $x \in \mathbb{M}$, the stochastic differential equation
\[
Y_t^x=x+\int_0^t X_0 (Y_s^x) ds +\sqrt{2} \sum_{i=1}^d \int_0^t X_i (Y_s^x) \circ dB^i_s,
\]
has a solution up to a stopping time $\mathbf{e}(x)$, where
$(B_t)_{t \ge 0}$ is a $d$-dimensional Brownian motion. It is seen
that
\[
(Q_t f)(x)=\mathbb{E}(1_{t < e(x)} f(Y_t^x)), \quad f\in L^2(\bM),
\]
defines a contraction semigroup on $L^2(\bM)$. By It\^o's formula
its generator on $\mathcal{C}_0 (\mathbb{M})$ is given by $L$. By
uniqueness of the heat semigroup, we actually have $Q_t f
(x)=P_tf(x)$, for $f \in \mathcal{C}_0 (\mathbb{M})$. By using the
definition of $Q_t$, we deduce that for $f \in \mathcal{C}_0
(\mathbb{M})$,
\[
P_tf(x)=\mathbb{E}(1_{t < e(x)} f(Y_t^x))
\]
By applying the previous equality with $fh_n$, and letting $n \to
+\infty$, we deduce from Beppo Levi's monotone convergence theorem
that
\[
P_tf(x)=\mathbb{E}(1_{t < e(x)} f(Y_t^x))
\]
when $f$ is a bounded and smooth function. Since $P_t 1=1$, almost surely $e(x)=+\infty$, and
\[
P_tf(x)=\mathbb{E}(f(Y_t^x)).
\]
Let now $u (t,x)$ be a bounded solution of the Cauchy problem:
\[
\begin{cases}
\frac{\p u}{\p t} - Lu = 0,
\\
u(x,0) = f(x),
\end{cases}
\]
where $f$ is a bounded and smooth function. From It\^o's formula,
for $t >0$ the process $(u(t-s, Y_s^x))_{0 \le s \le t}$ is a local
martingale and since $u$ is bounded, it is a martingale. The
expectation of $(u(t-s, Y_s^x))_{0 \le s \le t}$ is therefore
constant. From this we deduce
\[
u(t,x)=\mathbb{E}(f(Y_t^x))=P_tf(x).
\]
Finally, when $ f\in L^\infty(\bM)$, for the Cauchy problem
\[
\begin{cases}
\frac{\p u}{\p t} - Lu = 0,
\\
u(x,0) = f(x)
\end{cases}
\]
we have from the regularization property of the heat semigroup and the above uniqueness result: For every $\tau >0$,
\[
u(x,t+\tau)=\int_\bM p(x,y,t)u(y,\tau) \mu (dy).
\]
By letting $\tau \to 0$, we obtain $u(x,t)=P_t f(x)$ as desired.
\end{proof}

\section{A parabolic Harnack inequality}\label{S:harnack}

In this section we generalize the celebrated Harnack inequality in
\cite{LY} to nonnegative solutions of the heat equation $H = L -
\frac{\p}{\p t}$ on $\bM$ which are in the form $u(x,t) = P_t f(x)$,
for some $f\in C^\infty(\bM)\cap L^\infty(\bM)$. Theorem
\ref{T:harnack} below should be seen as a generalization of (i) of
Theorem 2.2 in \cite{LY}, in the case of a zero potential $q$. One
should also see the paper \cite{Cao-Yau}, where the authors deal
with subelliptic operators on a compact manifold. As we have
mentioned, they do not obtain bounds which explicitly depend solely
on the geometry of the underlying manifold.

\begin{theorem}\label{T:harnack}
Let $\bM$ be a complete sub-Riemannian manifold such that
\eqref{geombounds} holds with $\rho_1\ge 0$. Let $f\in
C^\infty(\bM)$ be such that $0\le f\le M$, and consider $u(x,t) =
P_t f(x)$. For every $(x,s), (y,t)\in \bM\times (0,\infty)$ with
$s<t$ one has with $D$ as in \eqref{D}
\begin{equation}\label{beauty}
u(x,s) \le u(y,t) \left(\frac{t}{s}\right)^{\frac{D}{2}} \exp\left(
\frac{D}{d} \frac{d(x,y)^2}{4(t-s)} \right).
\end{equation}
\end{theorem}

\begin{proof}
Let $f$ be as in the statement of the theorem, and consider $f_n =
h_n f$, where $h_n$ is the sequence in Lemma \ref{L:exhaustion}.
Then, $f_n \in C^\infty_0(\bM)$, and $0\le f_n\nearrow f$. By Beppo
Levi's monotone convergence theorem we have $u_n(x,t) = P_t
f_n(x)\nearrow u(x,t) = P_t f(x)$ for every $(x,t)\in \bM\times
(0,\infty)$. Since $Lu_n = \frac{\p u_n}{\p t}$, in terms of $u_n$
inequality \eqref{liyaupositifzero} can be reformulated as
\[ \Gamma (\ln u_n) +\frac{2 \rho_2}{3} t \Gamma^Z (\ln u_n) \le
(1+\frac{3\kappa}{2\rho_2}) \frac{\p \log u_n}{\p t}  +
\frac{d\left( 1+\frac{3\kappa}{2\rho_2}\right)^2}{2t}.
\]
In particular, this implies
\begin{equation}\label{liyaupositif2}
- (1+\frac{3\kappa}{2\rho_2}) \frac{\p \ln u_n}{\p t} \le - \Gamma
(\ln u_n)  +\frac{d\left( 1+\frac{3\kappa}{2\rho_2}\right)^2}{2t}.
\end{equation}

We now fix two points $(x,s), (y,t)\in \bM\times (0,\infty)$, with
$s<t$. Let $\gamma(\tau)$, $0\le \tau \le T$ be a subunit path such
that $\gamma(0) = y$, $\gamma(T) = x$. Consider the path in $M\times
(0,\infty)$ defined by
\[
\alpha(\tau) = \left(\gamma(\tau),t + \frac{s-t}{T}\tau\right),\ \ \
\ 0\le \tau\le T,
\]
so that $\alpha(0) = (y,t)$, $\alpha(T) = (x,s)$. We have
\begin{align*}
\ln \frac{u_n(x,s)}{u_n(y,t)}& = \int_0^T \frac{d}{d\tau} \ln
u_n(\alpha(\tau)) d\tau \\
& = \int_0^T \left[<\gamma'(\tau),\nabla_R (\ln u_n)(\alpha(\tau))>
- \frac{t-s}{T} \frac{\p \ln u_n}{\p t}(\alpha(\tau))\right] d \tau
\end{align*}
Now since $\gamma(\tau)$ is subunitary we have
\[
\left|<\gamma'(\tau),\nabla_R (\ln u_n)(\alpha(\tau))>\right| \le
\Gamma(\ln u_n(\alpha(\tau)))^{\frac{1}{2}},
\]
and applying \eqref{liyaupositif2} for any $\epsilon >0$ we find
\begin{align*}
\log \frac{u_n(x,s)}{u_n(y,t)}& \le T^{\frac{1}{2}} \left(\int_0^T
\Gamma(\ln u_n)(\alpha(\tau)) d\tau\right)^{\frac{1}{2}} -
\frac{t-s}{T} \int_0^T \frac{\p \ln u_n}{\p t}(\alpha(\tau)) d \tau
\\
& \le \frac{1}{2\epsilon} T + \frac{\epsilon}{2} \int_0^T \Gamma(\ln
u_n)(\alpha(\tau)) d\tau - \frac{t-s}{T(1+\frac{3\kappa}{2\rho_2})}
\int_0^T \Gamma(\ln u_n)(\alpha(\tau)) d\tau
\\
&   - \frac{d\left( 1+\frac{3\kappa}{2\rho_2}\right)(s-t)}{2T}
\int_0^T \frac{d\tau}{t + \frac{s-t}{T} \tau}.
\end{align*}
If we now choose $\epsilon >0$ such that
\[
\frac{\epsilon}{2} = \frac{t-s}{T(1+\frac{3\kappa}{2\rho_2})},
\]
we obtain from the latter inequality
\[
\log \frac{u_n(x,s)}{u_n(y,t)} \le
\frac{\ell_s(\gamma)^2(1+\frac{3\kappa}{2\rho_2})}{4(t-s)}  +
\frac{d\left( 1+\frac{3\kappa}{2\rho_2}\right)}{2}
\ln\left(\frac{t}{s}\right),
\]
where we have denoted by $\ell_s(\gamma)$ the subunitary length of
$\gamma$. If we now minimize over all subunitary paths joining $y$
to $x$, and we exponentiate, we obtain
\[
u_n(x,s) \le u_n(y,t) \left(\frac{t}{s}\right)^{\frac{d}{2}\left(
1+\frac{3\kappa}{2\rho_2}\right)} \exp\left(\frac{d(x,y)^2
(1+\frac{3\kappa}{2\rho_2})}{4(t-s)} \right).
\]
Letting $n\to \infty$ in this inequality we finally obtain
\eqref{beauty}.

\end{proof}

The following result represents an important consequence of Theorem
\ref{T:harnack}.

\begin{corollary}\label{C:harnackheat}
Let $p(x,y,t)$ be the heat kernel on $\bM$. For every $x,y, z\in
\bM$ and every $0<s<t<\infty$ one has
\[
p(x,y,s) \le p(x,z,t) \left(\frac{t}{s}\right)^{\frac{D}{2}}
\exp\left(\frac{D}{d} \frac{d(y,z)^2}{4(t-s)} \right).
\]
\end{corollary}

\begin{proof}
Let $\tau >0$ and $x\in \bM$ be fixed. By H\"ormander's
hypoellipticity theorem \cite{Ho} we know that $p(x,\cdot,\cdot +
\tau)\in C^\infty(\bM \times (-\tau,\infty))$. From \eqref{sgp} we
have
\[
p (x,y,s+\tau)=P_s (p(x,\cdot,\tau))(y)
\]
and
\[
p (x,z,t+\tau)=P_t (p(x,\cdot,\tau))(z)
\]
Since we cannot apply Theorem \ref{T:harnack} directly to $u(y,t) =
P_t(p(x,\cdot,\tau))(y)$, we consider as in the proof of the latter
$u_n(y,t) = P_t(h_n p(x,\cdot,\tau))(y)$, where $h_n\in
C^\infty_0(\bM)$, $0\le h_n\le 1$, and $h_n\nearrow  1$. From
\eqref{beauty} we find
\[
P_s (h_np(x,\cdot,\tau))(y) \le P_t (h_np(x,\cdot,\tau))(z)
\left(\frac{t}{s}\right)^{\frac{D}{2}} \exp\left(\frac{D}{d}
\frac{d(y,z)^2}{4(t-s)} \right)
\]
Letting $n \to \infty$, by Beppo Levi's monotone convergence theorem
we obtain
\[
p (x,y,s+\tau) \le p (x,z,t+\tau)
\left(\frac{t}{s}\right)^{\frac{D}{2}} \exp\left(\frac{D}{d}
\frac{d(y,z)^2}{4(t-s)} \right).
\]
The desired conclusion follows by letting $\tau \to 0$.

\end{proof}

\section{Off-diagonal Gaussian upper bounds for $p(x,y,t)$}\label{S:gaussianub}

Let $\bM$ be a complete sub-Riemannian manifold such that
\eqref{geombounds} holds with $\rho_1\ge 0$. Fix $x\in \bM$ and
$t>0$. Applying Corollary \ref{C:harnackheat} to $(y,t)\to p(x,y,t)$
for every $y\in B(x,\sqrt t)$ we find
\[
p(x,x,t) \le  2^{\frac{D}{2}} e^{\frac{D}{4d}}\ p(x,y,2t) =
C(d,\kappa,\rho_2) p(x,y,2t).
\]
Integration over $B(x,\sqrt t)$ gives
\[
p(x,x,t)\mu(B(x,\sqrt t)) \le C(d,\kappa,\rho_2) \int_{B(x,\sqrt
t)}p(x,y,2t)d\mu(y) \le C(d,\kappa,\rho_2),
\]
where we have used $P_t1\le 1$. This gives the on-diagonal upper
bound \begin{equation}\label{odub} p(x,x,t) \le
\frac{C(d,\kappa,\rho_2)}{\mu(B(x,\sqrt t))}.
\end{equation}

The aim of this section is to establish the following off-diagonal
upper bound for the heat kernel.

\begin{theorem}\label{T:ub}
Let $\bM$ be a complete sub-Riemannian manifold such that
\eqref{geombounds} holds with $\rho_1\ge 0$. For any $0<\epsilon <1$
there exists a constant $C(d,\kappa,\rho_2,\epsilon)>0$, which tends
to $\infty$ as $\epsilon \to 0^+$, such that for every $x,y\in \bM$
and $t>0$ one has
\[
p(x,y,t)\le \frac{C(d,\kappa,\rho_2,\epsilon)}{\mu(B(x,\sqrt
t))^{\frac{1}{2}}\mu(B(y,\sqrt t))^{\frac{1}{2}}} \exp
\left(-\frac{d(x,y)^2}{(4+\epsilon)t}\right).
\]
\end{theorem}

\begin{proof}
We suitably adapt here an idea in \cite{Cao-Yau} for the case of a
compact manifold without boundary. Since, however, we allow the
manifold $\bM$ to be  non-compact, we need to take care of this
aspect. Corollary \ref{C:expdecay} will prove crucial in this
connection. Given $T>0$, and $\alpha>0$ we fix $0<\tau \le
(1+\alpha)T$. For a function $\psi\in C^\infty_0(\bM)$, with $\psi
\ge 0$, in $\bM \times (0,\tau)$ we consider the function
\[
f(y,t) = \int_{\bM} p(y,z,t) p(x,z,T) \psi(z) d\mu(z),\ \ \ x\in
\bM.
\]
Since $f = P_t(p(x,\cdot,T)\psi)$, it satisfies the Cauchy problem
\[
\begin{cases}
Lf - f_t = 0 \ \ \ \ \text{in}\ \bM \times (0,\tau),
\\
f(z,0) = p(x,z,T)\psi(z),\ \  \ z\in \bM.
\end{cases}
\]
Notice that thanks to H\"ormander's theorem \cite{Ho} we know $y\to
p(x,y,T)$ is in $C^\infty(\bM)$, and therefore $p(x,\cdot,T)\psi\in
L^\infty(\bM)$. Moreover, \eqref{smp} gives \[
||P_t(p(x,\cdot,T)\psi)||^2_{L^2(\bM)} \le
||p(x,\cdot,T)\psi||^2_{L^2(\bM)} = \int_{\bM} p(x,z,T)^2 \psi(z)
d\mu(z) <\infty, \]
and therefore
\begin{equation}\label{L2f}
\int_0^\tau \int_{\bM} f(y,t)^2 d\mu(z)dt \le \tau \int_{\bM}
p(x,z,T)^2 \psi(z) d\mu(z)dt <\infty.
\end{equation}
Invoking \eqref{pointwiseCaccioppoli} in Corollary \ref{C:expdecay}
we have
\[
\Gamma (f)(z,t) \le  e^{-\nu t} \left( P_t
\Gamma(p(x,\cdot,T)\psi)(z) + P_t
\Gamma^Z(p(x,\cdot,T)\psi)(z)\right).
\]
This allows to conclude
\begin{equation}\label{L2gradf}
\int_0^\tau \int_{\bM} \Gamma(f)(z,t)^2 d\mu(z)dt \leq e^{|\nu|\tau}
\int_{\bM} \left\{\Gamma(p(x,\cdot,T)\psi)(z)^2 +
\Gamma^Z(p(x,\cdot,T)\psi)(z)^2\right\} d\mu(z)<\infty.
\end{equation}
We now consider a function $g\in
C^1([0,(1+\alpha)T],\text{Lip}_d(\bM))\cap L^\infty(\bM\times
(0,(1+\alpha)T))$ such that
\begin{equation}\label{ineg}
- \frac{\p g}{\p t} \ge \frac{1}{2}\Gamma(g),\ \ \text{on}\ \bM \times
(0,(1+\alpha)T).
\end{equation}

Since
\[ (L-\frac{\p}{\p t})f^2 = 2 f(L-\frac{\p}{\p t})f + 2 \Gamma(f) =
2 \Gamma(f), \] multiplying this identity by $h_n^2(y) e^{g(y,t)}$,
where $h_n$ is the sequence in Lemma \ref{L:exhaustion}, and
integrating by parts, we obtain
\begin{align*}
0 & = 2 \int_0^\tau \int_{\bM} h_n^2 e^g \Gamma(f) d\mu(y) dt -
\int_0^\tau \int_{\bM}h_n^2 e^g (L-\frac{\p}{\p t})f^2 d\mu(y) dt
\\
& = 2 \int_0^\tau \int_{\bM} h_n^2 e^g \Gamma(f) d\mu(y) dt + 4
\int_0^\tau \int_{\bM} h_n e^g f \Gamma(h_n,f) d\mu(y) dt \\
& + 2 \int_0^\tau\int_{\bM}h_n^2 e^g f \Gamma(f,g)d\mu(y) dt -
\int_0^\tau \int_{\bM} h_n e^g f^2 \frac{\p g}{\p t} d\mu(y) dt
\\
& -   \int_{\bM} h_n e^g f^2 d\mu(y)\bigg|_{t=0} + \int_{\bM} h_n
e^g f^2 d\mu(y)\bigg|_{t=\tau}
\\
& \ge 2 \int_0^\tau \int_{\bM} h_n^2 e^g \sum_{i=1}^d \left(X_i f +
\frac{f}{2} X_i g\right)^2 d\mu(y) dt + 4 \int_0^\tau \int_{\bM} h_n
e^g f \Gamma(h_n,f) d\mu(y) dt
\\
& + \int_{\bM} h_n e^g f^2 d\mu(y)\bigg|_{t=\tau} -   \int_{\bM} h_n
e^g f^2 d\mu(y)\bigg|_{t=0},
\end{align*}
where in the last inequality we have made use of the assumption
\eqref{ineg} on $g$. From this we conclude
\[
\int_{\bM} h_n e^g f^2 d\mu(y)\bigg|_{t=\tau} \le \int_{\bM} h_n e^g
f^2 d\mu(y)\bigg|_{t=0} - 4 \int_0^\tau \int_{\bM} h_n e^g f
\Gamma(h_n,f) d\mu(y) dt.
\]
We now claim that
\[
\underset{n\to \infty}{\lim} \int_0^\tau \int_{\bM} h_n e^g f
\Gamma(h_n,f) d\mu(y) dt = 0.
\]
To see this we apply Cauchy-Schwarz inequality which gives
\begin{align*}
& \left|\int_0^\tau \int_{\bM} h_n e^g f \Gamma(h_n,f) d\mu(y)
dt\right|\le \left(\int_0^\tau \int_{\bM} h_n^2 e^g f^2 \Gamma(h_n)
d\mu(y) dt\right)^{\frac{1}{2}} \left(\int_0^\tau \int_{\bM} e^g
\Gamma(f) d\mu(y) dt\right)^{\frac{1}{2}}
\\
& \le \left(\int_0^\tau \int_{\bM} e^g f^2 \Gamma(h_n) d\mu(y)
dt\right)^{\frac{1}{2}} \left(\int_0^\tau \int_{\bM} e^g \Gamma(f)
d\mu(y) dt\right)^{\frac{1}{2}} \to 0,
\end{align*}
as $n\to \infty$, thanks to \eqref{L2f}, \eqref{L2gradf}. With the
claim in hands we now let $n\to \infty$ in the above inequality
obtaining
\begin{equation}\label{CYpsi}
\int_{\bM}  e^{g(y,\tau)} f^2(y,\tau) d\mu(y) \le \int_{\bM}
e^{g(y,0)} f^2(y,0) d\mu(y).
\end{equation}
At this point we fix $x\in \bM$ and for $0<t\le\tau$ consider the
indicator function $\mathbf 1_{B(x,\sqrt t)}$ of the ball $B(x,\sqrt
t)$. Let $\psi_k\in C^\infty_0(\bM)$, $\psi_k \ge 0$, be a sequence
such that $\psi_k \to \mathbf 1_{B(x,\sqrt t)}$ in $L^2(\bM)$, with
supp$\ \psi_k\subset B(x,100\sqrt t)$. Slightly abusing the notation
we now set \[ f(y,s) = P_s(p(x,\cdot,T)\mathbf{1}_{B(x,\sqrt t)})(y)
= \int_{B(x,\sqrt t)} p(y,z,s) p(x,z,T) d\mu(z). \] Thanks to the
symmetry of $p(x,y,s) = p(y,x,s)$, we have
\begin{equation}\label{psquare}
 f(x,T) = \int_{B(x,\sqrt t)} p(x,z,T)^2 d\mu(z).
\end{equation}

Applying \eqref{CYpsi} to $f_k(y,s) = P_s(p(x,\cdot,T)\psi_k)(y)$,
we find
\begin{equation}\label{CYpsik}
\int_{\bM}  e^{g(y,\tau)} f^2_k(y,\tau) d\mu(y) \le \int_{\bM}
e^{g(y,0)} f^2_k(y,0) d\mu(y).
\end{equation}
At this point we observe that as $k\to \infty$
\begin{align*}
& \left|\int_{\bM}  e^{g(y,\tau)} f^2_k(y,\tau) d\mu(y) - \int_{\bM}
e^{g(y,\tau)} f^2(y,\tau) d\mu(y)\right|
\\
& \le 2 ||e^{g(\cdot,\tau)}||_{L^\infty(\bM)}
||p(x,\cdot,T)||_{L^2(\bM)} ||p(x,\cdot,\tau)||_{L^\infty(B(x,110
\sqrt t))} ||\psi_k - \mathbf 1_{B(x,\sqrt t)}||_{L^2(\bM)} \to 0.
\end{align*}
By similar considerations we find
\begin{align*}
& \left|\int_{\bM}  e^{g(y,0)} f^2_k(y,0) d\mu(y) - \int_{\bM}
e^{g(y,0)} f^2(y,0) d\mu(y)\right|
\\
& \le 2 ||e^{g(\cdot,0)}||_{L^\infty(\bM)}
||p(x,\cdot,T)||_{L^\infty(B(x,110 \sqrt t))} ||\psi_k - \mathbf
1_{B(x,\sqrt t)}||_{L^2(\bM)} \to 0.
\end{align*}
Letting $k\to \infty$ in \eqref{CYpsik} we thus conclude that the
same inequality holds with $f_k$ replaced by $f(y,s) =
P_s(p(x,\cdot,T)1_{B(x,\sqrt t)})(y)$. This implies in particular
the basic estimate
\begin{align}\label{be}
& \underset{z\in B(x,\sqrt t)}{\inf}\ e^{g(z,\tau)} \int_{B(x,\sqrt
t)} f^2(z,\tau) d\mu(z)
\\
& \le \int_{B(x,\sqrt t)} e^{g(z,\tau)} f^2(z,\tau) d\mu(z) \le
\int_{\bM} e^{g(z,\tau)} f^2(z,\tau) d\mu(z) \notag\\
& \le \int_{\bM} e^{g(z,0)} f^2(z,0) d\mu(z) = \int_{B(y,\sqrt t)}
e^{g(z,0)} p(x,z,T)^2 d\mu(z) \notag\\
& \le \underset{z\in B(y,\sqrt t)}{\sup}\ e^{g(z,0)} \int_{B(y,\sqrt
t)} p(x,z,T)^2 d\mu(z). \notag
\end{align}

At this point we choose in \eqref{be}
\[ g(y,t) = g_x(y,t) = -
\frac{d(x,y)^2}{2((1+2\alpha) T - t)}.
\]
Using the fact that $\Gamma(d)\le 1$, one can easily check that
\eqref{ineg} is satisfied for this $g$. Taking into account that
\[
\underset{z\in B(x,\sqrt t)}{\inf}\ e^{g_x(z,\tau)} = \underset{z\in
B(x,\sqrt t)}{\inf}\ e^{-\frac{d(x,z)^2}{2((1+2\alpha)T- \tau)}} \ge
e^{\frac{-t}{2((1+2\alpha)T- \tau)}},
\]
if we now choose $\tau = (1+\alpha)T$, then from the previous
inequality and from \eqref{psquare} we conclude that
\begin{equation}\label{lemmaub}
\int_{B(x,\sqrt t)} f^2(z,(1+\alpha)T) d\mu(z) \le
\left(\underset{z\in B(y,\sqrt t)}{\sup}\
e^{-\frac{d(x,z)^2}{2(1+2\alpha)T} + \frac{t}{2\alpha T}}\right)
\int_{B(y,\sqrt t)} p(x,z,T)^2 d\mu(z).
\end{equation}
We now apply Theorem \ref{T:harnack} which gives for every $z\in
B(x,\sqrt t)$
\[
f(x,T)^2 \le f(z,(1+\alpha)T)^2
(1+\alpha)^{d(1+\frac{3\kappa}{2\rho_2})}
e^{\frac{t(1+\frac{3\kappa}{2\rho_2})}{2\alpha T}}.
\]
Integrating this inequality on $B(x,\sqrt t)$ we find
\[
\left(\int_{B(y,\sqrt t)} p(x,z,T)^2 d\mu(z)\right)^2 = f(x,T)^2 \le
\frac{(1+\alpha)^{d(1+\frac{3\kappa}{2\rho_2})}
e^{\frac{t(1+\frac{3\kappa}{2\rho_2})}{2\alpha T}}}{\mu(B(x,\sqrt
t)) } \int_{B(x,\sqrt t)} f^2(z,(1+\alpha)T) d\mu(z).
\]
If we now use \eqref{lemmaub} in the last inequality we obtain
\begin{align*}
& \int_{B(y,\sqrt t)} p(x,z,T)^2 d\mu(z) \le
\frac{(1+\alpha)^{d(1+\frac{3\kappa}{2\rho_2})}
e^{\frac{t(1+\frac{3\kappa}{2\rho_2})}{2\alpha T}}}{\mu(B(x,\sqrt
t))} \left(\underset{z\in B(y,\sqrt t)}{\sup}\
e^{-\frac{d(x,z)^2}{2(1+2\alpha)T} + \frac{t}{2\alpha T}}\right).
\end{align*}
Choosing $T = (1+\alpha)t$ in this inequality we find
\begin{align}\label{ub2}
& \int_{B(y,\sqrt t)} p(x,z,(1+\alpha)t)^2 d\mu(z) \le
\frac{(1+\alpha)^{d(1+\frac{3\kappa}{2\rho_2})}
e^{\frac{(1+\frac{3\kappa}{2\rho_2})}{2\alpha(1+\alpha)}+
\frac{1}{2\alpha (1+\alpha)}}}{\mu(B(x,\sqrt t))}
\left(\underset{z\in B(y,\sqrt t)}{\sup}\
e^{-\frac{d(x,z)^2}{2(1+2\alpha)(1+\alpha)t} + \frac{1}{2\alpha
(1+\alpha)}}\right).
\end{align}
We now apply Corollary \ref{C:harnackheat} obtaining for every $z\in
B(y,\sqrt t)$
\[
p(x,y,t)^2 \le p(x,z,(1+\alpha)t)^2 (1+\alpha)^{d\left(
1+\frac{3\kappa}{2\rho_2}\right)} \exp\left(\frac{
1+\frac{3\kappa}{2\rho_2}}{2\alpha } \right).
\]
Integrating this inequality in $z\in B(y,\sqrt t)$, we have
\[
\mu(B(y,\sqrt t)) p(x,y,t)^2 \le (1+\alpha)^{d\left(
1+\frac{3\kappa}{2\rho_2}\right)} e^{\frac{
1+\frac{3\kappa}{2\rho_2}}{2\alpha }} \int_{B(y,\sqrt t)}
p(x,z,(1+\alpha)t)^2 d\mu(z).
\]
Combining this inequality with \eqref{ub2} we conclude
\[
p(x,y,t) \le \frac{(1+\alpha)^{d(1+\frac{3\kappa}{2\rho_2})}
e^{\frac{(1+\frac{3\kappa}{2\rho_2})(2+\alpha)}{4\alpha(1+\alpha)}+
\frac{3}{4\alpha (1+\alpha)}}}{\mu(B(x,\sqrt
t))^{\frac{1}{2}}\mu(B(y,\sqrt
t))^{\frac{1}{2}}}\left(\underset{z\in B(y,\sqrt t)}{\sup}\
e^{-\frac{d(x,z)^2}{2(1+2\alpha)(1+\alpha)t}}\right).
\]
If now $x\in B(y,\sqrt t)$, then
\[
d(x,z)^2 \ge (d(x,y) - \sqrt t)^2 > d(x,y)^2 - t,
\]
and therefore
\[
\underset{z\in B(y,\sqrt t)}{\sup}\
e^{-\frac{d(x,z)^2}{2(1+2\alpha)(1+\alpha)t}} \le
e^{\frac{1}{2(1+2\alpha)(1+\alpha)}}
e^{-\frac{d(x,y)^2}{2(1+2\alpha)(1+\alpha)t}}.
\]
If instead $x\not\in B(y,\sqrt t)$, then for every $\delta >0$ we
have
\[
d(x,z)^2 \ge (1-\delta) d(x,y)^2  - (1+ \delta^{-1}) t
\]
Choosing $\delta = \alpha/(\alpha+1)$ we find
\[
d(x,z)^2 \ge \frac{d(x,y)^2}{1+\alpha}  - (2 + \alpha^{-1}) t,
\]
and therefore
\[
\underset{z\in B(y,\sqrt t)}{\sup}\
e^{-\frac{d(x,z)^2}{2(1+2\alpha)(1+\alpha)t}} \le
e^{-\frac{d(x,y)^2}{2(1+2\alpha)(1+\alpha)^2 t} + \frac{2 +
\alpha^{-1}}{2(1+2\alpha)(1+\alpha)}}
\]
For any $\epsilon >0$ we now choose $\alpha>0$ such that
$2(1+2\alpha)(1+\alpha)^2 = 4+\epsilon$ to reach the desired
conclusion.

\end{proof}

\section{A generalization of Yau's Liouville theorem}

In his seminal 1975 paper \cite{Yau}, by using gradient estimates,
Yau proved his celebrated Liouville theorem that there exists no
non-constant positive harmonic function on a complete Riemannian
manifold with non-negative Ricci curvature. The aim of this section
is to extend Yau's theorem to the sub-Riemannian setting of this
paper. An interesting point to keep in mind here is that, even in
the Riemannian setting, our approach gives a new proof of Yau's
theorem which is not based on delicate tools from Riemann geometry
such as the Hessian and the Laplacian comparison theorems for the
geodesic distance. However, due to the nature of our proof at the
moment we are only able to deal with harmonic functions bounded from
two sides, whereas in \cite{Yau} the author is able to treat
functions satisfying a one-side bound.

\

In what follows we assume that the curvature-dimension inequality
\eqref{CD} hold with $\rho_1 = 0$, $\rho_2
>0$ and $\kappa >0$. We also assume that the metric space $(\bM
,d)$, or equivalently $(\bM, d_R)$, is complete. As we mentioned it
before, see Remark \ref{R:ricci}, in the Riemannian case of Example
\ref{E:riemannian} our assumptions are in fact equivalent to
assuming that $\bM$ is a complete Riemannian manifold with Ric$\ge
0$. We begin with a Harnack type inequality for the operator $L$.

\begin{theorem}\label{T:harnackL}
Let $\bM$ be a complete sub-Riemannian manifold and assume \eqref{CD}
with $\rho_1=0$. Let $0\le f\le M$ be a harmonic function on $\bM$,
then there exists a constant $C = C(d,\rho_2,\kappa)>0$ such that
for any $x_0\in \bM$ and any $r>0$ one has
\[
\underset{B(x_0,r)}{\sup} f  \le  C \underset{B(x_0,r)}{\inf} f.
\]
\end{theorem}

\begin{proof}
Let $f$ be as in the statement of the theorem. By H\"ormander's
theorem \cite{Ho} we know that $f\in C^\infty(\bM)$. Applying
Theorem \ref{T:harnack} to the function $u(x,t) = P_t f(x)$, we
obtain for $x,y\in B(x_0,r)$
\[
P_s f(x) \le P_t f(y) \left(\frac{t}{s}\right)^{\frac{D}{2}}
\exp\left(\frac{D r^2}{d(t-s)}\right),\ \ \ 0<s<t<\infty.
\]
At this point we observe that, thanks to the assumption $Lf = 0$,
the functions $u(x,t) = P_t f(x)$ and $v(x,t) = f(x)$ solve the same
Cauchy problem on $\bM$. By Proposition \ref{P:ucp} we must have
$P_t f(x) = f(x)$ for every $x\in \bM$ and every $t>0$. Therefore,
taking $s = r^2, t = 2r^2$, the latter inequality gives
\[
f(x) \le  \left(\sqrt 2 e^{\frac{1}{d}}\right)^D\ f(y), \ \ \ x,y\in
B(x_0,r).
\]
\end{proof}

\begin{corollary}[of Cauchy-Liouville type]\label{C:liouville}
Assume \eqref{CD} with $\rho_1=0$, then there exist no bounded
solutions to $Lf=0$ on $\bM$, other than the constants.
\end{corollary}

\begin{proof}
Suppose $a\le f\le b$ on $\bM$. Consider the function $g = f -
\underset{\bM}{\inf}\ f$. Clearly, $0\le g \le M = b-a$. If we apply
Theorem \ref{T:harnackL} to $g$ we find for any $x_0\in \bM$ and
$r>0$
\[
\underset{B(x_0,r)}{\sup} g  \le  C \underset{B(x_0,r)}{\inf} g.
\]
Letting $r\to \infty$ we reach the conclusion $\underset{\bM}{\sup}\
f = \underset{\bM}{\inf}\ f$, hence $f\equiv $ const.

\end{proof}

\section{Volume growth and Isoperimetric inequality when $\rho_1=0$}\label{S:Isoperimetry}

Throughout this section, we shall assume that $(\bM,d)$ is complete non compact
and that \eqref{geombounds} holds with $\rho_1\ge 0$.

\subsection{Volume growth}

We first derive a basic and straightforward consequence of the
parabolic Harnack inequality on the volume growth of metric balls.

\begin{proposition}
For every $x \in \bM$ and every $R_0 >0$ there is a constant
$C(d,\kappa,\rho_2)>0$ such that, with $D$ as in \eqref{D},
\[
\mu \left( B(x,R)\right) \le \frac{C(d,\kappa,\rho_2)}{R_0^D
p(x,x,R_0^2)} R^D, \quad\ \  R \ge R_0.
\]
\end{proposition}

\begin{proof}
Let $t>\tau >0$. From the Harnack inequality of Corollary
\ref{C:harnackheat}, we have
\[
p(x,x,t) \ge p(x,x,\tau) \left( \frac{\tau}{t} \right)^{\frac{D}{2}}
\]
On the other hand, the inequality (\ref{odub}) gives
\[
p(x,x,t) \le \frac{C(d,\kappa,\rho_2)}{\mu  \left(
B(x,\sqrt{t})\right) }.
\]
This implies the desired conclusion.

\end{proof}

\subsection{Isoperimetric inequality}

In \cite{GN} it was proved that in a Carnot-Carath\'eodory space
$(X,\mu,d)$ the doubling condition \[ \mu(B(x,2r)) \le C_1
\mu(B(x,r)), \ \ \ x\in X, r>0, \] for the volume of the metric
balls combined with a weak Poincar\'e inequality suffice to
establish the following basic relative isoperimetric inequality
\begin{align}\label{iso}
\min\left\{\mu(E\cap B(x,r)),\mu((X\setminus E)\cap
B(x,r))\right\}^{\frac{D-1}{D}} \le C_{\text{iso}}
\left(\frac{r^D}{\mu(B(x,r))}\right)^{\frac{1}{D}} P(E,B(x,r)),
\end{align}
where $P(E,B(x,r))$ represents a generalization of De Giorgi's
variational notion of perimeter, and $E\subset X$ is any set of
locally finite perimeter. In this inequality the number $D = \log_2
C_1$, where $C_1$ is the doubling constant, and $C_{\text{iso}}$ is
a constant which depends only on $C_1$ and on the constant in the
Poincar\'e inequality. If in addition the space $X$ satisfies the
maximum volume growth
\begin{equation}\label{maxvolgrowth}
\mu(B(x,r))\ge C_2 r^D,\ \ \ x\in \bM, r>0,
\end{equation}
then \eqref{iso} gives
\begin{align}\label{isovolgrowth}
\min\left\{\mu(E\cap B(x,r)),\mu((X\setminus E)\cap
B(x,r))\right\}^{\frac{D-1}{D}} \le C^*_{\text{iso}} P(E,B(x,r)),
\end{align}
where $C^*_{\text{iso}} = C_2^{-1/D} C_{\text{iso}}$.

When $X=M$ is a Riemannian manifold with Ric $\ge 0$, then the
doubling condition with $C_1 = 2^D$, $D=$ dim\ $\bM$, follows from
Bishop-Gromov comparison theorem (see \cite{Chavel}, Proposition 3.3
and Theorem 3.10), whereas the Poincar\'e inequality was proved by
Buser, see \cite{Buser}. As a consequence, one obtains the relative
isoperimetric inequality \eqref{iso} in this setting. When $\bM$
satisfies the maximum volume growth \eqref{maxvolgrowth}, one also
obtains from \eqref{iso} the global isoperimetric inequality
\begin{equation}\label{isoglobal}
\mu(E)^{\frac{D-1}{D}} \le C_{\text{iso}} P(E,\bM),
\end{equation}
for any measurable set of locally finite perimeter $E\subset \bM$.

\

In this subsection we investigate the sub-Riemannian counterpart of
the isoperimetric estimate \eqref{isoglobal} under the assumption
that $\rho_1 = 0$. Here, the main obstacle is precisely the a priori
lack of a global doubling condition and of a Poincar\'e inequality,
and therefore we cannot rely on the above cited results from
\cite{GN}. Instead, using our Li-Yau type estimate \eqref{LY} we
adapt some beautiful ideas of Varopoulos and Ledoux to provide a
characterization of those sub-Riemannian manifolds which support an
inequality such as \eqref{isoglobal}. We stress that, due to the
nature of our approach, we obtain a lower bound on the dimension $D$
in \eqref{isoglobal}, but in the case of a graded nilpotent Lie
group our $D$ is not optimal, see Remark \ref{R:ultrac}.

In what follows, given an open set $\Om \subset \bM$ we will
indicate with
\[
\mathcal F(\Om) = \{\phi\in C^1_0(\Om,\mathcal H)\mid
||\phi||_\infty \le 1\}.
\]
Here, for $\phi = \sum_{i=1}^d \phi_i X_i$, we have let
$||\phi||_\infty = \underset{\Om}{\sup} \sqrt{\sum_{i=1}^d
\phi_i^2}$. Following \cite{CDG}, given a function $f\in
L^1_{loc}(\Om)$ we define the horizontal total variation of $f$ in
$\Om$ as
\[
\text{Var}_{\Ho}(f;\Om) = \underset{\phi\in \F(\Om)}{\sup} \int_\Om
f \left(\sum_{i=1}^d X_i \phi_i\right) d\mu.
\]
The space \[ BV_\Ho(\Om) = \{f\in L^1(\Om)\mid
\text{Var}_\Ho(f;\Om)<\infty\},
\]
endowed with the norm
\[
||f||_{BV_\Ho(\Om)} = ||f||_{L^1(\bM)} + \text{Var}_\Ho(f;\Om),
\]
is a Banach space. It is well-known that $W^{1,1}_\Ho(\Om) = \{f\in
L^1(\Om)\mid X_i f\in L^1(\Om),i=1,...,d\}$ is a strict subspace of
$BV_\Ho(\Om)$. It is important to note that when $f\in
W^{1,1}_\Ho(\Om)$, then $f\in BV_\Ho(\Om)$, and one has in fact
\[
\text{Var}_\Ho(f;\Om) = ||\sqrt{\Gamma(f)}||_{L^1(\Om)}.
\]
Given a measurable set $E\subset \bM$ we say that it has finite
horizontal perimeter in $\Om$ if $\mathbf 1_E\in BV_\Ho(\Om)$. In
such case the horizontal perimeter of $E$ relative to $\Om$ is by
definition
\[
P_\Ho(E;\Om) = \text{Var}_\Ho(\mathbf 1_E;\Om).
\]
We say that a measurable set $E\subset \bM$ is a Caccioppoli set if
$P_\Ho(E;\Om)<\infty$ for any $\Om \Subset \bM$. We will need the
following approximation result, see Theorem 1.14 in \cite{GN}.

\begin{proposition}\label{P:ag}
Let $f\in BV_\Ho(\Om)$, then there exists a sequence $\{f_n\}_{n\in
\mathbb N}$ of functions in $C^\infty(\Om)$ such that:
\begin{itemize}
\item[(i)] $||f_n - f||_{L^1(\Om)} \to 0$;
\item[(ii)] $\int_\Om \sqrt{\Gamma(f_n)} d\mu \to
\text{Var}_\Ho(f;\Om)$.
\end{itemize}
If $\Om = \bM$, then the sequence $\{f_n\}_{n\in \mathbb N}$ can be
taken in $C^\infty_0(\bM)$.
\end{proposition}

 Our intent is
to establish the following result.

\begin{theorem}[Isoperimetric inequality]\label{T:iso} Suppose that $\bM$ is not compact in
the metric topology and that there exists $D>1$ such that
\begin{equation}\label{volgrowth}
\mu(B(x,r)) \ge C_1 r^D.
\end{equation}
Then, there is a constant $C_{\emph{iso}} =
C_{\emph{iso}}(d,\rho_2,\kappa, C_1,D)>0$, such that for every
Caccioppoli set $E\subset \bM$
\[
\mu(E)^{\frac{D-1}{D}} \le C_{\emph{iso}} P_\Ho(E,\bM).
\]
\end{theorem}

The essence of the proof of Theorem \ref{T:iso} is contained in the
following result.

\begin{proposition}\label{P:char}
Let $D>1$. Let us assume that $\bM$ is not compact in the metric
topology, then the following assertions  are equivalent:
\begin{itemize}
\item[(1)] There exists a constant $C_1 >0$ such that for every $x \in \bM$, $r \ge 0$,
\[
\mu (B(x,r)) \ge C_1 r^D.
\]
\item[(2)] There exists a constant $C_2>0$ such that for $x \in \bM$, $t>0$,
\[
p(x,x,t) \le \frac{C_2}{t^{\frac{D}{2}}}.
\]
\item[(3)] There exists a constant $C_3 >0$ such that for every Caccioppoli set $E\subset \bM$ one has
\[
\mu(E)^{\frac{D-1}{D}} \le C_3 P_\Ho(E;\bM).
\]
\item[(4)] With the same constant $C_3>0$ as in (3), for every $f \in
BV_\Ho(\bM)$ one has
\[
\left( \int_\bM |f|^{\frac{D}{D-1}} d\mu\right)^{\frac{D-1}{D}}  \le
C_3 \emph{Var}_\Ho(f;\bM).
\]
\end{itemize}
\end{proposition}

\begin{proof}
That (1) $\rightarrow$ (2) follows immediately from the Gaussian
bound of Theorem \ref{T:ub} (observe that we may take
$C_2=\frac{1}{C_1}$).

The proof that (2) $\rightarrow$ (3) is not straightforward. First
we note that \eqref{geombounds} implies the Li-Yau type estimate
\eqref{liyaupositifzero}. This enables us to adapt some beautiful
ideas of Varopoulos (see \cite{Varopoulos2}, pp.256-58) and Ledoux
(see pp. 22 in \cite{ledoux-bourbaki}, see also Theorem 8.4 in
\cite{ledoux-stflour}). Let $f \in C_0(\bM)$ with $f\ge 0$. By
\eqref{liyaupositifzero} we obtain
\begin{equation}\label{LY}
\Gamma (P_t f) - (1+\frac{3\kappa}{2\rho_2})P_tf  \frac{\p P_tf}{\p
t} \le \frac{d\left(
1+\frac{3\kappa}{2\rho_2}\right)^2}{2t}(P_tf)^2.
\end{equation}
This gives in particular, with $\nu = d\left(
1+\frac{3\kappa}{2\rho_2}\right)$,
\begin{equation}\label{LYpt}
\left(\frac{\p P_tf}{\p t}\right)^- \le \frac{\nu}{2t} P_tf,
\end{equation}
where we have denoted $a^+ = \sup\{a,0\}$, $a^- = \sup\{-a,0\}$. It
follows that for every $0<T_1<T_2<\infty$
\begin{equation}\label{pminus}
\int_{T_1}^{T_2} \int_\bM \left(\frac{\p P_tf}{\p t}\right)^- d\mu
dt \le \frac{\nu}{2} \ln(T_2/T_1) ||f||_{L^1(\bM)}< \infty.
\end{equation}
Using Tonelli's theorem we now have for any $t>0$
\[
\int_\bM P_t f(x) d\mu(x) = \int_\bM f(y) \int_\bM p(x,y,t)d\mu(x)
d\mu(y) = \int_\bM f(y) d\mu(y),
\]
where we have used Theorem \ref{T:sc}. This gives for any $t,h>0$
\begin{align*}
0 & = \int_\bM P_{t+h}f(x) d\mu(x) - \int_\bM P_{t}f(x) d\mu(x) =
\int_\bM \int_t^{t+h} \frac{\p P_sf}{\p s}(x) ds d\mu(x)
\\
& = \int_t^{t+h} \int_\bM  \frac{\p P_sf}{\p s}(x) d\mu(x) ds,
\end{align*}
where the exchange of order of integration is justified by Fubini's
theorem, which we can apply in view of \eqref{pminus}. The latter
equation implies
\[
\int_t^{t+h} \int_\bM  \left(\frac{\p P_sf}{\p s}(x)\right)^+
d\mu(x) ds = \int_t^{t+h} \int_\bM  \left(\frac{\p P_sf}{\p
s}(x)\right)^- d\mu(x) ds,\ \ \ t, h>0,
\]
and therefore, we have
\begin{align*}
\frac{1}{h}\int_t^{t+h} \int_\bM  \left|\frac{\p P_sf}{\p
s}(x)\right| d\mu(x) ds & = \frac{2}{h} \int_t^{t+h} \int_\bM
\left(\frac{\p P_sf}{\p s}(x)\right)^- d\mu(x) ds
\\
& \le \frac{\nu}{h} \left(\int_t^{t+h} \frac{ds}{s}\right)
||f||_{L^1(\bM)}.
\end{align*}
Letting $h\to 0^+$ we finally conclude
\[
||\frac{\p P_tf}{\p t}||_{L^1(\bM)} \le \frac{\nu}{t}
||f||_{L^1(\bM)},\ \ \ \ t>0.
\]
By duality, we deduce that for every $f \in C_0(\bM)$, $f\ge 0$,
\[
 \|\frac{\p P_tf}{\p t} \|_{L^\infty(\bM)} \le \frac{\nu}{t} \| f\|_{L^\infty(\bM)}.
 \]
Once we have this crucial information we can return to \eqref{LY}
and infer
\begin{align*}
\Gamma (P_t f)  \le \frac{1}{t} \frac{3\nu^2}{2d} \|
f\|^2_{L^\infty(\bM)},\ \ \ \ t>0.
\end{align*}
Thus,
\[
\| \sqrt{\Gamma (P_t f)} \|_{L^\infty(\bM)} \le
\nu\sqrt{\frac{3d}{2t} }\| f\|_{L^\infty(\bM)}.
\]
Applying this inequality to $g \in C_0^\infty(\bM)$, with $g\ge 0$
and $||g||_{L^\infty(\bM)}\le 1$, if $f \in C_0^1(\bM)$ we have
\begin{align*}
\int_\bM g(f-P_tf) d\mu & = \int_0^t \int_\bM g \frac{\p P_sf}{\p s}
d\mu ds = \int_0^t \int_\bM g L P_sf d\mu ds =  \int_0^t \int_\bM L
g P_sf d\mu ds
\\
& = \int_0^t \int_\bM P_sLg f d\mu ds = \int_0^t \int_\bM L P_sg f
d\mu ds  = - \int_0^t \int_\bM \Gamma(P_sg,f) d\mu ds
\\
& \le \int_0^t  \| \sqrt{\Gamma(P_sg)} \|_{L^\infty(\bM)}\int_\bM
\sqrt{\Gamma(f)} d\mu ds  \le \sqrt{6d} \nu\ \sqrt{t}
\int_\bM\sqrt{\Gamma(f)} d\mu.
\end{align*}
We thus obtain the following basic inequality: for $f \in
C_0^1(\bM)$,
\begin{align}\label{PoincareP_t}
\|P_tf - f\|_{L^1(\bM)} \le \sqrt{6d}\ \nu\ \sqrt{t}\ \|
\sqrt{\Gamma(f)} \|_{L^1(\bM)},\ \ \ t>0.
\end{align}
Suppose now that $E\subset \bM$ is a bounded Caccioppoli set. But
then, $\mathbf 1_E\in BV_\Ho(\Om)$, for any bounded open set
$\Om\supset E$. It is easy to see (see e.g. the proof of Lemma 2.5
in \cite{DGNisoper}) that $\text{Var}_\Ho(\mathbf 1_E;\Om) =
\text{Var}_\Ho(\mathbf 1_E;\bM)$, and therefore $\mathbf 1_E\in
BV_\Ho(\bM)$. By Proposition \ref{P:ag} there exists a sequence
$\{f_n\}_{n\in \mathbb N}$ in $C^\infty_0(\bM)$ satisfying (i) and
(ii). Applying \eqref{PoincareP_t} to $f_n$ we obtain \[ \|P_tf_n -
f_n\|_{L^1(\bM)} \le \sqrt{6d} \nu\ \sqrt{t}\ \| \sqrt{\Gamma(f_n)}
\|_{L^1(\bM)} = \sqrt{6d} \nu\ \sqrt{t}\ Var_\Ho(f_n,\bM),\ \ \ n\in
\mathbb N.
\]
Letting $n\to \infty$ in this inequality, we conclude
\[ \|P_t \mathbf 1_E -
\mathbf 1_E\|_{L^1(\bM)} \le \sqrt{6d} \nu\ \sqrt{t}\
Var_\Ho(\mathbf 1_E,\bM) = \sqrt{6d} \nu\ \sqrt{t}\ P_\Ho(E;\bM),\ \
\ \ t>0.
\]
Observe now that, using $P_t 1 = 1$, we have
\[
||P_t \mathbf 1_E - \mathbf 1_E||_{L^1(\bM)}  = 2\left(\mu(E) -
\int_E P_t \mathbf 1_E d\mu\right).
\]
On the other hand,
\[
\int_E  P_t \mathbf 1_E d\mu  = \int_\bM \left(P_{t/2}\mathbf
1_E\right)^2 d\mu.
\]
We thus obtain
\[
||P_t \mathbf 1_E - \mathbf 1_E||_{L^1(\bM)} = 2 \left(\mu(E) -
\int_\bM \left(P_{t/2}\mathbf 1_E\right)^2 d\mu\right).
\]
We now observe that \eqref{odub} and the assumption (1) imply
\[
p(x,x,t) \le \frac{C(d,\kappa,\rho_2)}{\mu(B(x,\sqrt t))} \le
\frac{C_4}{t^{D/2}},\ \ \ x\in \bM, t>0,
\]
where $C_4 = C_1^{-1} C(d,\kappa,\rho_2)$. This gives
\begin{align*}
\int_\bM (P_{t/2} \mathbf 1_E)^2 d\mu & \le \left(\int_E
\left(\int_\bM p(x,y,t/2)^2
d\mu(y)\right)^{\frac{1}{2}}d\mu(x)\right)^2
\\
& = \left(\int_E p(x,x,t)^{\frac{1}{2}}d\mu(x)\right)^2 \le
\frac{C_4}{t^{D/2}} \mu(E)^2.
\end{align*}
Combining these equations we reach the conclusion
\[
\mu(E)  \le \frac{\sqrt{6d}}{2} \nu\ \sqrt{t}\ P_\Ho(E;\bM) +
\frac{C_4}{t^{D/2}} \mu(E)^2,\ \ \ \ t>0.
\]
Now the absolute minimum of the function $g(t) = A t^\alpha + B
t^{-\beta}$, $t>0$, where $A, B, \alpha, \beta>0$, is given by
\[
g_{\min} =
\left[\left(\frac{\alpha}{\beta}\right)^{\frac{\beta}{\alpha +
\beta}} + \left(\frac{\beta}{\alpha}\right)^{\frac{\alpha}{\alpha +
\beta}}\right] A^{\frac{\beta}{\alpha + \beta}}
B^{\frac{\alpha}{\alpha + \beta}}
\]
Applying this observation with $\alpha = \frac{1}{2}, \beta =
\frac{D}{2}$, we conclude
\[
\mu(E)^{\frac{D-1}{D}} \le C_3 P_\Ho(E,\bM),
\]
with \[ C_3 = (1+D)^{\frac{D+1}{D}}\left(\frac{\sqrt{6d}\
\nu}{2D}\right) C_4^{\frac{1}{D}}. \] The latter inequality proves
(3).

The proof that (3) is equivalent to (4) follows the classical ideas
of Fleming-Rishel and Maz'ya, and it is based on a generalization of
Federer's co-area formula for the space $BV_\Ho$, see for instance
\cite{GN}.

\

Finally, we show that $(4) \rightarrow (1)$. We adapt an idea in
\cite{saloff1} (see Theorem 3.1.5 on p. 58). In what follows we let
$\nu = D/(D-1)$. Let $p,q\in (0,\infty)$ and $0<\theta\le 1$ be such
that
\[
\frac{1}{p} = \frac{\theta}{\nu} + \frac{1-\theta}{q}.
\]
H\"older inequality, combined with assumption (4), gives for any
$f\in Lip_d(\bM)$ with compact support
\[
||f||_{L^p(\bM)} \le ||f||^\theta_{L^{\nu}(\bM)}
||f||^{1-\theta}_{L^q(\bM)}\le \left(C_3
||\sqrt{\Gamma(f)}||_{L^1(\bM)}\right)^\theta
||f||^{1-\theta}_{L^q(\bM)}.
\]
For any $x\in \bM$ and $r>0$ we now let $f(y) = (r-d(y,x))^+$.
Clearly such $f\in Lip_d(\bM)$ and supp$\ f = \overline B(x,r)$.
Since with this choice $||\sqrt{\Gamma(f)}||_{L^1(\bM)}^\theta \le
\mu(B(x,r))^{\theta}$, the above inequality implies
\[
\frac{r}{2} \mu(B(x,\frac{r}{2})^{\frac{1}{p}} \le r^{1-\theta}
\left(C_3 \mu(B(x,r)\right)^\theta \mu(B(x,r))^{\frac{1-\theta}{q}},
\]
which, noting that $\frac{1-\theta}{q} + \theta = \frac{D+\theta
p}{pD}$, we can rewrite as follows
\[
\mu(B(x,r)) \ge \left(\frac{1}{2C_3^\theta}\right)^{pa}
\mu(B(x,\frac{r}{2}))^a r^{\theta p a},
\]
where we have let $a = \frac{D}{D+\theta p}$. Notice that $0<a<1$.
Iterating the latter inequality we find
\[
\mu(B(x,r)) \ge \left(\frac{1}{2C_3^\theta}\right)^{p\sum_{j=1}^k
a^j} r^{\theta p \sum_{j=1}^k a^j} 2^{-\theta p \sum_{j=1}^k
(j-1)a^j}\mu(B(x,\frac{r}{2^k}))^{a^k},\ \ \ k\in \mathbb N.
\]
>From Theorem \ref{T:doubling} for any $x\in \bM$ there exist
constants $C(x), R(x)>0$ such that with $Q(x) = \log_2 C(x)$ one has
\[
\mu(B(x,tr)) \ge C(x)^{-1} t^{Q(x)} \mu(B(x,r)),\ \ \ 0\le t\le 1,
0<r\le R(x).
\]
This estimate implies that
\[
\underset{k\to \infty}{\liminf}\ \mu(B(x,\frac{r}{2^k}))^{a^k}\ge
1,\ \ \ x\in \bM, r>0.
\]
Since on the other hand $\sum_{j=1}^\infty a^j = \frac{D}{\theta
p}$, and $\sum_{j=1}^\infty (j-1) a^j = \frac{D^2}{\theta^2p^2}$, we
conclude that
\[
\mu(B(x,r)) \ge \left(2^{-\frac{1}{\theta}(1+\frac{D}{p})}
C_3^{-1}\right)^D r^D,\ \ \ x\in \bM, r>0.
\]
This establishes (1), thus completing the proof.

\end{proof}

\section{A sub-Riemannian Bonnet-Myers theorem}\label{S:myer}

Let  $(\mathbb{M},g)$ be a complete, connected  Riemannian manifold
of dimension $d\ge 2$. It is well-known that if the Ricci tensor of
$\mathbb{M}$ satisfies the following bound for all $V\in T\bM$
\begin{equation}\label{ricbd}
\text{Ric}(V,V) \geq (d-1) \rho_1 |V|^2, \ \ \ \
\rho_1>0,
\end{equation} then $\bM$ is compact, with a finite
fundamental group, and diam$(\bM) \le \frac{\pi}{\sqrt{\rho_1}}$.
This is the celebrated Myer's theorem, which strengthens Bonnet's
theorem. Like the latter, Myer's theorem is usually proved by using
Jacobi vector fields (see e.g. Theorem 2.12 in \cite{Chavel}).

\

A different approach is based on the curvature-dimension inequality
\begin{equation}\label{Rcdrho2}
\Gamma_2(f,f) \ge \frac{1}{n} (L f)^2 + (n-1)\rho_1 \Gamma(f,f),
\end{equation}
which one obtains from \eqref{Rcdrho} assuming \eqref{ricbd}. By using ingenious non linear methods based on the study of the partial differential equation
\[
c(f^{p-1}-f)=-Lf, \\\\\\\\\\   1\le p \le
\frac{2d}{d-2},
\]
the
inequality \eqref{Rcdrho2} implies (for $d>2$) the following Sobolev
inequality
\begin{equation}\label{sobolev}
\frac{d}{(d-2)R^2} \left[ \left( \int_{\mathbb M} |f|^p d\mu
\right)^{2/p} -\int_{\mathbb M} f^2 d\mu \right] \le \int_{\mathbb
M} \Gamma (f,f) d\mu, \ \ f\in C^\infty_0(\mathbb M).
\end{equation}
where $\mu$ is the Riemannian measure and. Using a simple iteration procedure, it is deduced
from \eqref{sobolev} that the diameter of $\mathbb{M}$ is finite and
bounded by $\frac{\pi}{\sqrt{\rho_1}}$, see
\cite{ledoux-zurich}. This non linear method seems difficult to extend in our framework. However, a weak version of the Myers theorem could be proven by Bakry in \cite{bakry-stflour} by using linear methods only. This method based on entropy-energy inequalities (a strong form of log-sobolev inequalities) can be extended to our framework.

\

In this section we establish the following sub-Riemannian
Bonnet-Myer's compactness theorem. Here again, we assume that
$(\bM,d)$ is complete.

\begin{theorem}\label{T:BM}
If there exist constants $\rho_1,\rho_2 >0$ and a constant $\kappa >0$ such that for every smooth function $f :\mathbb{M} \rightarrow \mathbb{R}$:
\begin{equation}\label{Ricci_positive1}
\mathcal{R}(f,f) \ge \rho_1 \Gamma (f,f) +\rho_2 \Gamma^Z (f,f)
\end{equation}
\begin{equation}\label{Ricci_positive2}
\mathcal{T}(f,f) \le \kappa \Gamma (f,f),
\end{equation}
then the metric space $(\mathbb{M},d)$ is compact in the metric
topology with  a Hausdorff dimension less than $d\left( 1+\frac{3
\kappa}{2\rho_2} \right)$ and we have \[ \emph{diam}\ \bM \le
2\sqrt{3} \pi \sqrt{ \frac{\kappa+\rho_2}{\rho_1\rho_2} \left(
1+\frac{3\kappa}{2\rho_2}\right)d }.
\]
\end{theorem}

We shall proceed in several steps. Throughout this section, we assume that  (\ref{Ricci_positive1}) and (\ref{Ricci_positive2}) are satisfied.

\subsection{Global heat kernel bounds}\label{S:global}

Our first result is the following large-time exponential decay for
the heat kernel.

\begin{proposition}\label{bound_kernel}
Let $0< \nu < \frac{\rho_1\rho_2}{\rho_2+ \kappa}$. There exist $t_0
>0$ and $C_1>0$ such that for every $f\in C^\infty_0(\mathbb M)$, $f\ge 0$:
\[
\left| \frac{\partial }{\partial t} \ln P_t f  (x) \right| \le C_1
e^{-\nu t} , \quad \ \ \ \ x \in \mathbb{M}, t \ge t_0.
\]
\end{proposition}

\begin{proof}
In  Proposition \ref{variational}, we choose
\[
b(t)=(e^{-\alpha t} -e^{-\alpha T})^\beta,\ \ \ \ 0\le t\le T,
\]
with $\beta >2$ and $\alpha >0$. With such choice a simple
computation gives,
\[
\gamma (t) =\frac{d}{4}  \left( 2\rho_1 -\alpha \beta -\alpha \beta
\frac{\kappa}{\rho_2} -e^{-\alpha T} \left( \alpha(\beta -1) +
\frac{\alpha \beta\kappa}{\rho_2}\right) b(t)^{
-\frac{1}{\beta}}\right).
\]
Keeping in mind that $b(T) = b'(T) = 0$, and that $b(0) =
(1-e^{-\alpha T})^\beta$, $b'(0) = - \alpha \beta (1-e^{-\alpha
T})^{\beta-1}$, we obtain from \eqref{phis}
\begin{align}\label{lb}
& - \frac{\alpha \beta (1-e^{-\alpha T})^{\beta-1}}{2\rho_2}
\Gamma(\ln P_Tf) - (1-e^{-\alpha T})^\beta \Gamma^Z(\ln P_T f)
\\
& \ge - \frac{2}{\rho_2} \left(\int_0^T b'(t) \gamma (t) dt\right)
\frac{L P_T f}{P_T f} + \frac{1}{d\rho_2} \left(\int_0^T b'(t)
\gamma (t)^2 dt\right). \notag
\end{align} Now,
\begin{align*}
\int_0^T b'(t) \gamma (t) dt  = & - \frac{d}{4} \left( 2 \rho_1
-\alpha\beta - \alpha \beta \frac{\kappa}{\rho_2}\right) (1
-e^{-\alpha T})^\beta \\
& + \frac{d}{4}\frac{1}{1-\frac{1}{\beta}} \left( \alpha\beta
-\alpha+ \alpha \beta \frac{\kappa}{\rho_2} \right) e^{-\alpha T}(1
-e^{-\alpha T})^{\beta-1},
\end{align*}
\begin{align*}
\int_0^T b'(t) \gamma (t)^2 dt= & -\frac{d^2}{16} \left( 2 \rho_1 -\alpha\beta - \alpha \beta \frac{\kappa}{\rho_2}\right)^2 (1 -e^{-\alpha T})^\beta
 \\ &+\frac{d^2}{8} \frac{ \left( 2 \rho_1 -\alpha\beta - \alpha \beta \frac{\kappa}{\rho_2}\right) \left( \alpha\beta -\alpha+ \alpha \beta \frac{\kappa}{\rho_2} \right)}{1-\frac{1}{\beta}} e^{-\alpha T}(1 -e^{-\alpha T})^{\beta-1} \\
 &-\frac{d^2}{16}\frac{\left( \alpha\beta -\alpha+ \alpha \beta \frac{\kappa}{\rho_2} \right)^2}{ 1-\frac{2}{\beta}}e^{-2\alpha T}(1 -e^{-\alpha T})^{\beta-2}.
\end{align*}
If we choose
\[
\alpha= \frac{2 \rho_1\rho_2}{\beta(\rho_2+\kappa)},
\]
then \[ 2\rho_1 - \alpha \beta - \alpha \beta \frac{\kappa}{\rho_2}
= 0, \ \ \ \alpha \beta - \alpha + \alpha \beta
\frac{\kappa}{\rho_2} = 2\rho_1 - \alpha,\]
 and we obtain from \eqref{lb}:
\begin{align}\label{gamma_bound}
0 \le & \frac{\rho_1}{\rho_2+ \kappa} \Gamma (\ln P_T
f)+(1-e^{-\alpha T})\Gamma^Z (\ln P_T f) \le
\frac{d(2\rho_1-\alpha)}{2\rho_2 \left(1-\frac{1}{\beta} \right)}
e^{-\alpha T} \frac{LP_T f}{P_T f}
\\
& + \frac{d(2\rho_1-\alpha)^2}{16 \rho_2 \left(1-\frac{2}{\beta}
\right)} \frac{e^{-2\alpha T}}{ 1-e^{-\alpha T}}. \notag
\end{align}
Noting that $2\rho_1 - \alpha =
\frac{2\rho_1}{\beta(\rho_2+\kappa)}((\beta-1)\rho_2 +
\beta\kappa)>0$, and that $\beta
>2$ implies $\alpha < \frac{\rho_1 \rho_2}{\rho_2 + \kappa}$,
\eqref{gamma_bound} gives in particular the desired lower bound for
$\frac{\partial }{\partial t} \ln P_t f (x) $ with $\nu = \alpha$.

The upper bound is more delicate. We fix  $ 0<\eta = \frac{2
\rho_1\rho_2}{\beta(\rho_2+\kappa)}$, and with $\gamma =
2\beta\rho_1\rho_2$ we now choose in \eqref{lb}
\[
\alpha= \frac{2 \rho_1\rho_2- \gamma e^{-\eta T}
}{\beta(\rho_2+\kappa)} = \eta - \frac{\gamma e^{-\eta T}
}{\beta(\rho_2+\kappa)}.
\]
Clearly, $\alpha>0$ provided that $T$ be sufficiently large. This
choice gives
\[ 2\rho_1 - \alpha \beta - \alpha \beta \frac{\kappa}{\rho_2}
= \frac{\gamma e^{-\eta T}}{\rho_2}, \ \ \ \alpha \beta - \alpha +
\alpha \beta \frac{\kappa}{\rho_2} = 2\rho_1 - \alpha - \frac{\gamma
e^{-\eta T}}{\rho_2}.
\]
We thus have
\begin{align*}
\int_0^T b'(t) \gamma (t) dt & = - \frac{d}{4} e^{-\alpha T} (1 -
e^{-\alpha T})^{\beta-1} \left\{\frac{\gamma (1 - e^{-\alpha T})
e^{-(\eta-\alpha)T}}{\rho_2} - \frac{\beta}{\beta-1} (2 \rho_1 -
\alpha - \frac{\gamma e^{-\eta T}}{\rho_2})\right\}. \end{align*}
Noting that $e^{-(\eta-\alpha)T} = e^{-\frac{\gamma T e^{-\eta
T}}{\beta(\rho_2+\kappa)}}\to 1$, and $\alpha \longrightarrow
 \frac{2\rho_1\rho_2}{\beta(\rho_2+\kappa)}$ as $T\to \infty$, we
obtain
\[
\frac{\gamma (1 - e^{-\alpha T}) e^{-(\eta-\alpha)T}}{\rho_2} -
\frac{\beta}{\beta-1} (2 \rho_1 - \alpha - \frac{\gamma e^{-\eta
T}}{\rho_2})\ \longrightarrow\ \frac{\gamma}{\rho_2} -
\frac{\beta}{\beta-1}\left(2\rho_1  -
\frac{2\rho_1\rho_2}{\beta(\rho_2+\kappa)}\right).\] Since by our
choice of $\gamma$ we have $\frac{\gamma}{\rho_2} -
\frac{\beta}{\beta-1}\left(2\rho_1  -
\frac{2\rho_1\rho_2}{\beta(\rho_2+\kappa)}\right)>0$, it is clear
that we have \[ \int_0^T b'(t) \gamma (t) dt \le -
\frac{d}{8}\left(\frac{\gamma}{\rho_2} -
\frac{\beta}{\beta-1}\left(2\rho_1
-\frac{2\rho_1\rho_2}{\beta(\rho_2+\kappa)}\right)\right) e^{-\alpha
T} (1 - e^{-\alpha T})^{\beta-1} , \]
 provided that $T$ be large
enough. We also have
\begin{align*}
\int_0^T b'(t) \gamma (t)^2 dt= & -\frac{d^2}{16} e^{-2\alpha T} (1
-e^{-\alpha T})^{\beta-2} \bigg\{\frac{\beta}{\beta-2} (2 \rho_1 -
\alpha - \frac{\gamma e^{-\eta T}}{\rho_2})^2
\\
&  + \frac{\gamma^2}{\rho_2^2} (1 -e^{-\alpha T})^2
e^{-2(\eta-\alpha)T} - 2 \frac{\gamma}{\rho_2}\frac{\beta}{\beta -1}
(1 -e^{-\alpha T}) (2 \rho_1 - \alpha - \frac{\gamma e^{-\eta
T}}{\rho_2}) e^{-(\eta-\alpha)T}\bigg\}.
\end{align*}
Using our choice of $\gamma$ we see that, if we let $T\to \infty$,
the quantity between curly bracket in the right-hand side converges
to
\[
\frac{\beta}{\beta-2} 4\rho_1^2 \left(\frac{(\beta-1)\rho_2+\beta
\kappa}{\beta(\rho_2+\kappa)}\right)^2 + 4 \beta^2\rho_1^2 -
\frac{8\beta^2\rho_1^2}{\beta-1}\frac{(\beta-1)\rho_2+\beta
\kappa}{\beta(\rho_2+\kappa)}.
\]
This quantity is strictly positive provided that
\[
\frac{2\beta}{\beta-1}\frac{(\beta-1)\rho_2+\beta
\kappa}{\beta(\rho_2+\kappa)} < \frac{1}{\beta-2}
\left(\frac{(\beta-1)\rho_2+\beta
\kappa}{\beta(\rho_2+\kappa)}\right)^2 +  \beta,
\]
and this latter inequality is true, as one recognizes by applying
the inequality $2xy\le x^2 + y^2$. From these considerations and
from \eqref{lb} we conclude the desired upper bound for
$\frac{\partial }{\partial t} \ln P_t f (x) $.

\end{proof}

\begin{proposition}\label{harnack_spectral}
Let $0< \nu < \frac{\rho_1\rho_2}{\kappa +\rho_2}$. There exist $t_0
>0$ and $C_2>0$ such that for every $f\in C^\infty_0(\mathbb
M)$, with $f\ge 0$,
\[
e^{-C_2 e^{-\nu t} d(x,y)} \le \frac{P_t f  (x)}{P_t f (y)} \le
e^{C_2 e^{-\nu t} d(x,y)} , \quad \ \ \ x,y \in \mathbb{M},\ t \ge
t_0.
\]
\end{proposition}

\begin{proof}
If we combine \eqref{gamma_bound} with the upper bound of
Proposition \ref{bound_kernel}, we obtain that for $x \in
\mathbb{M}$ and $t \ge t_0$,
\[
\Gamma (\ln P_t f)(x) \le C_2 e^{-\nu t} .
\]
We infer that the function $u(x) = C_2^{-1} e^{\nu t} \ln P_t f(x)$,
which belongs to $C^\infty(\bM)$, is such that $\Gamma(u)(x) \le 1$,
$x\in \bM$. From \eqref{d} we obtain that \[ |u(x) - u(y)| \le
d(x,y), \ \ \ \ \ x, y\in \bM. \] This implies the sought for
conclusion.

\end{proof}

If we now fix $x \in \mathbb{M}$, and denote by $p(x,\cdot,t)$ the
heat kernel with singularity at $(x,0)$, then according to
Proposition \ref{bound_kernel} we obtain for $t \ge t_0$,
\begin{align}\label{estimeenoyau}
\left| \frac{\partial \ln p(x,y,t)}{\partial t}\right|   \le  C_1
\exp\left(-\nu t \right),
\end{align}
with $0< \nu < \frac{\rho_1\rho_2}{\kappa +\rho_2}$. This shows that
$\ln p(\cdot,\cdot,t) $ converges when $t\to\infty$. Let us call
$\ln p_\infty $ this limit. Moreover, from Proposition
\ref{harnack_spectral}  the limit, $\ln p_{\infty}( x, \cdot)$ is a
constant $C(x)$. By the symmetry property $p(x,y,t)=p(y,x,t)$, so
that $C(x)$ actually does not depend on $x$ . We deduce from this
that the invariant measure $\mu$ is finite. We may then as well
suppose that $\mu$ is a probability measure, in which case
$p_{\infty}=1$. We assume this from now on.

We now can prove a global and explicit upper bound for the heat kernel $p(x,y,t)$.

\begin{proposition}\label{globalbound}
For $x,y \in \bM$ and $t>0$,
\[
p(x,y,t) \le \frac{1}{\left( 1-e^{-\frac{2\rho_1 \rho_2
t}{3(\rho_2+\kappa)}}
\right)^{\frac{d}{2}\left(1+\frac{3\kappa}{2\rho_2}\right)} }.
\]
\end{proposition}

\begin{proof}
We apply \eqref{gamma_bound} with $\beta=3$ and obtain
\begin{align}\label{bla}
\frac{\rho_1}{\rho_2+\kappa} \Gamma (\ln P_t f)+(1-e^{-\alpha
t})\Gamma^Z (\ln P_t f) & \le   \frac{\rho_1}{2\rho_2}
\frac{2\rho_2+3\kappa}{\rho_2+\kappa} e^{-\alpha t} \frac{LP_t
f}{P_t f}
\\
& + \frac{d \rho_1^2}{12 \rho_2} \left(\frac{2\rho_2+
3\kappa}{\rho_2+\kappa} \right)^2 \frac{e^{-2\alpha t}}{
1-e^{-\alpha t}}, \notag
\end{align}
where $\alpha=\frac{2\rho_1\rho_2}{3(\rho_2+\kappa)}$. We deduce
\[
\frac{\partial \ln P_tf}{\partial t} \ge -\frac{d\rho_1}{6}
\frac{2\rho_2+3\kappa}{\rho_2+\kappa}\frac{e^{-\alpha t}}{
1-e^{-\alpha t}}.
\]
By integrating from $t$ to $\infty$, we obtain
\[
-\ln p(x,y,t) \ge -\frac{d}{2} \left( 1+\frac{3\kappa}{2\rho_2}\right) \ln (1-e^{-\alpha t}).
\]
This gives the desired conclusion.

\end{proof}

\subsection{Diameter bound}

In this subsection we conclude the proof of Theorem \ref{T:BM} by
showing that the diam$\ \bM$ is bounded. Since we have assumed that
$(\bM,d)$ be complete, this implies that such metric space is
compact. The idea is to show that the operator $L$ satisfies an
entropy-energy inequality. Such inequalities have been extensively
studied by Bakry in \cite{bakry-stflour} (see chapters 4 and 5).

To simplify the computations, in what follows we denote by $D$ the
number defined in \eqref{D}, and we set
\[
\alpha=\frac{2\rho_1 \rho_2 }{3(\rho_2+\kappa)}.
\]

\begin{proposition}\label{P:entropyenergy}
For $f \in L^2 (\bM)$ such that $\int_\bM f^2 d\mu =1$, we have
\[
\int_\bM f^2 \ln f^2 d\mu \le \Phi \left( \int_\bM \Gamma(f) d\mu \right),
\]
where
\[
\Phi(x)=  D \left[  \left( 1+\frac{2}{\alpha D } x\right)\ln \left(
1+\frac{2}{\alpha D} x\right)-\frac{2}{\alpha D } x  \ln \left(
\frac{2}{\alpha D} x  \right)\right].
\]
\end{proposition}

\begin{proof}
>From Proposition \ref{globalbound}, for every $f \in L^2(\bM)$ we
have
\[
\|P_t f \|_\infty \le \frac{1}{\left( 1-e^{-\alpha t}
\right)^{\frac{D}{2} }} \| f \|_2.
\]
Therefore, from Davies theorem (Theorem 2.2.3 in \cite{Davies}), for $f \in L^2
(\bM)$ such that $\int_\bM f^2 d\mu =1$, we obtain
\[
\int_\bM f^2 \ln f^2 d\mu \le 2t \int_\bM \Gamma(f) d\mu -D \ln \left( 1-e^{-\alpha t} \right), \quad t >0.
\]
By minimizing over $t$ the right-hand side of the above inequality,
we obtain
\[
\int_\bM f^2 \ln f^2 d\mu \le -\frac{2}{\alpha} x
\ln\left(\frac{2x}{2x + \alpha D}\right) +D
\ln\left(\frac{2x+\alpha D}{\alpha D}\right).
\]
where $ x = \int_\bM \Gamma(f) d\mu$. It is now an easy exercise to
recognize that the right-hand side of the latter inequality is the
same as $\Phi(x)$.

\end{proof}

With Proposition \ref{P:entropyenergy} in hands, we can finally
complete the proof of Theorem \ref{T:BM}.

\begin{proposition}
One has
\[ \emph{diam}\ \bM \le 2 \sqrt 2 \sqrt{\frac{D}{\alpha}}\pi = 2\sqrt{3} \pi \sqrt{
\frac{\rho_2+ \kappa}{\rho_1\rho_2} \left(
1+\frac{3\kappa}{2\rho_2}\right)d }.
\]
\end{proposition}

\begin{proof}
The function $\Phi$ that appears in the Proposition
\ref{P:entropyenergy} enjoys the following properties:
\begin{itemize}
\item $\Phi'(x)/x^{1/2}$  and $\Phi(x)/x^{3/2}$ are integrable on $(0,\infty)$;
\item $\Phi$ is concave;
\item $\frac{1}{2}\int_0^{+\infty} \frac{\Phi(x)}{x^{3/2}}dx=\int_0^{+\infty} \frac{\Phi'(x)}{\sqrt{x}}dx =-2\int_0^{+\infty} \sqrt{x} \Phi''(x)dx <+\infty.$
\end{itemize}
We can therefore apply the beautiful Theorem 5.4 in \cite{bakry-stflour} to deduce that the diameter of $\bM$ is finite and
\[
\emph{diam}\ \bM \le-2\int_0^{+\infty} \sqrt{x} \Phi''(x)dx.
\]
Since $\Phi''(x) = - \frac{2D}{x(2x+\alpha D)}$, a routine
calculation shows
\[
-2\int_0^{+\infty} \sqrt{x} \Phi''(x)dx= 2\sqrt{3} \pi \sqrt{
\frac{\rho_2+\kappa}{\rho_1\rho_2} \left(
1+\frac{3\kappa}{2\rho_2}\right)d }.
\]
\end{proof}

\begin{remark}
The constant $2\sqrt{3} \pi \sqrt{
\frac{\rho_2+\kappa}{\rho_1\rho_2} \left(
1+\frac{3\kappa}{2\rho_2}\right)d } $ is not sharp. For instance, in
the Riemannian case $\kappa=\rho_2=0$, we obtain
\[
\emph{diam}\ \bM \le  2\sqrt{3}\pi\sqrt{ \frac{d}{\rho_1}},
\]
whereas it is known from the Bonnet-Myer's theorem that
\[
\emph{diam}\ \bM  \le  \pi \sqrt{ \frac{d-1}{\rho_1}}.
\]
\end{remark}

\subsection{Dimension  bound}

We now  turn to an upper bound for the Hausdorff dimension of the compact metric space $(\bM,d)$.

\begin{proposition}
The Hausdorff dimension of the  metric space $(\bM,d)$ is less than
$D$ given by \eqref{D}.
\end{proposition}

\begin{proof}
Let us recall that, from our assumptions,
\[
T_x \mathbb{M} = \mathcal{H}(x) \oplus \mathcal{V}(x),\ \ \ \ x\in
\bM,
\]
where
\[
\mathcal{H}(x)=\text{span} \left\{X_1 (x),...,X_d (x) \right\},
\quad x\in \mathbb{M},
\]
and
\[
\mathcal{V}(x)=\text{span}\left\{  Z_{mn} (x), 1\le m,n \le \di.
\right\},
\]
We moreover assumed that
\[
\text{dim} \mathcal{H}(x) =d,\ \ \ \ x\in \bM,
\]
and this implies that also $\text{dim} \mathcal V(x) = \text{dim}
\bM - d$, is independent of $x\in \bM$. From Theorem 2 in
\cite{mitchell} we deduce that the Hausdorff dimension of the
compact metric space $(\bM,d)$ is equal to
$\dim_{\mathbf{Haus}}(\bM)=d+2\text{dim} \mathcal V(x)$. Moreover,
from \cite{benarous} and \cite{takanobu} (see also Chapter 3 in
\cite{baudoin}), there exists a smooth and positive function $m$ on
$\bM$ such that
\[
\lim_{t \to 0} t^{\frac{D}{2}} p(x,x,t) =m(x).
\]
From the bound (\ref{ultracontractivity}) we conclude that
\[
\dim_{\mathbf{Haus}} (\bM) \le D.
\]
\end{proof}

\subsection{Isoperimetric bounds and $L^1$ Poincar\'e inequality}

We recall our assumption, following Proposition
\ref{harnack_spectral}, that $\mu(\bM)=1$. Also, let $D$ be defined
by \eqref{D}. With this in hands, we can now proceed as in section
\ref{S:Isoperimetry}.

\begin{proposition}\label{P:isocompact}
Let $E\subset \bM$ be a Caccioppoli  set. We have
\[
\mu(E)(1-\mu(E))  \le \frac{3}{2} D \sqrt{\frac{\kappa+\rho_2}{d\rho_1 \rho_2}} P_{\mathcal{H}}(E,\bM).
\]
\end{proposition}

\begin{proof}
We proceed exactly as in the proof of Proposition \ref{P:char} to
obtain from \eqref{bla} the inequalities
\[
\|\sqrt{\Gamma(P_t f)}\|_\infty \le D \sqrt{\frac{\rho_1 \rho_2}{d(\kappa+\rho_2)} } \frac{e^{-\alpha t}}{\sqrt{1-e^{-\alpha t}}} \|f \|_\infty
\]
and
\[
\| f -P_tf \|_1 \le 3 D \sqrt{\frac{\kappa+\rho_2}{d\rho_1 \rho_2}} \sqrt{1-e^{-\alpha t}} \|f \|_1.
\]
Combining this with Proposition \ref{globalbound} gives
\[
3 D \sqrt{\frac{\kappa+\rho_2}{d\rho_1 \rho_2}}  \sqrt{1-e^{-\alpha t}}P_{\mathcal{H}}(E,\bM) \ge 2 \left( \mu(E) -\frac{1}{(1-e^{-\alpha t})^{D/2}} \mu(E)^2 \right).
\]
We conclude by letting $t \to +\infty$.

\end{proof}

The previous isoperimetric inequality leads to the following $L^1$
Poincar\'e inequality.

\begin{proposition}
Let $f\in C^\infty(\bM)$, then
\[
\inf_{c \in \mathbb{R}} \int_\bM | f-c | d \mu \le  6 D
\sqrt{\frac{\kappa+\rho_2}{d\rho_1 \rho_2}} \int_\bM
\sqrt{\Gamma(f)} d\mu.
\]
\end{proposition}

\begin{proof}
Let $m$ be a median for $f$, that is
\[
\mu ( f \ge m) \ge \frac{1}{2}, \quad \mu( f \le m) \ge \frac{1}{2}.
\]
Set
\[
f^+=\max(f-m,0),\quad f^-=-\min(f-m,0)
\]
so that $f-m=f^+-f^-$. We have
\[
\int_\bM | f-m| d\mu = \int_\bM f^+ d\mu +\int_\bM f^- d\mu,
\]
and thus
\[
\int_\bM | f-m| d\mu =\int_0^{+\infty} \mu(f^+>t)dt
+\int_0^{+\infty} \mu(f^->t)dt.
\]
Observe that for every $t>0$,
\[
\mu(f^+ \ge t) \le \frac{1}{2}, \quad \mu(f^- \ge t) \le
\frac{1}{2},
\]
from Proposition \ref{P:isocompact} we obtain
\[
\mu(f^+>t) \le 3D \sqrt{\frac{\kappa+\rho_2}{d\rho_1 \rho_2}}
P_{\mathcal{H}}(\{f^+>t\},\bM),
\]
and
\[
\mu(f^->t) \le 3D \sqrt{\frac{\kappa+\rho_2}{d\rho_1 \rho_2}}
P_{\mathcal{H}}(\{f^->t\},\bM).
\]
This gives
\[
\int_\bM | f-m| d\mu \le 6D \sqrt{\frac{\kappa+\rho_2}{d\rho_1
\rho_2}}\left( \int_\bM \sqrt{\Gamma(f^+)} d\mu+\int_\bM
\sqrt{\Gamma(f^-)} d\mu\right).
\]
Observing that
$\sqrt{\Gamma(f^+)}+\sqrt{\Gamma(f^-)}=\sqrt{\Gamma(f^++f^-)}$,
completes the proof.

\end{proof}

\subsection{A Lichnerowicz type theorem}\label{lichnerowicz}

A well-known theorem of Lichnerowicz asserts that on a
$d$-dimensional complete Riemannian manifold whose Ricci curvature
is bounded below by a non negative constant $\rho$, then the first
eigenvalue of the Laplace-Beltrami operator is bounded below by
$\frac{\rho d}{d-1}$. In this section, we provide a similar theorem
for our operator $L$. Let us observe that in \cite{Greenleaf},
Greenleaf obtained a similar result for the sub-Laplacian on a CR
manifold.

\begin{proposition}
The first non zero eigenvalue $\lambda_1$ of $-L$ satisfies the estimate
\[
\lambda_1 \ge \frac{\rho_1 \rho_2}{\frac{d-1}{d} \rho_2 +\kappa}.
\]
\end{proposition}

\begin{proof}
Let $f:\mathbb{M} \rightarrow \mathbb{R}$ be an eigenfunction corresponding to the eigenvalue $-\lambda_1$.
>From our assumptions,
\begin{equation*}
\Gamma_{2}(f,f)+\nu \Gamma^Z_{2}(f,f) \ge \frac{1}{d} (Lf)^2 + \left( \rho_1 -\frac{\kappa}{\nu}\right)  \Gamma (f,f) + \rho_2 \Gamma^Z (f,f).
\end{equation*}
By integrating this inequality on the manifold $\mathbb{M}$, we obtain
\begin{equation*}
\int_{\mathbb{M}} \Gamma_{2}(f,f) d\mu+\nu\int_{\mathbb{M}} \Gamma^Z_{2}(f,f)d\mu \ge \frac{1}{d} \int_{\mathbb{M}}(Lf)^2 d\mu+ \left( \rho_1 -\frac{\kappa}{\nu}\right) \int_{\mathbb{M}} \Gamma (f,f) d\mu+ \rho_2\int_{\mathbb{M}} \Gamma^Z (f,f)d\mu.
\end{equation*}
Let us now recall that
\begin{equation*}
\Gamma(f,f) =\frac{1}{2}(L(f^2)-2fLf)
\end{equation*}
\begin{equation*}
\Gamma_{2}(f,f) = \frac{1}{2}\big[L\Gamma(f,f) -2 \Gamma(f,
Lf)\big],
\end{equation*}
and
\begin{equation*}
\Gamma^Z_{2}(f,f) = \frac{1}{2}\big[L\Gamma^Z (f,f) - 2\Gamma^Z(f,
Lf)\big].
\end{equation*}
Therefore, by using $Lf=-\lambda_1f$ and integrating by parts in the
above inequality, we find
\begin{equation*}
\left( \lambda_1^2 -\frac{\lambda^2_1}{d}+\frac{\kappa \lambda_1}{\nu}  -\rho_1 \lambda_1 \right)
\int_{\mathbb{M}} f^2 d\mu \ge (\rho_2 -\nu \lambda_1) \int_{\mathbb{M}} \Gamma^Z (f,f)d\mu.
\end{equation*}
By choosing $\nu=\frac{\rho_2}{\lambda_1}$, we obtain the desired
inequality
\[
\lambda_1 \ge \frac{\rho_1 \rho_2}{\frac{d-1}{d} \rho_2 +\kappa}.
\]
\end{proof}

\begin{remark}
We note that when $\kappa = 0$, we recover the classical theorem of
Lichnerowicz.
\end{remark}

\end{document}